\documentclass[11pt,english]{article}

\usepackage[top=2.7cm,bottom=3.2cm,inner=2.7cm,outer=2.7cm]{geometry}

\usepackage{amsthm}
\usepackage{amsmath}
\usepackage{amssymb}
\usepackage{setspace}
\usepackage[numbers]{natbib}
\usepackage{mathtools}
\usepackage{graphics}
\usepackage{graphicx}
\usepackage{epstopdf}
\usepackage[hidelinks]{hyperref}
\usepackage{enumerate}
\usepackage{csquotes}
\usepackage{verbatim}
\usepackage{tikz}
\usetikzlibrary{calc}
\usetikzlibrary{decorations.pathreplacing}
\usetikzlibrary{positioning,patterns}
\usepackage{caption}
\usepackage{subcaption}
\usepackage{mathtools}
\usepackage{float}

\usepackage{framed}

\MakeOuterQuote{"}

\makeatletter
\def\imod#1{\allowbreak\mkern10mu({\operator@font mod}\,\,#1)}

\makeatother

\theoremstyle{plain}
\newtheorem*{thm*}{Theorem}
\newtheorem{thm}{Theorem}[section]
\newtheorem{lem}[thm]{Lemma}
\newtheorem*{lem*}{Lemma}
\newtheorem{claim}[thm]{Claim}
\newtheorem{prop}[thm]{Proposition}
\newtheorem*{claim*}{Claim}
\newtheorem{cor}[thm]{Corollary}

\newtheorem{obs}[thm]{Observation}

\theoremstyle{definition}

\newtheorem{define}[thm]{Definition}

\theoremstyle{remark}
\newtheorem*{rem}{Remark}

\usepackage{babel}

\def\eps{\varepsilon}

\def\C{\mathcal}

\DeclareMathOperator\con{con}

\newcommand{\vmod}[1]{\,(\textnormal{mod}\,{#1})}
\newcommand{\remove}[1]{}
\newcommand{\enter}{\vspace{\parskip}\\}

\begin{document}

\title{Monochromatic cycle partitions of $2$-coloured graphs with minimum degree $3n/4$}

\author{Shoham Letzter
\thanks{Department of Pure Mathematics and Mathematical Statistics, University of Cambridge, Wilberforce Road, Cambridge CB3\thinspace0WB, UK. e-mail: s.letzter@dpmms.cam.ac.uk.
}}
\maketitle

\begin{abstract}
\setlength{\parskip}{\medskipamount}
\setlength{\parindent}{0pt}
\noindent
Balogh, Bar\'at, Gerbner, Gy\'arf\'as, and S\'ark\"ozy proposed the following conjecture. Let $G$ be a graph on $n$ vertices with minimum degree at least $3n/4$. Then for every $2$-edge-colouring of $G$, the vertex set $V(G)$ may be partitioned into two vertex-disjoint cycles, one of each colour.

We prove that this conjecture holds for $n$ large enough, improving approximate results by the aforementioned authors and by DeBiasio and Nelsen.
\end{abstract}

\section{Introduction}

\subsection{History}

While undergraduates in Budapest, Gerencs\'er and Gy\'arf\'as \cite{gyarfas-gerencser} proved the following simple result: for any $2$-edge-colouring of the complete graph $K_n$, there exists a monochromatic path of length at least $\lceil 2n/3 \rceil$. It is easy to see that this statement is sharp.
In their paper, Gerencs\'er and Gy\'arf\'as observe that a weaker result, asserting the existence of a monochromatic path of length at least $n/2$, can be deduced from the following simple observation: for any red and blue colouring of $K_n$, there is a Hamilton path which is the union of a red path and a blue path.
The latter observation, simple as it is, inspired intensive research.

In a later paper, Gy\'arf\'as \cite{gyarfas} proved that, in fact, more is true. He showed that for any red and blue colouring of $K_n$ the vertices may be covered by a red cycle and a blue one sharing at most one vertex.
Lehel went even further: he conjectured that for every $2$-colouring of $K_n$ the vertex set may be partitioned into two monochromatic cycles of distinct colours. This conjecture first appeared in \cite{ayel}, where it was proved for some special colourings of $K_n$.

Almost twenty years after this conjecture was made, {\L}uczak, R\"odl and Szemer\'edi  \cite{luczak_rodl_szemeredi} proved it for large $n$, using the Regularity Lemma. Ten years later, Allen \cite{allen} proved it for large $n$ avoiding the use of the Regularity Lemma.
Finally, Lehel's conjecture was fully resolved by Bessy and Thomass\'e \cite{bessy_thomasse} with an elegant and short proof. 

\subsection{Conjectures and progress}
In the hope of generalising the above result of Gerencs\'er and Gy\'arf\'as, Schelp \cite{schelp} considered $2$-colourings of graphs which are not necessarily complete. In particular, he conjectured that given a graph $G$ on $n$ vertices with $\delta(G) > 3n/4$, there is a monochromatic path of length at least $2n/3$. Benevides, {\L}uczak, Skokan, Scott and White \cite{benevides_et_al} and  Gy\'arf\'as and S\'ark\"ozy \cite{gyarfas_sarkozy}  proved approximate versions of this conjecture.

Inspired by Schelp's conjecture, Balogh, Bar\'at, Gerbner, Gy\'arf\'as, and S\'ark\"ozy \cite{balogh_et_al} proposed the following conjecture: given a graph $G$ on $n$ vertices with minimum degree $\delta(G) > 3n/4$, for every $2$-colouring of $G$ the vertex set can be partitioned into two monochromatic cycles of distinct colours.
We remark that for the purpose of this conjecture, the empty set, a single vertex and an edge are considered to be cycles.
We note that there are examples of $2$-coloured graphs $G$ on $n$ vertices with $\delta(G) = \lceil 3n/4 - 1 \rceil$ which do not admit such a partition (see Section \ref{sec_extremal_examples}).

In \cite{balogh_et_al}, the authors prove the following approximate result of their conjecture. For every $\eps >0$ there exists $n_0$ such that for every $2$-coloured graph $G$ on $n\ge n_0$ vertices with minimum degree $\delta(G)\ge (3/4 + \eps)n$, there exist vertex-disjoint monochromatic cycles of distinct colours covering all but at most $\eps n$ of the vertices.

Recently, DeBiasio and Nelsen \cite{debiasio} proved the following stronger approximate result of the latter conjecture: for every $\eps>0$ there exists $n_0$ such that, for every $2$-coloured graph $G$ on $n\ge n_0$ vertices and $\delta(G)\ge (3/4 + \eps)n$, the vertex set may be partitioned into two monochromatic cycles of distinct colours.

\subsection{The main result}
Our main aim is to prove that the conjecture of Balogh et al.~\cite{balogh_et_al} holds if $n$ is large enough.
\begin{thm}\label{thm_main}
There exists $n_0$ such that if a graph on $n \ge n_0$ vertices and minimum degree at least $3n/4$ is $2$-coloured then its vertex set may be partitioned into two monochromatic cycles of different colours.
\end{thm}

In \cite{luczak}, \L{}uczak introduced a technique that uses the Regularity Lemma to reduce problems about paths and cycles into problems about connected matchings, which are matchings that are contained in a connected component.
This technique, which we shall describe in more detail in Section \ref{sec_regularity_lem}, has become fairly standard by now and can be used to prove the approximate result of Balogh et al.~\cite{balogh_et_al}.
The second result by DeBiasio and Nelsen \cite{debiasio}, requires further ideas, most notably the ``absorbing technique'' of R\"odl, Ruci\'nski and Szemer\'edi (see \cite{rodl_et_al_absorbing} and \cite{levitt_et_al_absorbing}).
Nevertheless, the stronger conditions on the minimum degree make their proof a great deal easier than ours. In order to prove Theorem \ref{thm_main}, we use a variety of additional ideas and techniques.

We remark that Theorem \ref{thm_main} is sharp. Indeed, for every $n \ge 4$, there exists a $2$-coloured graph on $n$ vertices with minimum degree $\lceil\frac{3n}{4} - 1 \rceil$ admitting no partition into two monochromatic cycles of distinct colours.
We give such extremal examples in the following section.
These examples disprove the conjecture from \cite{debiasio}, that a slightly stronger version of Theorem \ref{thm_main} may hold, namely that the conclusion holds for graphs with minimum degree at least $\frac{3n - 3}{4}$.

The following section consists of some extremal examples for Theorem \ref{thm_main}. In Section \ref{sec_overview} we give an overview of the proof as well as the structure of the rest of this paper.

\section{Sharpness examples} \label{sec_extremal_examples}
Before we turn to the proof of Theorem \ref{thm_main}, we give some extremal examples showing that the theorem is sharp. More precisely, we give examples of $2$-coloured graphs on $n$ vertices with minimum degree $\lceil \frac{3n}{4} - 1 \rceil$ admitting no partition into two monochromatic cycles of distinct colours.
Figures (\ref{figure_sharpness_1}, \ref{figure_sharpness_2}, \ref{figure_sharpness_3}, \ref{figure_sharpness_4}) depict several families of such examples differing in the values of $n \vmod 4$.

In these figures, we use black and grey for the colours of the edges.
The areas coloured in either black or grey denote a complete (or complete bipartite) subgraph of the corresponding colour.
Shaded areas denote complete (or complete bipartite) subgraphs which may be coloured arbitrarily.
White areas are empty subgraphs.
A small dot denotes a single vertex and a larger shape denote a cluster of vertices whose size is written in it.
It is not hard to see that, indeed, each of these figures is an extremal example for Theorem \ref{thm_main}. We leave the details to the reader.

\begin{figure}[H]\centering 
\captionsetup{width=0.8\textwidth}

\includegraphics[scale=.7]{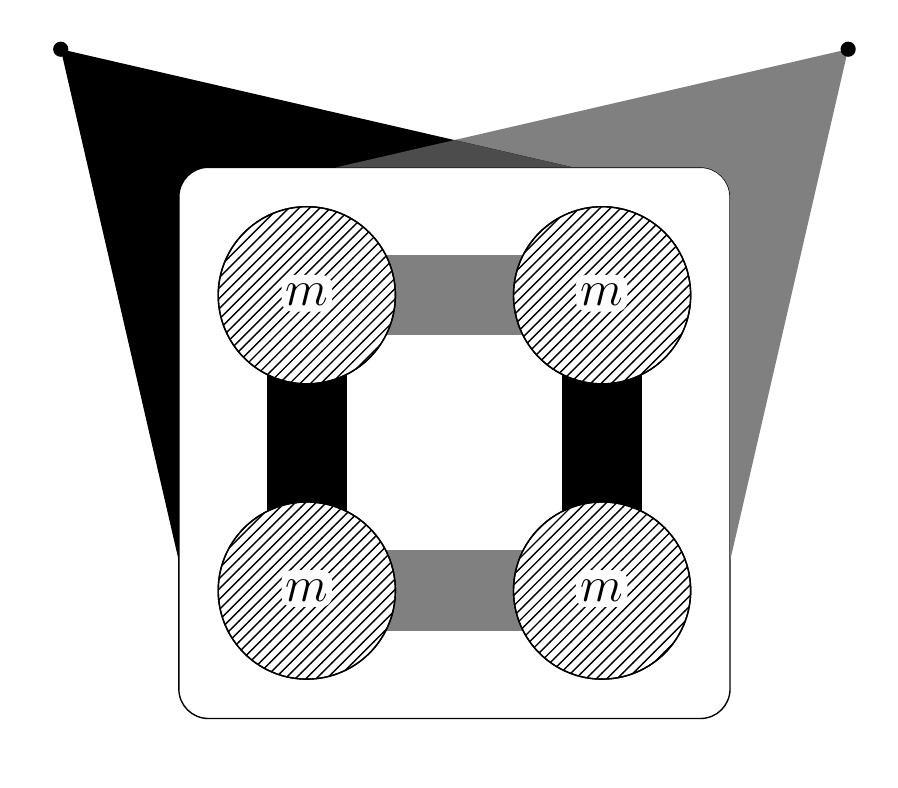} 
\hspace{.5in}
\includegraphics[scale=.7]{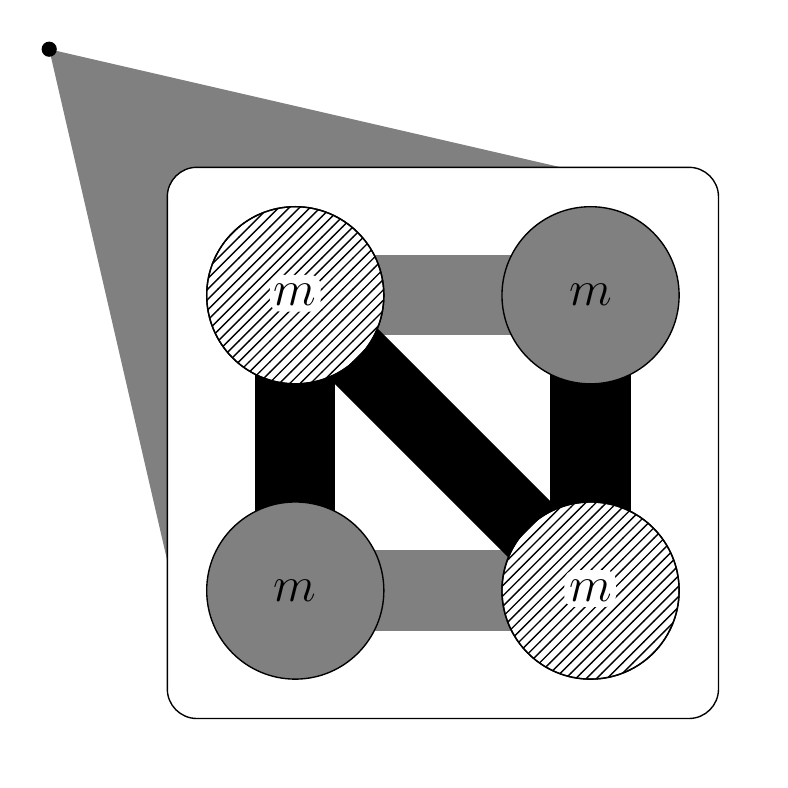}

\caption{Black and grey graphs of orders $4m + 2$ and $4m + 1$ (from left to right) with minimum degrees $3m + 1$ and $3m$ respectively, admitting no partition into a black cycle and a grey one. 
\enter
Other extremal examples may be formed by removing a vertex or two from the left-hand graph, or by removing the extra vertex from the right-hand graph.
}
\label{figure_sharpness_1}
\end{figure}

\begin{figure}[H]\centering 

\includegraphics[scale=.7]{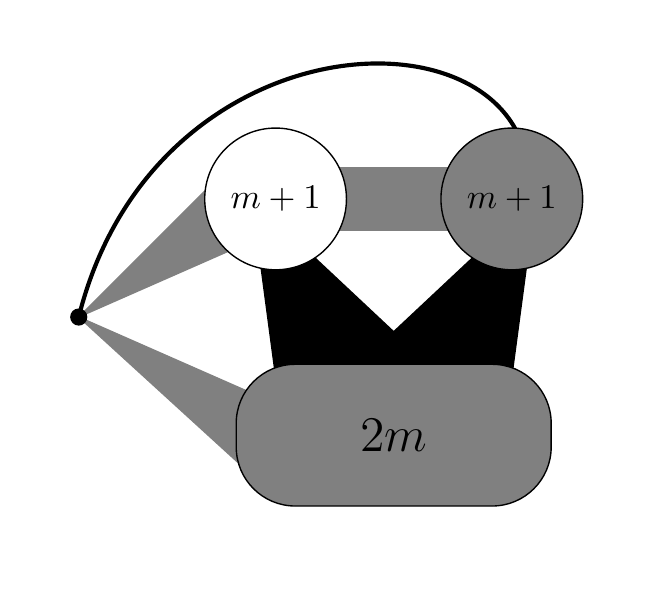}
\includegraphics[scale=.7]{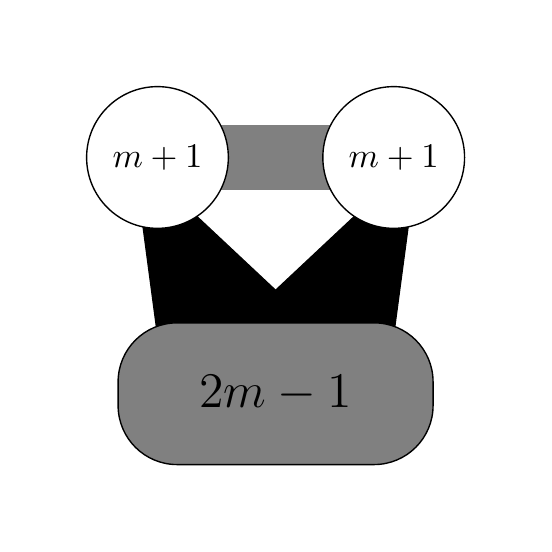}
\includegraphics[scale=.7]{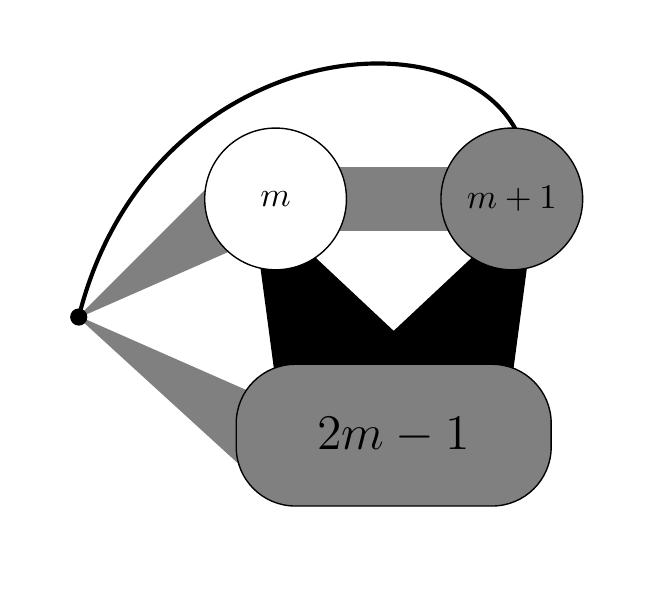}

\caption{Black and grey graphs of orders $4m$, $4m + 1$ and $4m + 2$ (left to right) with minimum degrees $3m - 1$, $3m$ and $3m + 1$ respectively and no partition into two monochromatic cycles of distinct colours.
}
\label{figure_sharpness_2}
\end{figure}

\begin{figure}[h]\centering 

\includegraphics[scale=.8]{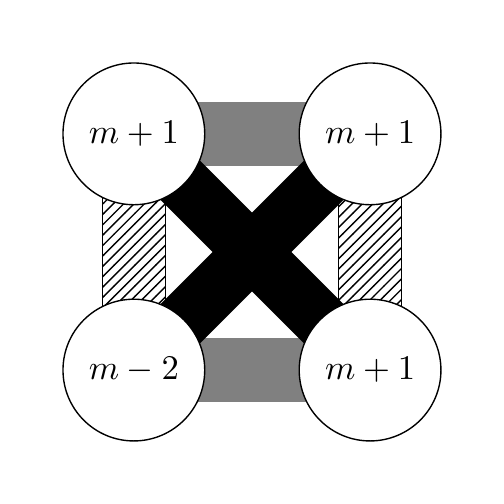}

\caption{A black and grey graph on $4m + 1$ vertices and minimum degree $3m$ with no partition into a grey cycle and a black one.
}
\label{figure_sharpness_3}
\end{figure}

\begin{figure}[H]\centering 
\includegraphics[scale=.8]{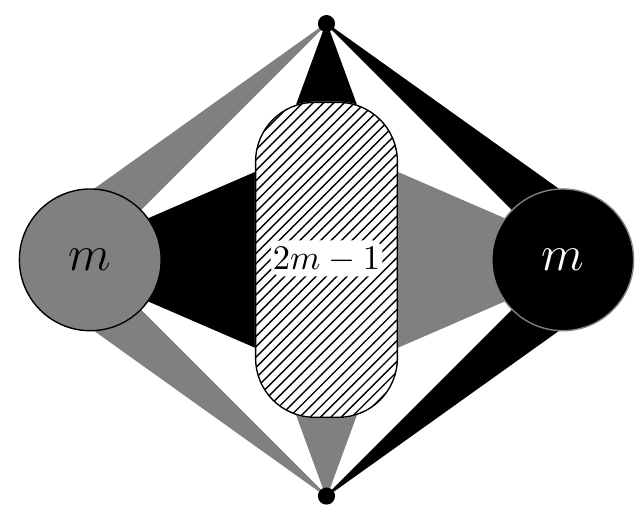}

\caption{A black and grey graph on $4m + 1$ vertices and minimum degree $3m$ with no partition into a black cycle and a grey one.
Any subgraph of this graph, obtained by removing one vertex, is also an extremal example.
}
\label{figure_sharpness_4}
\end{figure}

We remark that there may exist other extremal examples. It may be possible to find all of them by carefully analysing the proof of Theorem \ref{thm_main}, but we make no attempt to do so at this point.

\section{Outline of the proof and structure of the paper}\label{sec_overview}
In this section we describe the proof of our main theorem (\ref{thm_main}).
We are given a graph $G$ on $n$ vertices and minimum degree $3n / 4$ with a red and blue colouring of the edges.
In what follows, we give an outline of our proof that $V(G)$ may be partitioned into a red cycle and a blue one.

Given an edge coloured graph, a \emph{monochromatic connected matching} is a monochromatic matching which is contained in a connected component of the same colour.
Similarly to the earlier results on our problem (by Balogh et al.~\cite{balogh_et_al} and DeBiasio and Nelsen \cite{debiasio}), as well as many other results in the area, one of the key tools is the technique of reducing problems about cycles to problems about connected matchings using the Regularity Lemma.
This technique was introduced by \L{}uczak \cite{luczak} and since then has become fairly standard.
In our setting, the basic idea, which is described in more detail in Section \ref{sec_regularity_lem}, is as follows. We are given a $2$-coloured graph $G$ and consider the reduced graph obtained by applying the Regularity Lemma. If the reduced graph has a perfect matching consisting of a red connected matching and a blue connected matching, we may use the blow-up lemma \cite{blowup} (or in fact, a much simpler special case), to find two vertex-disjoint monochromatic cycles, a red one and a blue one, which cover almost all of the vertices.

The next ingredient is the ``absorbing method'' of R\"odl, Ruci\'nski and Szemer\'edi (see \cite{rodl_et_al_absorbing} and \cite{levitt_et_al_absorbing}).
As in \cite{debiasio}, in order to apply this method, we use a notion of ``robust subgraphs'', which are defined to be graphs with certain expansion properties (see Section \ref{sec_robust_compts} for the exact definition). Such graphs can be shown to contain short ``absorbing paths'', which are paths that can absorb small sets of vertices.
We observe (see Section \ref{sec_regularity_lem}) that monochromatic connected components in the reduced graph (obtained from the Regularity Lemma) correspond to monochromatic robust subgraphs in the original graph.
This observation allows us to obtain information about the rough structure, by using the Regularity Lemma and finding the corresponding robust subgraphs (see Section \ref{sec_rough_struct}).

After a study of the some properties of the rough ``robust structure'' of the graph, we aim to apply the Regularity Lemma again, in order to find a suitable perfect matching, namely a perfect matching consisting of a connected red matching and a connected blue matching.
This matching is used to find two vertex-disjoint cycles, a red one and a blue one, which cover most of the vertices and have the additional absorbing property implying that the leftover vertices can be inserted into one of these cycles (see Sections (\ref{sec_large_bip_blue_large_red} - \ref{sec_large_blue_non_bip})).

We should like to emphasize that as we prove the sharp result, namely that we only assume that the minimum degree is at least $3n / 4$, rather than $(3/4 + \eps)n$ as in the previous results, new difficulties arise, making our proof much harder.
Firstly, we have to deal with several different cases for the rough robust structure, some of which do not arise when $\delta(G) \ge (3/4 + \eps)n$. Interestingly, these cases require a variety of ideas and techniques, making the proof of the general theorem rather intricate.
Secondly, when applying the Regularity Lemma we cannot guarantee that the minimum degree would be at least $3n / 4$, thus it may not be possible to find a suitable matching in the reduced graph directly.

The combination of the following two ideas helps us with these challenges.
The first idea is to use stability versions of results promising a perfect matching.
These enable us to conclude that if the reduced graph does not have the required perfect matching, it has some specific structure which we can further analyse to find the required monochromatic cycle partition.

The second idea is the following simple yet important observation: given two robust components of the same colour, if they can be connected with two vertex-disjoint paths, they may essentially be treated as one larger component.
In several parts of the proof (see Sections (\ref{sec_large_bip_blue_almost_balanced}, \ref{sec_two_three-quarter-sized}, \ref{sec_four_midsized})), we use this observation to conclude that either we may join two robust components to obtain a larger one, or the graph admits some restrictive structure, for which the desired partition may be found ``by hand'' (see Sections \ref{sec_red_graph_disconnected} and \ref{sec_red_disconnected_2}).
We remark that even at this stage, the proof is rather hard due to the fact that our result is sharp.

\subsection{Structure of the paper}

In the next section, Section \ref{sec_notation}, we introduce the notation that will be used in this paper.
In Section \ref{sec_robust_compts}, we define our notion of robustness and prove some properties of robust components, most notably the existence of absorbing paths.
In Section \ref{sec_regularity_lem}, we state the version of the Regularity Lemma that we use here. We also prove some results about the correspondence between connected components of the reduced graph and robust subgraphs of the original graph and describe the method of converting connect matchings in the reduced graph into cycles in the original graph.
In Section \ref{sec_prelim_results}, we list some results which will be used throughout the proof.

Sections (\ref{sec_rough_struct} - \ref{sec_red_disconnected_2}) are devoted to the proof of Theorem \ref{thm_main}. 
In Section \ref{sec_rough_struct}, we obtain some information about the rough structure and point out how to prove Theorem \ref{thm_main} using the results in subsequent lemmas. In each of Sections \ref{sec_large_bip_blue_large_red} - \ref{sec_large_blue_non_bip}, we consider one of the cases arising from the structural result in Section \ref{sec_rough_struct}.
These cases vary in difficulty and we have to use various techniques used to deal with them.
In Sections \ref{sec_red_graph_disconnected} and \ref{sec_red_disconnected_2}, we prove Lemmas \ref{lem_red_graph_disconnected_1} and \ref{lem_red_disconnected_2}, which are used in earlier sections and prove the main theorem under certain restrictive conditions on the colouring and the structure of the graph.
Finally, Section \ref{sec_conclusion} is devoted to some concluding remarks.

\section{Notation}\label{sec_notation}
We use mostly standard notation.
Write $|G|$ for the order of a graph $G$ and $\delta(G)$ and $\Delta(G)$ for its minimum and maximum degrees respectively. The neighbourhood of a vertex $x \in V (G)$ is denoted by $N_G(x)$ and its degree by $d_G(x) = |N_G(x)|$. Given $A \subseteq V (G)$, we write $N_G(x, A) = N_G(x) \cap A$ and $d_G(x, A) = |N_G(x, A)|$.
We will write, for example, $d(x, A)$ for $d_G(x, A)$ if this is unambiguous.
Given a set of vertices $X \subseteq V(G)$, we write $G[X]$ for the graph induced by $G$ on $X$. Similarly, for disjoint subsets $X, Y \subseteq V(G)$, we write $G[X, Y]$ for the bipartite graph with bipartition $\{X, Y\}$ induced by $G$.
We denote $e_G(X, Y ) = |E(G[X, Y])|$.

Given a graph $G$, we denote a $2$-colouring of $G$ by $E(G) = E(G_B) \cup E(G_R)$, where $G_B, G_R$ are graphs on vertex set $V(G)$ (in proper colouring, the graphs $G_B, G_R$ are edge-disjoint).
The edges of $G_B$ are called blue edges and the edges of $G_R$ are red edges.
We sometimes use $B$ or $R$ for a subscript instead of $G_B$ or $G_R$. For instance, $N_B(x)$ is a shorthand for $N_{G_B}(x)$.

We denote by $(u_1 u_2 \ldots u_k)$ the path on vertices $u_1, \ldots, u_k$ taken in this order. We use the same notation to denote the cycle obtained by adding the edge $(u_k, u_1)$ to the given path. It should be clear from the context if we are dealing with a path or a cycle.
Given paths $P_1, P_2$ which share an end and are otherwise disjoint, we denote by $(P_1P_2)$ the concatenation of the two paths. Similarly, if the paths share both ends but are otherwise disjoint, the same notation denotes the cycle obtained by joining the two paths.

Throughout this paper we omit floors and ceilings whenever the argument is unaffected. The constants in the hierarchies used to state our results are chosen from right to left. For example, the claim that a statement holds for $0 < a, \frac{1}{n} \ll b \ll c \ll 1$ means that there
exist non-decreasing functions $f, g: (0, 1] \rightarrow (0, 1]$ and a constant $c_0$ such that the
statement holds for all $0 < a, b, c \le c_0$ and integers $n$ with $b \le f(c), a \le g(b)$ and $n \ge \frac{1}{g(b)}$. 
We normally do not specify the functions in question.

\section{Robust components and absorbing paths}\label{sec_robust_compts}

Similarly to the proof of DeBiasio and Nelsen \cite{debiasio}, one of the main tools in our proof is the notion of robust subgraphs.
As we shall see, these are graphs with certain expansion properties.
The role of robust subgraphs in our proof is similar to their role in \cite{debiasio}, but our definition is different and is often easier to apply. Nevertheless, the two definitions are in some sense equivalent, as can be seen in Lemma 5.4 in \cite{debiasio}.
After defining a robust subgraph, we state and prove some simple properties of these components. Finally, we prove that robust subgraphs contain ``absorbing paths'', which may absorb small sets of vertices.

\subsection{Definitions of robust subgraphs}

We define two notions of robustness: strong and weak.
The difference between the two is that strong robust subgraphs are far from being bipartite.
It will be easier for our application, though not essential, to define a robust subgraph relative to a fixed ground graph.
The precise definitions are as follows.

Given a graph $G$, vertices $x,y\in V(G)$ and an integer $l$, denote by $\con_{G,l}(x,y)$ the number of paths of length $l + 1$ in $G$ between $x$ and $y$.

\begin{define}\label{def_robust_compt}
Let $G$ be a graph on $n$ vertices.

A subgraph $F$ of $G$ is called $(\alpha,k)$ \emph{strongly robust} if there exists $l\le k$ such that for every pair of vertices $x,y$ in $F$ we have $|con_{F,l}(x,y)|\ge \alpha n^l$.

A subgraph $F$ of $G$ is called $(\alpha,k)$ \emph{weakly robust} if there exists a partition $\{X,Y\}$ of $V(F)$ such that for some $l \le k$ and every $x \in X, y \in Y$ we have $|con_{F',l}(x,y)|\ge \alpha n^l$, where $F' = F[X, Y]$.
\end{define}
As a shorthand, we often omit the parameters $\alpha$ and $k$ when they are clear from the context.
From now on, we use the term ``robust'' to signify either strongly robust or weakly robust with suitable parameters.
We point out that in our context $\alpha$ and $k$ are fixed and $n$ tends to infinity.
We remark that an $(\alpha, k)$ robust subgraph $F$ of a graph $G$ on $n$ vertices has $\delta(F) \ge \alpha n$. In particular, robust subgraph are always dense.

Before discussing some properties of robust subgraphs, let us give a few examples.
Any graph of minimum degree at least $(1/2 + \alpha/2)n$ is $(\alpha, 1)$ strongly robust, because any two vertices have at least $\alpha n$ common neighbours. 
Similarly, the random graph $G(n, \alpha)$, is w.h.p.~$(\alpha^2 / 2, 1)$ strongly robust.
Furthermore, the blow-up of a path of length $k \ge 2$, where every vertex is replaced by a complete graph on $n/k$ vertices, is $(\alpha, k - 1)$ strongly robust for a suitable $\alpha$. 
Similarly, if vertices of a path of length $k$ are replaced by empty graphs, we obtain an $(\alpha, k - 1)$ weakly robust graphs.

\subsection{Properties of robust subgraphs}
We shall make use of some simple properties of robust graphs.
The following lemma states that a robust subgraph remains robust after removing a small number of vertices.
\begin{lem}\label{lem_robust_removing_vertices}
Given $\alpha > 0$ and $k$ an integer, the following holds for small enough $\beta$.
Let $G$ be a graph on $n$ vertices and let $F$ be an $(\alpha, k)$-robust subgraph.
Suppose that $F'$ is obtained from $F$ by removing at most $\beta n$ vertices. Then $F'$ is $(\alpha/2, k)$-robust.
\end{lem}

\begin{proof}
We prove the lemma under the assumption that $F$ is strongly robust; the proof in case $F$ is weakly robust is analogous.
Let $l \le k$ satisfy $|con_{F, l} (x, y)| \ge \alpha n^l$ for every $x, y \in V(F)$.
For every $x,y\in V(F)$, the number of paths of length $l + 1$ between $x$ and $y$ containing at least one vertex from $V(F) \setminus V(F')$ is at most $l \beta n^l \le \frac{\alpha}{2} n^l$. It follows that $|\con_{F', l}(x, y)| \ge \frac{\alpha}{2} n^l$ for every $x, y \in V(F)$, i.e.~$F'$ is $(\alpha/2, l)$ strongly robust.

\end{proof}

The next lemma shows that a robust subgraph remains robust after removing a graph of small maximum degree.
\begin{lem}\label{lem_robust_removing_edges}
Given $\alpha>0$ and $k$, the following holds for suitably small $\beta$ and large $n$.
Let $G$ be a graph on $n$ vertices and let $F$ be an $(\alpha, k)$-robust subgraph.
Suppose that $F'$ is a subgraph of $F$ such that for every vertex $v\in V(F)$ we have $\deg_{F'}(x) \ge \deg_F(x) - \beta n$. Then $F'$ is $(\alpha/2, k)$-robust.
\end{lem}

\begin{proof}
We prove this lemma for $F$ strongly robust; the proof for $F$ weakly robust is similar.
Let $l\le k$ be such that $|con_{F, l} (x, y)| \ge \alpha n^l$ for every $x, y \in V(F)$.
Fix some $x,y\in V(F)$.
We consider the family of paths in $F$ of length $l + 1$ between $x$ and $y$ which contain at least one edge outside of $F'$.
There are at most $l\beta n^l$ such paths, i.e. $|con_{F', l} (x, y)| \ge |con_{F', l} (x, y)| - l \beta n^l \ge \frac{\alpha}{2} n^l$ (for small enough $\beta$).
It follows that $F'$ is $(\alpha/2, l)$ strongly robust.

\end{proof}

The following lemma states that a robust subgraph $F$ remains robust after the addition of vertices which have a large neighbourhood in $F$.
\begin{lem}\label{lem_robust_adding_vertices}
Given $\alpha > 0$ and $k$ an integer, the following holds for large enough $n$.
Let $G$ be a graph on $n$ vertices and let $F$ be an $(\alpha, k)$-robust component.
Let $F'$ be a subgraph of $G$ containing $F$, such that every vertex in $V(F')\setminus V(F)$ has at least $\alpha n$ neighbours in $F$.
Then $F'$ is $(\alpha^3/2, k + 2)$-robust.
\end{lem}

\begin{proof}
We prove the statement assuming that $F$ is strongly robust; the proof in case $F'$ is weakly robust is very similar and we omit the details.
Let $l\le k$ satisfy $|con_{F, l} (x, y)| \ge \alpha n^l$ for every $x, y \in V(F)$.
Fix some $x, y\in V(F')$. For every $z, w\in V(F)$ such that $z \in N(x)$ and $w \in N(y)$, we have $|con_{F, l} (w, z)| \ge \alpha n^l$.
Thus, since every vertex in $V(F')$ has at least $\alpha n$ neighbours in $F$, we have that the number of walks between $x$ and $y$ in $F'$ with $l + 2$ inner vertices is at least $\alpha^3 n^{l+2}$ for large $n$.
Since there are at most $O(n^{l + 1})$ such walks which are not paths, we have that $|con_{F', l} (x, y)| \ge \frac{\alpha^3}{2} n^{l + 1}$. 
It follows that $H'$ is $(\alpha^3/2, k + 2)$-strongly robust.
\end{proof}
So far we listed and proved several simple properties of robust subgraphs. In the following subsection we state and prove a more interesting property.

\subsection{Absorbing paths}
The main reason robust subgraphs are so useful in our context, is the fact, which was proved by DeBiasio and Nelsen \cite{debiasio} that they contain short ``absorbing paths''. We conclude this section with a proof of this fact.

\begin{lem}\label{lem_absorbing_paths}
Let $\frac{1}{n} \ll \rho \ll \alpha, \frac{1}{k} \ll 1$, let $G$ be a graph on $n$ vertices and let $F$ be an $(\alpha,k)$-robust subgraph of $G$.
Then there exists a path $Q$ in $F$ satisfying the following conditions.
\begin{enumerate}
\item
If $F$ is strongly robust, for every set $W \subseteq V(F) \setminus V(Q)$ of size at most $\rho^2 n$, there exists a Hamilton path in $F[V(Q) \cup W]$ with the same ends as $Q$.
\item
If $F$ is weakly robust with bipartition $\{X, Y\}$, for every $W \subseteq V(F) \setminus V(P)$, with $|W \cap X| = |W \cap Y|\le \rho^2 n$, the graph $F[V(Q) \cup W]$ contains a Hamilton path with the same ends as $Q$.
\end{enumerate}
\end{lem}

We follow the footsteps of DeBiasio and Nelsen in their proof of Lemma 5.6 from \cite{debiasio}.
The main tool is absorbing method of R\"odl, Ruci\'nski and Szemer\'edi \cite{rodl_et_al_absorbing}.
We shall use ``gadgets'', which we will define to be Hamiltonian graphs that are can absorb a single vertex under the condition that it is adjacent to some of the vertices in the gadget.
By a simple application of the probabilistic method and the robustness of the given graph, we show that there exists a not too large collection of vertex-disjoint gadgets, such that every vertex may be absorbed by a rather large number of them.
From there it will be easy to construct the required path $Q$. 

\begin{proof}[Proof of Lemma \ref{lem_absorbing_paths}]
We start by proving the first part of Lemma \ref{lem_absorbing_paths}.
Suppose that $F$ is $(\alpha, k)$-strongly robust.
In particular, $\delta(F) \ge \alpha n$ so we may apply the following claim.
\begin{claim}\label{claim_disjoint_pairs}
Let $p, \frac{1}{n} \ll \alpha \ll 1$ and let $F$ be a graph on at most $n$ vertices with $\delta(G) \ge \alpha n$.
Then there exists a family $\C{F}$ of disjoint pairs of vertices of $V(F)$ such that the following conditions hold.
\begin{itemize}
\item
$|\C{F}| \le p n$
\item
For every $u \in V(F)$ there are at least $\frac{1}{16}p \alpha n^2$ pairs $(x, y) \in \C{F}$ such that $x, y$ are neighbours of $u$.
\end{itemize}
\end{claim}

This claim is a simple application of Chernoff's bound.
We shall pick a family of pairs of vertices randomly and then delete a small number of pairs so as to ensure that the pairs are disjoint.

\begin{proof}[Proof of Claim \ref{claim_disjoint_pairs}]
Let $\C{F}$ be the family of pairs obtained by choosing each pair of vertices in $V(G)$ independently with probability $\frac{p}{n}$.
By Chernoff's bound, we have that with high probability, the following properties hold.
\begin{itemize}
\item
$|\C{F}| \le 2\frac{p}{n} \binom{n}{2} \le p n$.
\item
For every $u \in V(F)$, $\C{F}$ contains at least $\frac{1}{2} \frac{p}{n} \binom{\alpha n}{2} \ge \frac{1}{8} p \alpha^2 n$ pairs $(x, y)$ such that $x, y \in N(u)$.
\end{itemize}
The expected number of pairs of intersecting pairs in $\C{F}$ is at most $(\frac{p}{n})^2 n^3 \le p^2 n$.
It follows by Markov's inequality that with probability at least $1/2$, the number of pairs of intersecting pairs in $\C{F}$ is at most $2 p^2 n$.
In particular, we may pick a family $\C{F}$ which satisfies the above conditions and which has at most $2 p^2 n$ pairs of intersecting pairs.
We obtain a subfamily $\C{F'}$ of $\C{F}$ containing no intersecting pairs by deleting at most $2 p^2 n$ pairs from $\C{F}$.
It is easy to verify that if $p$ is suitably small, $\C{F'}$ satisfies the requirements of the claim.
\end{proof}

Let $\C{F} = \{(x_j, y_j)\}_{j = 1}^N$ be a family of pairs as in Claim \ref{claim_disjoint_pairs} (so $N \le pn$).
We use the following simple technical claim, to avoid divisibility issues.
\begin{claim}\label{claim_4l}
Let $\beta$ be suitably small and $n$ suitably large. Then for some $1 \le l \le k$, there are at least $\beta n^{4l - 2}$ paths of length $4l - 1$ between each pair of vertices in $F$.
\end{claim}
\begin{proof}
Since $F$ is $(\alpha, k)$ strongly robust, there exists $l \le k$ such that between every $u, v \in V(F)$ there are at least $\alpha n^{l - 2}$ paths in $F$ of length $l - 1$.
We conclude that for every $u, v \in V(F)$ there are at least $\alpha^{10}n^{4l - 2}$ walks of length $4l - 1$ between $u$ and $v$.
Indeed, given $u, v \in V(F)$, there are at least $(\alpha n)^6$ ways to pick edges $e_1, e_2, e_3 \in E(F)$, since we may pick one end of each edge in at least $|F| \ge \alpha n$ ways, and then there are at least $\alpha n$ ways to pick a neighbour.
Denote $e_i = (a_i, b_i)$.
There are $(\alpha n^{l - 2})^4$ ways to pick paths $P_1, P_2, P_3, P_4$ in $F$ of length $l - 1$ with ends $u$ and $a_1$, $b_1$ and $a_2$, $b_2$ and $a_3$, $b_3$ and $v$ respectively.
It follows that there are at least $\alpha^{10} n^{4l - 2}$ walks in $F$ between $u$ and $v$. At most $O(n^{4l - 3})$ of them are not paths, so for large enough $n$, there are at least $\frac{\alpha^{10}}{2} n^{4l -2}$ walks of length $4l - 2$ between $u$ and $v$.
\end{proof}

We shall build vertex-disjoint paths $Q_j$ of length $8l^2 - 4l + 1$ one by one for $j = 1, \ldots, N$ as follows.
Suppose that $Q_1, \ldots, Q_{j - 1}$ are already defined.
We would like to pick paths as follows.
\begin{align*}
&P = (u_1, \ldots, u_{4l}) \text{ a path with ends }x_j, y_j \\
&P_i \text{ a path of length } 4l - 1 \text{ with ends } u_i, u_{i + 3} \text{, for } i = 1, 3, \ldots 4l - 5. \\
&P_{4l - 3} \text{ a path of length } 4l - 1 \text{ with ends } u_{4l - 3}, u_{4l - 1}.
\end{align*}
It is easy to see, by the choice of $l$ according to Claim \ref{claim_4l}, that if $p$ is small enough, we may pick such paths to be vertex-disjoint of all previously defined paths and to have pairwise disjoint interiors.
\begin{equation*}
Q_j = (u_2 u_1 P_1 u_4 u_3 P_3 u_6 \ldots u_{4l - 2} u_{4l - 3} P_{4l - 3} u_{4l - 1} u_{4l}).
\end{equation*}

Suppose that $w \in V(G)$ is a neighbour of $x_j = u_1$ and $y_j = u_{2l}$.
The following path is a path in $F$ with vertex set $V(Q_j) \cup \{w\}$ and same ends as $Q_j$ (this path is illustrated in Figure \ref{figure_absorbing1} together with $Q_j$).
\begin{align*}
(u_2 u_3 P_3 u_6 u_7 P_7 \ldots u_{4l - 5} P_{4l - 5} u_{4l - 2} u_{4l - 1} P_{4l - 3} u_{4l - 3} u_{4l - 4} P_{4l - 7} \ldots u_5 u_4 P_1 u_1 w u_{4l}).
\end{align*}

\begin{figure}[h]\centering 

\includegraphics[scale=1.5]{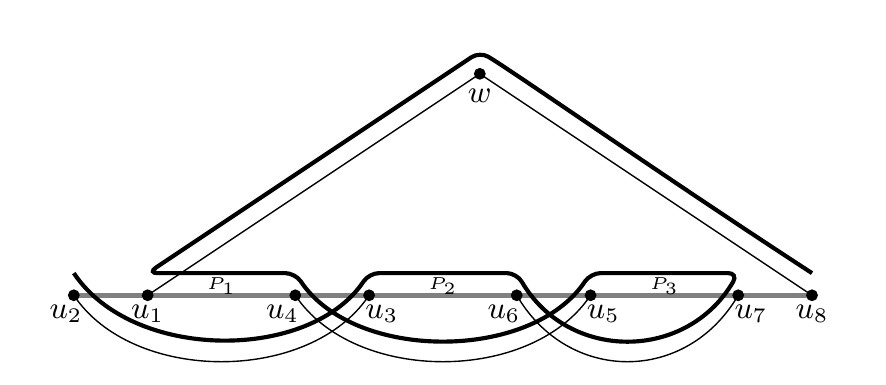}

\caption{An illustration of the absorbing structure for $l = 2$. The path $Q_j$ is represented by the straight line between $u_2$ and $u_8$ which is marked in grey, and the path absorbing $w$ is represented by the bold black path.}
\label{figure_absorbing1}
\end{figure}

Finally, we let $Q$ be a path which contains $Q_1, \ldots, Q_N$ by connecting the ends of the $Q_j$'s with paths of length $4l - 1$.
Denote $\rho = 8 p k^2$, and note that we may pick $p$ small enough such that $\rho^2 \le \frac{1}{16} p \alpha ^2$.
The length of $Q$ is at most $\rho n$ and for every vertex $z \in V(F) \setminus V(Q)$, there are at least $\rho^2 n$ values of $j \in [N]$ such that $x_j, y_j \in N(z)$.
We show that $Q$ has the desired absorbing property.
Let $W$ be a set of at most $\rho^2 n$ vertices in $V(F) \setminus V(Q)$ and denote $W = \{w_1, \ldots, w_M\}$.
We may pick distinct $j_1, \ldots, j_M \in [N]$ such that $x_{j_i}, y_{j_i} \in N(w_i)$.
Recall that for each $i \in [M]$, there is a path $Q_{j_i}'$ in $F$ on vertex set $V(Q_{j_i}) \cup \{w_i\}$ with the same ends as $Q_j$.
By replacing the occurrence of $Q_{j_i}$ by $Q_{j_i}'$ in the path $Q$, we obtain a path on vertex set $V(Q) \cup W$ with the same ends as $Q$.

We now turn to the proof of the second part of Lemma \ref{lem_absorbing_paths}.
Let $F$ be an $(\alpha, k)$ weakly robust component with bipartition $\{X, Y\}$.
The proof will use similar ideas, with some variations which take into account the bipartition of $F$.
A similar argument as in Claim \ref{claim_disjoint_pairs} implies that for small enough $p$ we may find a family $\C{F}$ of disjoint quadruples of vertices of $F$ with the following properties.
\begin{itemize}
\item
$|\C{F}| \le p n$
\item
For every $x \in X, y \in Y$ the number of quadruples $(a, b, c, d)$ such that $a, c \in N(x)$ and $b, d \in N(y)$ is at least $\frac{1}{16} p \alpha^4 n$.
\end{itemize}
Denote $\C{F} = \{(a^j, b^j, c^j, d^j)\}_{j = 1} ^{j = N}$.
For $j = 1, \ldots, N$ we pick a path $Q_j$ as follows.
As before, there exists $1 \le l \le k$ such that $|\con_{F, 4l - 2}(x, y)| \ge \beta n^{4l - 2}$ for every $x \in X, y \in Y$.
Assuming that $Q^1, \ldots, Q^{j - 1}$ were already chosen to be paths of length at most $10l^2$, we may pick paths
\begin{align*}
& (a_1 b_1 \ldots a_{2l} b_{2l}) \text{ a path with ends }a_1 = a^j, b_{2l} = b^j \\
& (c_1 d_1 \ldots c_{2l} d_{2l}) \text{ a path with ends }c_1 = c^j, d_{2l} = d^j 
\end{align*}
which are vertex-disjoint of each other and of previously defined paths.
We may further pick paths of length at most $k$
\begin{align*}
&P_i \text{ a path with ends } 
\left\{
\begin{array}{ll}
a_1, c_{2l} & \text{if } i = 1 \\
a_i, b_{i + 1} & \text{if } 2 \le i \le 2l - 2 \\
a_{2l - 1}, a_{2l} & \text{if } i = 2l - 1
\end{array}
\right.
\\
&Q_i \text{ a path with ends } 
\left\{
\begin{array}{ll}
d_2, d_1 & \text{if }i = 2 \\
d_i, c_{i - 1} & \text{if } 3 \le i \le 2l 
\end{array}
\right.
\\
&R \text{ a path with ends } c_1, b_2.
\end{align*}
which are pairwise vertex-disjoint and are disjoint of previously defined paths.
Let
\begin{align*}
Q^j = (b_1 a_1P_1c_{2l} d_{2l} Q_{2l} c_{2l - 1} \ldots d_3 Q_3 c_2 d_2 Q_2 d_1 c_1 R b_2 a_2 P_2 b_3 \ldots a_{2l - 2} P_{2l - 2} b_{2l - 1} a_{2l - 1} P_{2l - 1} a_{2l} b_{2l}).
\end{align*}
Suppose that $x, y$ satisfy $a^j, c^j \in N(x)$ and $b^j, c^j \in N(y)$. Then the following path is a path in $F$ with vertex set $V(Q^j) \cup \{x, y\}$ with the same ends as $Q^j$ (See figure \ref{figure_absorbing2}).
\begin{align*}
&(b_1 a_2 P_2 b_3 a_4 P_4 b_5  \ldots a_{2l - 2} P_{2l - 2} b_{2l - 1} a_{2l} 
P_{2l - 1} a_{2l - 1} b_{2l - 2} P_{2l - 3} a_{2l - 3} \ldots  b_4 P_3 a_3 b_2 R c_1 x \\
&a_1 P_1 c_{2l} d_{2l - 1} Q_{2l - 1} c_{2l - 2} \ldots d_3 Q_3 c_2 d_1 Q_2 d_2 c_3 Q_4 d_4 \ldots c_{2l - 1} Q_{2l} d_{2l} y b_{2l}).
\end{align*} 
 We may proceed as in the previous part to complete the proof of Lemma \ref{lem_absorbing_paths}.

\end{proof}

\begin{figure}[h]\centering 

\includegraphics[width = \textwidth]{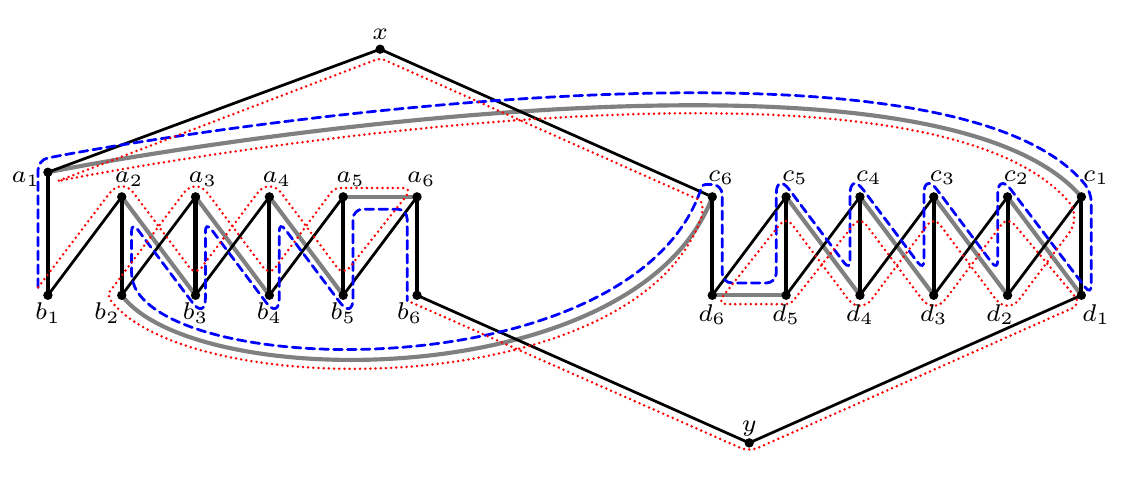}

\caption{An illustration of the absorbing structure for $l = 3$. The black lines represent edges, whereas the grey ones represent paths. The dashed blue line represents the path $Q^j$ and the dotted red one represents the path obtained from $Q^j$ by absorbing $w$.}
\label{figure_absorbing2}
\end{figure}

The proof of Lemma \ref{lem_absorbing_paths} concludes our introduction of the notion of robust subgraphs and their properties.
In order to make use of the properties we established, we shall use Lemmas \ref{lem_reduced_components}, \ref{lem_many_reduced_components} and \ref{lem_robust_regularity} from the next section, Section \ref{sec_regularity_lem}. These lemmas establish the connection between connected components of the reduced graph (obtained by an application of the Regularity Lemma) and robust subgraphs.

\section{The Regularity Lemma}\label{sec_regularity_lem}
In our proof, we shall use Szemer\'edi's Regularity Lemma extensively.
Before stating the version we use, we introduce some notation.
Let $U,W$ be disjoint subsets of vertices of a graph $G$.
The density $d(U,W)$ of edges between $U$ and $W$ is defined to be 
\begin{equation*}
d(U,W)=\frac{e(U,W)}{|U|W|},
\end{equation*}
where $e(U,W)$ is the number of edges between $U$ and $W$.
A bipartite graph with bipartition $U,W$ is said to be \emph{$\eps$-regular} if for every $U'\subseteq U$ and $W'\subseteq W$ with $|U'|\ge \eps |U|$ and $|W'|\ge \eps |W|$, the density $d(U',W')$ satisfies
$|d(U',W') - d(U,W)|\le \eps$.

We use a variant of the so-called degree form of the Regularity Lemma (see \cite{regularity}), which is applicable to $2$-coloured graphs. Furthermore, it will be useful for our purpose to start with a cover of the vertices of a graph and require that the partition obtained by the lemma is a refinement of the initial cover.

\begin{lem}\label{lem_regularity}
For every $\eps>0$ and integer $l$ there exists $M=M(l, \eps)$ such that the following holds.
Let $G$ be a $2$-coloured graph on $n$ vertices, $\C{C}$ a cover of $V(G)$ with at most $l$ parts and $d > 0$.
Then there exists a partition $\{V_0, \ldots, V_m\}$ of $V(G)$ and a subgraph $G'$ of $G$ with vertex set $V(G) \setminus V_0$, such that the following conditions hold.
\begin{enumerate}[(R1)]
\item
$m\le M$.
\item
Every $V_i$ ($i\in [m]$) is contained in one of the parts of $\C{C}$.
\item
$|V_0|\le \eps n$ and $|V_1| = \ldots = |V_m| \le \lceil \eps n \rceil$.
\item \label{itm_reg_deg}
$\deg_{G'}(v) \ge \deg_G(v) - (2d + \eps)n$ for every $v \in V_1 \cup \ldots \cup V_m$.
\item
$e(G'[V_i]) = 0$ for $i\in [m]$.
\item
All pairs $(V_i,V_j)$ are $\eps$-regular in both colours in $G'$, with density in each colour either $0$ or at least $d$.
\end{enumerate}
\end{lem}

It is often useful to work with the reduced graph, obtained from the partition given by the Regularity Lemma as follows.
Given a $2$-coloured graph $G$, and parameters $\eps, d, l$, we define the $(\eps, d)$-reduced graph $\Gamma$ as follows.
Let $\{V_0, \ldots, V_m\}$ be the partition obtained by an application of Lemma \ref{lem_regularity}, and let $G'$ be the given subgraph of $G$. 
We take $V(\Gamma)=\{V_1, \ldots, V_m\}$.
A pair $V_i V_j$ is a $t$-coloured edge in $\Gamma$ if it has density at least $d$ in colour $t$ in $G'$.
Note that an edge of $\Gamma$ can have more than one colour.

The following observation shows why it is useful to work with the degree form of the Regularity Lemma.

\begin{obs}\label{obs_min_deg_reduced_graph}
Let $G$ be a $2$-coloured graph on $n$ vertices with $\delta(G) \ge cn$ and let $\Gamma$ be the $(\eps, d)$-reduced graph obtained by applying Lemma \ref{lem_regularity}.
Then $\delta(\Gamma) \ge (c - 2d - \eps)m$, where $m = |\Gamma|$.
\end{obs}

The rest of this section is divided into two parts. In the first part, Subsection \ref{subsec_robust_reg}, we establish the connection between robust subgraph and connected components of a reduced graph. In the second part, Subsection \ref{subsec_conn_matching}, we describe the connection between connected matchings in a reduced graph and cycles in the original graph.

\subsection{From connected components of the reduced graph to robust subgraphs} \label{subsec_robust_reg}

One of our main tools in the proof of Theorem \ref{thm_main} is the following lemma. It gives us the means to obtain a robust subgraph from a connected subgraph of the reduced graph.

\begin{lem} \label{lem_reduced_components}
Let $\alpha, \frac{1}{k}, \frac{1}{n} \ll \eps, \frac{1}{l} \ll 1$ and $d \ge 4\eps$.
Let $G$ be a graph on $n$ vertices and let $\Gamma$ be the $(\eps, d)$-reduced graph obtained by an application of Lemma \ref{lem_regularity}.
Suppose that $\Phi$ is a connected subgraph of $\Gamma$.
Then there exists a subgraph $F$ of $G$ with the following properties.
\begin{enumerate}
\item \label{itm_1_lem_reduced_components}
Let $U$ be the set of vertices contained in clusters of $V(\Phi)$. Then $V(F) \subseteq U$ and $|U \cap V(F)| \ge (1 - \eps)|U|$.
\item \label{itm_2_lem_reduced_components}
$F$ is $(\alpha, k)$-robust.
If $\Phi$ is bipartite, $F$ is weakly robust, otherwise it is strongly-robust.
\end{enumerate}
\end{lem}

In order to find the required robust component $F$, we consider the clusters represented by $V(\Phi)$, and for each of the clusters we remove vertices with low degree. The regularity of pairs of clusters which are connected by an edge, together with the choice of $d$, implies that the number of low degree vertices in each cluster is small.
We show that the subgraph induced by the remaining vertices has the required expansion properties, using the regularity of the pairs of clusters.

\begin{proof}[Proof of Lemma \ref{lem_reduced_components}]

Denote $I = \{i \in [m]: V_i \in V(\Gamma)\}$, so $U = \bigcup_{i \in I}V_i$.
For every $i\in I$, denote by $I_i$ the set of indices $j\in I$ such that $V_iV_j\in E(\Phi)$ and let $N_i = \bigcup_{j \in I_i}V_j$.
Let $G'$ be the subgraph of $G$ given by Lemma \ref{lem_regularity}). Let
\begin{equation*}
W_i = \{v\in V_i: \deg_{H}(v,N_i)\le 3\eps |N_i|\}.
\end{equation*}

\begin{claim}
$ |W_i|\le \eps n$ for every $i \in I$.
\end{claim}
\begin{proof}
Suppose otherwise.
Recall that for every $j \in I_i$, $(V_i,V_j)$ is an $\eps$-regular pair in $G'$ with density at least $4\eps$.
It follows by the definition of a regular pair that $d_{G'} (W_i,V_j)\ge d_{G'}(V_i,V_j)-\eps \ge 3\eps$.
Hence, $e_{G'}(W_i,V_j)\ge 3\eps |W_i||V_j|$ for $j\in I_i$ and $e_{G'}(W_i,N_i)\ge 3\eps|W_i||N_i|$.
It follows that there exists a vertex in $W_i$ which is incident to at least $3\eps|N_i|$ edges into $N_i$, a contradiction to the choice of $W_i$.
\end{proof}

Define $W=\bigcup_{i\in I} W_i$ and $F=G_B[U\setminus W]$.
Note that $F$ satisfies Property (\ref{itm_1_lem_reduced_components}) in Lemma \ref{lem_reduced_components}.
It remains to show that Property (\ref{itm_2_lem_reduced_components}) holds.
We suppose that $\Phi$ is non-bipartite, the proof for the bipartite case follows similarly.
We use the following simple claim.
\begin{claim}\label{claim_uniform_path_length}
Let $G$ be a connected non-bipartite graph on $n$ vertices.
Then there exists $k\le 3n$ such that between every two vertices of $G$ there is a walk of length $k$.
\end{claim}

\begin{proof}
We first show that between every two vertices of $G$ there is a walk of odd length not exceeding $3n$.
Indeed, let $x, y\in V(G)$.
Let $C$ be an odd cycle, and pick some $z\in C$.
Pick some path from $x$ to $z$ and a path from $z$ to $y$.
Combining the two paths, we obtain a walk from $x$ to $y$ of length at most $2n$.
If this walk has even length, we add the cycle $C$ to it. In any case we obtain an odd walk of length at most $3n$. Let $k$ be the length of the longest of these walks. By possibly adding $2$-cycles to the given walks, we obtain walks of length $k$ between each pair of vertices.
\end{proof}

Fix some $k \le 3n$ as in the previous claim.
In the following claim, we show that the existence of a walk of length $k$ in $\Phi$ implies the existence of many paths of the same length in $F$.

\begin{claim}\label{claim_many_paths}
Let $k\le 3m$. There exists $\beta = \beta(l, \eps)$ such that the following holds.
Suppose that $V_{i_1}, \ldots, V_{i_k}$ is a walk in $\Phi$.
Let $U_1$ and $U_k$ be subsets of $V_{i_1} \cap V(F)$ and $V_{i_k} \cap V(F)$ respectively of size at least $2\eps |V_1|$.
Then there exists at least $\beta n^{k+1}$ paths of length $k$ in $F$ between $U_1$ and $U_k$.
\end{claim}

\begin{proof}
Let $X_j=\{v\in V_{i_j}: \deg_H(v,V_{i_{j+1}})\le 3\eps |V_1|\}$ for $j\in [k-1]$.
As before it is easy to show that $|X_j|\le \eps |V_1|$.
Let
\begin{equation*}
Y_j=\left\{
\begin{array}{ll}
U_1\setminus X_1 & j=1 \\
 V_{i_j}\setminus (X_j\cup W_{i_j})& j\in [2,k-1]
\end{array} \right .
\end{equation*}
Recall that $|W_j|\le \eps |V_1|$ for every $j$.
It follows that for $j\in [k-1]$ and $v\in Y_j$, we have $\deg_H(v,Y_{j+1})\ge \eps |V_1|$.
Thus there are at least $(\eps |V_1|)^{k-2}$ walks of length $k - 3$ between $Y_1$ and $Y_{k-2}$. Fix any such walk $(v=v_1,\ldots, v_{k-2})$. Denote by $A$ the neighbourhood of $v_{k-2}$ in $Y_{k-1}$. Then $|A|\ge \eps |V_1|$ and recall that $|U_k|\ge \eps |V_1|$.
Since $(V_{i_{k - 1}}, V_{i_k})$ is an $\eps$-regular pair in $H$ with density at least $4\eps$, we have $e(A, U_k) \ge 3\eps |A| |U_k| \ge 3 \eps^3 |V_1|^2$. Each such edge completes the above walk into a distinct walk in $H$ between $v_1$ and $U_k$ of length $k - 1$.

We conclude that the total number of walks of length $k - 1$ between $U_1$ and $U_k$ is at least $\eps^{k+1}|V_1|^k \ge \eps^{k + 1}(\frac{1 - \eps}{M})^k n^k$, where $M$ is as in Lemma \ref{lem_regularity}.
Note that the number of such walks which are not paths (namely, a vertex appears more than once) is $O(n^{k-1})$.
It follows that indeed, there are at least $\beta n^k$ paths between $U_1$ and $U_k$ of length $k - 1$, where $\beta$ is a suitable constant.

\end{proof}

Let $x, y \in V(F)$ and suppose that $x \in V_i$ and $y \in V_j$.
Note that the definition of $F$ implies the existence of $t, s \in I$ such that $(V_i, V_t),(V_j, V_s)\in E(\Phi)$ and $x$ and $y$ have at least $2 \eps |V_1|$ neighbours in $H$ in $V_t$ and $V_s$ respectively.
 It follows from Claim \ref{claim_many_paths} that there are at least $\beta n^k$ paths from the neighbourhood of $x$ in $V_t$ to the neighbourhood of $y$ in $V_s$.
This shows that $|con_{F, k} (x, y)| \ge \frac{\beta}{2} n^k$ for every $x, y \in V(F)$, implying that $F$ is $(\beta/2, m)$ strongly robust.

\end{proof}

In fact, we need a stronger version of Lemma \ref{lem_reduced_components}. In our application, since we deal with $2$-coloured graph, we will typically have two collections of connected subgraphs of $\Gamma$, one for each colour, and it would be useful to obtain collections of robust components which preserve containment. For example, if in the reduced graph we have blue components $\Phi_1, \Phi_2$ and red components $\Phi_3, \Phi_4$ satisfying $V(\Phi_1) \cup V(\Phi_2) = V(\Phi_3) \cup V(\Phi_4)$, we would like the corresponding robust subgraphs $F_1, F_2, F_3, F_4$ to satisfy the corresponding equality, namely $V(F_1) \cup V(F_2) = V(F_3) \cup V(F_4)$.
This is achieved by the following lemma. We remark that the proof is very similar to the previous one, so we omit it.

\begin{lem}\label{lem_many_reduced_components} 
Let $\alpha, \frac{1}{k}, \frac{1}{n} \ll \eps, \frac{1}{l} \ll 1$ and $d \ge 6\eps$. Let $G$ be a graph on $n$ vertices with a $2$-colouring $E(G) = E(G_B) \cup E(G_R)$.
Let $\Gamma$ be the $(\eps, d)$-reduced graph obtained by an application of Lemma \ref{lem_regularity} and let $\{V_0, \ldots, V_m\}$ be the corresponding partition of $V(G)$.
Let $\C{P}_B$ and $\C{P}_R$ be collections of disjoint connected subgraphs of $\Gamma_B$ and $\Gamma_R$ respectively. 
Then there exist subsets $U_i\subseteq V_i$ for $i\in [m]$ satisfying the following properties.
\begin{enumerate}[(i)]
\item
$|U_i|\ge (1- 2 \eps)|V_i|$ for $i\in [m]$.
\item
Let $\Phi \in \C{P}_t$, where $t \in \{B, R\}$ and denote $I = \{i \in [m]: V_i \in V(\Phi)\}$.
Then the graph $F = G_t[\bigcup_{i\in I}U_i]$ is $(\alpha,k)$-robust.
If $\Phi$ is bipartite, $F$ is weakly robust, otherwise it is strongly-robust.
\end{enumerate}

\end{lem}

In our proof we shall find robust components and then apply the Regularity Lemma. Therefore we need the following result, stating that given a robust component $F$ in $G$, the corresponding subgraph of $F$ in the reduced graph $\Gamma$ is connected.

\begin{lem}\label{lem_robust_regularity}
Let $\eps, \frac{1}{n} \ll \alpha, \frac{1}{k}, \frac{1}{l} \ll 1$ and let $G$ be a graph on $n$ vertices with a $2$-colouring $E(G) = E(G_B) \cup E(G_R)$.
Suppose that $F$ is an $(\alpha,k)$-robust component of $G_t$, were $t \in \{B, R\}$ and let $\C{C}$ be a cover of $V(G)$ with at most $l$ parts refining $\{V(F),V(G)\setminus V(F)\}$.
Let $\Gamma$ be the $(\eps, d)$-reduced graph obtained by an application of Lemma \ref{lem_regularity}.
Then the $t$-coloured subgraph $\Phi$ of $\Gamma$ spanned by the clusters contained in $V(F)$ is connected.
\end{lem}

\begin{proof}
Let $G'$ be the corresponding subgraph of $G$ obtained by applying Lemma \ref{lem_regularity}.
Let $U$ be the set of vertices of $F$ which belong to sets $V_i$ contained in $V(F)$.
Denote $F' = G'_t[U]$.
Recall that by Lemma \ref{lem_regularity}, we have that $|V(F) \setminus U|\le \eps n$ and $\deg_{G'}(v) \ge \deg_G(v) - 9\eps n$. In particular $\deg_{F'}(v)\ge \deg_F(v) - 9\eps n$ for every $v \in V(F')$.
Hence $F'$ is obtained by removing at most $\eps n$ vertices of $F$ and then removing a subgraph with maximum degree at most $9\eps n$.
By Lemmas \ref{lem_robust_removing_vertices}, \ref{lem_robust_removing_edges}, we have that $F'$ is $(\alpha/4, k)$-robust in $G'$.
In particular, $F'$ is connected and it follows that $\Phi$ (the $t$-coloured subgraph of $\Gamma$ spanned by clusters contained in $V(F)$) is connected.
\end{proof}

\subsection{From connected matchings to long cycles} \label{subsec_conn_matching}

We shall use the technique of converting connected matchings in the reduced graph into cycles in the original graph. This was introduced by {\L}uczak \cite{luczak}, and since then has become fairly standard (see \cite{balogh_et_al}, \cite{debiasio} and \cite{ruszinko_et_al_improved_c_r}, \cite{ruszinko_et_al}, \cite{gyarfas_sarkozy} and \cite{luczak_rodl_szemeredi}).

For the sake of completeness, we prove the following result, stating that given a connected matching in the reduced graph, there exists a cycle in the original graph through most of the vertices in the clusters and few additional vertices.

\begin{lem}\label{lem_connected_matching}
Let $\eps > 0$ and $d \ge 3\eps$ and let $n$ be suitably large. 
Let $G$ be a graph on $n$ vertices and let $\Gamma$ be the $(\eps, d)$-reduced graph obtained by an application of Lemma \ref{lem_regularity}.
Suppose that $\C{M}$ is a connected matching in $\Gamma$ and denote by $U$ the set of vertices spanned by the clusters of $\C{M}$.
Then $G$ contains a cycle covering at least $(1 - 6\eps)|U|$ of the vertices of $U$.
\end{lem}

Let us first sketch the proof.
Using the fact that $\C{M}$ is connected we can connect the matching edges by paths, following a cyclic ordering of the edges in $\C{M}$. We replace these paths by short vertex-disjoint paths in the original graph $G$ between the cluster pairs associated to the edges of $\C{M}$. These connecting paths will be parts of the final cycle. To define the rest of the cycle, remove the internal vertices of these connecting paths, and  in each cluster pair find a long path to close the connecting pairs to a cycle. There are various ways to find these long cycles, e.g.~using the rotation-extension technique of P\'osa \cite{posa} or as a very special case of the Blow-up Lemma \cite{blowup}. We shall use a much simpler result to obtain these cycles.

\begin{proof}[Proof of Lemma \ref{lem_connected_matching}]
Denote by $\{V_0, \ldots, V_m\}$ the corresponding partition of $V(G)$ and $\C{M} = {(V_{i_1}, V_{j_1}), \ldots, (V_{i_m}, V_{j_m})}$.
It is easy to see that we may find paths $P_1, \ldots, P_m$ with the following properties.
\begin{itemize}
\item
The paths $P_l$ are vertex disjoint.
\item
Each path $P_l$ contains at most one vertex from each cluster $V_i$.
\item
Denote the ends of $P_l$ by $x_l, y_l$.
Then
\begin{align*}
& |N(x_l) \cap V_{j_l}| \ge 2\eps |V_1| \\
& |N(y_l) \cap V_{i_{l + 1}}| \ge 2\eps |V_1|.
\end{align*}
\end{itemize}
Given such paths $P_1, \ldots, P_m$, denote $U = V(P_1) \cup \ldots \cup V(P_m)$.
Pick sets 
\begin{align*}
& A_l \subseteq (N(x_l) \cap V_{j_l}) \setminus U \\
& B_l \subseteq (N(y_l) \cap V_{i_{l + 1}}) \setminus U
\end{align*}
of size $\eps |V_1|$.
Note that this is indeed possible because by the choice of $x_l, y_l$ and since for large enough $n$, $|U| \le \eps |V_1|$.
Pick sets 
\begin{align*}
& C_l \subseteq N(x_l) \setminus (U \cup A_l) \\
& D_l \subseteq N(y_l) \setminus (U \cup B_l)
\end{align*}
of size $(1 - 2\eps)|V_1|$ each.
We shall show that for each $l$, $G[C_l, D_l]$ contains a path $Q_l$ missing at most $2\eps |V_1|$ vertices from each set $C_l, D_l$.
Assuming that we may find such paths, it is now easy to construct the required cycle.
By the definition of the reduced graph, the density of edges between subsets of $V_{i_l}$ and $V_{j_l}$ of size at least $\eps |V_1|$ is positive.
Thus there is an edge between $A_l$ and the last $2\eps |V_1|$ vertices of $Q_l$ as well as between $B_l$ and the first $2 \eps |V_1|$ vertices of $Q_{l + 1}$.
Thus by losing at most $4 \eps |V_1|$ vertices from each path $Q_l$ we may use the paths $P_l$ to obtain a cycle $C$ which misses at most $6 \eps |V_1|$ vertices from each cluster in $V(\C{M})$.

It remains to show that such paths $Q_l$ may be found.
We use the following simple but useful claim which was proved independently by Pokrovskiy \cite{pokrovskiy} and Dudek and Pra\l{}at \cite{dudek}.
For the sake of completeness, we include the proof here.

\begin{prop} \label{prop_partition_empty_bip_path}
For every graph $G$ there exist two disjoint sets of vertices $U, W$ of equal size such that $G$ has no edges between $U$ and $W$ and the graph $G \setminus (U \cup W)$ has a Hamilton path.
\end{prop}

\begin{proof}
In order to find sets with the desired properties, we apply the following algorithm, maintaining a partition of $V(G)$ into sets $U, W$ and a path $P$. Start with $U = V(G)$, $W = \emptyset$ and $P$ an empty path. At every stage in the algorithm, do the following. If $|U | \le |W|$, stop.
￼￼￼￼￼￼￼￼
Otherwise, if $P$ is empty, move a vertex from $U$ to $P$ (note that $U \neq \emptyset$). If $P$ is non-empty, let $v$ be its endpoint. If $v$ has a neighbour $u$ in $U$, put $u$ in $P$, otherwise move $v$ to $W$.
Note that at any given point in the algorithm there are no edges between $U$ and $W$. Furthermore, the value $|U| - |W|$ is positive at the beginning of the algorithm and decreases by one at every stage, thus at some point the algorithm will stop and will produce sets $U, W$ with the required properties.
\end{proof}

We need the following corollary of Proposition \ref{prop_partition_empty_bip_path}.
\begin{cor}\label{cor_partition_bip_empty_bip_graph}
Let $G$ be a balanced bipartite graph on $n$ vertices with bipartition $V_1, V_2$, which has no path of length $k$. Then there exist $X_i \subseteq V_i$ such that $|X_1| = |X_2| \ge (n - k)/4$ and $G$ has no edges between $X_1$ and $X_2$.
\end{cor}

\begin{proof}

Let $U, W$ be as in Proposition \ref{prop_partition_empty_bip_path} and let $P$ be a Hamilton path in $G \setminus (U \cup W)$. Note that $P$ must alternate between $V_1$ and $V_2$, thus $|V(P ) \cap V_1| = |V (P) \cap V_2|$ (by the assumptions, the number of vertices in $P$ is even). Denote $U_i = U \cap V_i$ and $W_i =W \cap V_i$ for $i = 1,2$, and assume that $|U_1| \ge |U_2|$. It follows that $|U_1| + |W_1| = |U_2| + |W_2|$. Thus, using the fact that $|U| = |W|$, we have $|U_1| = |W_2| \ge |U|/2 \ge (n - k)/4$. Set $X_1 = U_1$ and $X_2 = W_2$. 
\end{proof}

The graph $G[C_l, D_l]$ is a balanced bipartite graph, satisfying that for every choice of sets $C' \subseteq C_l, D' \subseteq D_l$ of size at least $\eps |V_1|$ each, we have $e(G[C', D']) > 0$.
It follows from Corollary \ref{cor_partition_bip_empty_bip_graph} that $G[C_l, D_l]$ contains a path missing at most $2\eps |V_1|$ vertices from each side.
As explained previously, Lemma \ref{lem_connected_matching} follows.

\end{proof}

In fact, we shall need a slight generalisation of Lemma \ref{lem_connected_matching}.
The usual setting in which we apply the described technique is as follows.
We consider the reduced graph obtained by applying the Regularity Lemma. In the reduced graph, we find a perfect matching consisting of a connected blue matching and a connected red matching, and we use the described technique to find two disjoint cycles, one blue and one red, which together cover almost all of the vertices.

To obtain disjoint cycles, we set aside small sets of each cluster of the red component which will be used to connect the cluster pairs of the red connected matching. We then find a blue cycle as above covering most of the clusters of the blue matching, and then we find a red cycle similarly, using the sets we set aside.

We shall not write the exact statement of the modified version of Lemma \ref{lem_connected_matching}, but let us remark that the constants change slightly. We need $d \ge 6\eps$ and are able to cover $(1 - 9\eps)n$ of the vertices of each cluster in the matchings.

Finally, we point out that we often first find ``absorbing paths'' in the original graph and then apply the Regularity Lemma to the remaining graph. When building the cycles obtained by the connected matchings, we would like them to contain these predefined paths. This is obtained by the same method, enabling us to find a path (rather than a cycle) between the neighbourhoods of the two ends of the absorbing path.

This concludes the introduction of the tools we shall need for our proof of Theorem \ref{thm_main}.
To complete the preliminary material needed for our proof, we list several extremal results in the next section.

\remove{
We shall need a more technical result.

\begin{framed}
COMPLETE THE STATEMENT OF THE LEMMA
\end{framed}
\begin{lem}
Let $\eps > 0$ and let $d \ge 6\eps$ and $n$ be suitably large. 
Let $G$ be a $2$-coloured graph on $n$ vertices, let $(V_0, \ldots, V_m)$ be the partition obtained by an application of Lemma \ref{lem_regularity} and let $\Gamma$ be the corresponding $(\eps, d)$-reduced graph.
Suppose that $\C{M}$ is a connected matching in $\Gamma$. 
Suppose that we have subsets $U_i \subseteq V_i$ such that 
\begin{align*}
|U_i| \ge
\left\{ 
\begin{array}{ll}
(1 - 3\eps)|V_1| & \text{if }V_i \in V(\C{M}) \\
3\eps |V_1| & \text{if } V_i \text{ is in the connected component of } \C{M} 
\end{array}
\right.
\end{align*}
Then $G$ has a cycle $C$ whose vertex set $W$ is contained in $U_1 \cup \ldots \cup U_m$ and satisfies $|W \cap U_i| \ge (1 - 6\eps)|V_1|$ for every $i$ such that $V_i \in V(\C{M})$.

Suppose that we are given in addition  sets $A, B \subseteq V(G)$ whose intersection with the set of vertices spanned by clusters in the connected component of $\C{M}$. Then $G$ has a path $P$ with one end in $A$ and the other in $B$, such that the set of inner vertices of $P$ satisfy the same conditions as the set $W$ above.
\end{lem}
}

\section{Extremal results}\label{sec_prelim_results}

In this section we list a number of extremal results we shall use in our proofs. They concern mainly with the existence of matchings, paths and cycles in graphs with certain structural conditions.

The following is Chv\'atal's theorem \cite{chvatal} giving sufficient conditions on the degree sequence of a graph for containing a Hamilton cycle.

\begin{thm} \label{thm_chvatal}
Let $G$ be a graph on $n \ge  3$ vertices and let $d_1 \le \ldots \le d_n$ be the degree sequence of $G$.
Suppose that $d_i \ge i + 1$ or $d_{n - i} \ge n - i$ for every $i \le n/2$.
Then $G$ contains a Hamilton cycle.
\end{thm}

In some cases it is easier to use the following version of Chv\'atal's result for bipartite graphs.

\begin{cor}\label{cor_chvatal_bip}
Let $G$ be a balanced bipartite graph on $2n$ vertices with bipartition $\{X, Y\}$.
Let $x_1 \le \ldots \le x_n$ be the degree sequence of $X$ and let $y_1 \le \ldots \le y_n$ be the degree sequence of $Y$.
Suppose that $x_i \ge i + 1$ or $y_{n - i} \ge n - i + 1$ for every $i \in [n]$.
Then $G$ contains a Hamilton cycle. 
\end{cor}

\begin{proof}
Consider the graph $G'$ obtained from $G$ by adding all edges with both ends in $X$. By Theorem \ref{thm_chvatal}, $G'$ contains a Hamilton cycle $C$. As $|X| = |Y|$, the cycle $C$ contains no edges with both ends in $X$, i.e.~$C$ is a Hamilton cycle in $G$.
\end{proof}

The following is a simple result by Erd\H{o}s and Gallai \cite{erdos_gallai}, giving an upper bound on the number of edges in a graph with no path of a given length.

\begin{thm} \label{thm_erdos_gallai}
Let $G$ be a graph on $n$ vertices with no paths of length at least $l + 1$, then $e(G) \le nl/2$.
\end{thm}

A graph on $n$ vertices is called \emph{pancyclic} if for every $l \le n$, $G$ contains a cycle on $l$ vertices. The following result by Bondy \cite{bondy} is a generalisation of Dirac's Theorem, asserting that graphs with large enough minimum degree are pancyclic.
\begin{thm} \label{thm_bondy}
Let $G$ be a graph on $n$ vertices with $\delta(G) > n/2$. Then $G$ is pancyclic.
\end{thm}

In the following two subsections, we use the following well known theorem of Tutte, giving a necessary and sufficient condition for having a perfect matching.
\begin{thm} \label{thm_tutte}
Let $G$ be a graph on an even number of vertices.
Then $G$ has a perfect matching if and only if for every set of vertices $U$, the number of odd components of $G\setminus U$ is at most $|U|$.
\end{thm}

\subsection{Matchings in tripartite graphs}
In section \ref{sec_large_blue_non_bip} we shall analyse conditions for certain tripartite graph to have a perfect matching.
Here we describe the extremal results we shall need for the analysis.
We use a stability version of the following lemma of DeBiasio and Nelsen \cite{debiasio}.
For the sake of completeness, we prove it here.

\begin{lem}\label{lem_trip_pft_matching_exact}
Let $n$ be even, and let $G$ be a tripartite graph on $n$ vertices with tripartition $\{X_1, X_2, X_3\}$.
Suppose that $|X_i| \le n/2$ and $\deg(x) >3n/4 - |X_i|$ for every $x \in X_i$, $i \in [3]$.
Then $G$ has a perfect matching.
\end{lem}

\begin{proof}
Suppose to the contrary that $G$ has no perfect matching.
Without loss of generality, we assume that $G$ is a maximal counter example.
It is easy to check that the graph obtained from the complete tripartite graph with partition $\{X_1,X_2,X_3\}$ by removing any edge has a perfect matching.
Thus we may assume that $G$ is not complete as a tripartite graph. Without loss of generality, suppose that $v_1v_2\notin E(G)$ where $v_1\in X_1,v_2\in X_2$.
It follows from the maximality of $G$ that the addition of $v_1v_2$ to $G$ completes a perfect matching $M$ in $G$.
Clearly, $v_1v_2\in M$.

$M$ contains no edge $u_1u_2\neq v_1v_2$ where $u_i\in N(v_i)$ for $i=1,2$, otherwise there exists a perfect matching in $G$ (replace $v_1v_2,u_1u_2$ by $v_1u_1,v_2u_2$).
Denote by $M_1$ the edges of $M$ which have ends in $X_1$ and $X_2$, and let $M_2=M\setminus M_1$.
It follows that for every $e\in M_1\setminus\{v_1v_2\}$, $\deg(v_1,e)+\deg(v_2,e)\le 1$ and for every $e\in M_2$, $\deg(v_1,e)+\deg(v_2,e)\le 2$.

Thus 
\begin{align*}
& \deg(v_1)+\deg(v_2)=\sum_{e\in M}(\deg(v_1,e)+\deg(v_2,e))\le \\
& (|M_1|-1)+2|M_2|=|M|-1+|M_2|=n/2-1+|X_3|.
\end{align*}
For the last equality, we used the fact that $|M_2|=|X_3|$.
By the degree condition on $G$, we have $\deg(v_1) + \deg(v_2) > 3n/2 - |X_1| - |X_2| = n/2 + |X_3|$, a contradiction to the above inequality.
\end{proof}

The following lemma is a stability version of Lemma \ref{lem_trip_pft_matching_exact}.
We prove it by applying Tutte's theore, Theorem \ref{thm_tutte}.
\begin{lem}\label{lem_matching_tripartite_density}
Let $\frac{1}{n} \ll \eps \ll 1$, where $n$ is even.
Let $G$ be a graph on $n$ vertices with tripartition $\{X_1,X_2,X_3\}$ such that $|X_i|\le (1/2-4\eps)n$ and $\deg(x, V(G) \setminus X_i)\ge (3/4 - \eps)n - |X_i|$ for $i \in [3]$ and $x\in X_i$.
Then one of the following holds.
\begin{itemize}
\item
$G$ has a perfect matching.
\item
There exists an independent set $Y$ in $G$ such that $Y \subseteq X_i \cup X_j$ and $|Y \cap X_i|, |Y \cap X_j| \ge (1/4 - 5\eps)n$ for some $1 \le i < j \le 3$.
\end{itemize}
\end{lem}

\begin{proof}
We assume that $G$ has no perfect matching.
By Tutte's theorem, Theorem \ref{thm_tutte}, there exists a subset $S\subseteq V(G)$ such that the number of odd components of $G\setminus S$ is larger than $S$.

Denote $G'=G \setminus S$.
We show first that $\delta(G') \le \eps n$.
Suppose not. Then the number of components of $G'$ is at most $1 / \eps$, thus $|S| \le 1/ \eps$. We show that $G'$ is connected, contradicting the choice of $S$. Given $u, v \in X_i$, they are non-adjacent to at most $(1/4 + \eps)n$ vertices in $V(G)\setminus X_i$. But $|V(G) \setminus X_i| \ge (1/2 + 4\eps)$, hence $u, v$ have at least $2\eps n \ge  1/\eps$ common neighbours. It follows that indeed, $G'$ is connected.

By the assumptions on $G$, $\delta(G) \ge (3/4 - \eps)n - \max\{|X_1|,|X_2|,|X_3|\}\ge (1/4 + 3\eps)n$. Since $\delta(G')\le \eps n$, we have $|S| \ge (1/4 + 2\eps)n$.

Note that $|S| \le n/2$, because the number of odd component of $G'$ is at most $n - |S|$.
Denote $X_i' = X_i \setminus S$.
We show that $|X_i'| \le (1/4 + 2\eps)$ for $i \in [3]$.
Indeed, suppose that $|X_1'| \ge (1/4 + 2\eps)$. Then for every vertex $u \in  X_2' \cup X_3'$ we have $\deg_{G'}(u) \ge \eps n$. In particular, $u$ is in a component of $G'$ order at least $\eps n$. Furthermore, every non isolated vertex of $X_1'$ is adjacent to some vertex in $X_2'\cup X_3'$, and thus is in a  component of size at least $\eps n$.
Since $X_2' \cup X_3'$ is non empty (e.g.~because $|S| \le n/2$), it follows that $G'$ has at most $(1/4 + \eps)n$ isolated vertices, and the rest of the vertices are in components of order at least $\eps n$.
Hence the number of odd components of $G'$ is at most $(1/4 + \eps)n + 1 / \eps\le |S|$, a contradiction.

Denote $S_i = X_i \cap S$, $i \in [3]$.
Consider the three quantities $|X_1| + |S_2| + |S_3|, |X_2| + |S_1| + |S_3|, |X_3| + |S_1| + |S_2|$.
Their sum is $n + 2|S| \le 2n$.
Without loss of generality, it follows that $|X_1| + |S_2| + |S_3| \le 2n/3$.
Thus for every $u \in X_1'$, 
$$\deg_{G'}(u) \ge (3/4 - \eps)n - |X_1| - |S_2| - |S_3|\ge (1/12 - \eps)n.$$
If in addition $|X_1'| + |X_2'| \ge (1/4 + 2\eps)n$, every vertex in $X_3'$ is in a component of order at least $\eps n$.
But the number of isolate vertices in $X_2'$ s at most $(1/4 + \eps)n$, so the number of odd components in $G'$ is at most $(1/4 + \eps)n + 1 / \eps \le |S|$.
We conclude that $|X_1'| + |X_2'| \le (1/4 + 2\eps)n$.
Recall that $|X_3'| \le (1/4 + 2\eps)n$, implying that $|V(G')| \le (1/2 + 4\eps)n$, i.e.~$|S| \ge (1/2 - 4\eps)n$.
If $X_1' \neq \emptyset$, $G'$ contains a component of order at least $(1/12 - \eps)n$ in $G'$, so there are at most $(5/12 + 5\eps)n + 1 \le |S|$ components in $G'$, a contradiction.

The set $Y$ obtained by picking one vertex from each component of $G'$ is an independent set in $X_2' \cup X_3'$ of size at least $|S|$.
Let $Y_i = Y \cap X_i'$.
Then $G$ contains no $Y_2 - Y_3$. Recall that each vertex in $Y_2$ has at most $(1/4 + \eps)n$ non neighbours in $Y_3$, implying that $|Y_3| \le (1/4 + \eps)n$ and $|Y_2| \ge (1/4 - 5\eps)n$. By symmetry, we also have $|Y_3| \ge (1/4 - 5\eps)n$.
\end{proof}

\subsection{Matchings in bipartite graphs}

The following lemma is a simple consequence of Hall's theorem for perfect matchings in bipartite graphs.
\begin{lem} \label{lem_hall}
Let $G$ be a balanced bipartite graph with bipartition $\{X_1,X_2\}$ on $n$ vertices with $\delta(G)\ge (1/4 - \eps)n$. Then one of the following conditions hold.
\begin{itemize}
\item
$G$ contains a perfect matching.
\item
There are subsets $A_i\subseteq X_i$ of size $(1/4 - \eps) n \le |A_1|, |A_2|\le (1/4 + \eps)n$ such that $G[A_1, A_2]$ contains no edges.
\end{itemize}
\end{lem}

\begin{proof}
Suppose that $G$ contains no perfect matching.
Then by Hall's Theorem there exists $A_1\subseteq X_1$ such that $|N(A_1)|< |A_1|$.
Denote $A_2=X_2\setminus N(A_1)$.
Then $G[A_1,A_2]$ has no edges.
Note that since $\delta(G)\ge (1/4 - \eps)n$, we have $|A_1|> |N(A_1)| \ge(1/4 - \eps)n$.
By symmetry we have that $|A_2|\ge (1/4 - \eps)n$.
Thus $(1/4 - \eps) n \le |A_1|=|A_2|\le (1/4 + \eps)n$.

\end{proof}

In Section \ref{sec_large_blue_non_bip}, we shall also need the following stability result for graphs with a bipartition satisfying certain conditions. The proof is again an applications of Tutte's theorem, Theorem \ref{thm_tutte}.
\begin{lem}\label{lem_matching_bip_technical}
Let $\frac{1}{n} \ll \eps \ll 1$, where $n$ is even.
Let $G$ be a graph on $n$ vertices, and suppose that $\{X_1, X_2\}$ is a partition of $V(G)$ such that $|X_i| \ge (1/2 - \eps)n$ and $\deg(u, X_{3 - i})\ge (3/4-\eps)n - |X_i|$ for every $i \in [3]$, $u \in X_i$.
Then one of the following conditions holds.
\begin{itemize}
\item
$G$ has a perfect matching.
\item
$G$ is not $2$-connected.
\item
There exists an independent set $A_i \subseteq X_i$ (for some $i \in [2]$) of order at least $(1/4 - 4\eps)m$ such that $|N(A_i)| \le (1/4 + 3\eps)n$.
\item
$|X_i| > |X_{3 - i}|$ and $X_i$ contains an independent set of size at least $(1/2 - \eps)n$.
\item
There exists an independent set $A$ of size at least $(1/2 - 6\eps)n$ such that $|A \cap X_i| \ge (1/4 - 9\eps)n$. 
\end{itemize}
\end{lem}

\begin{proof}
Suppose that $G$ has no perfect matching.
It follows from Tutte's theorem, Theorem \ref{thm_tutte}, that there exists a set $S$ such that the number of odd components in $G' = G \setminus S$ is larger than $S$.

Denote $S_i = S \cap X_i$ and $X_i' = X_i \setminus S$.
If $\delta(G') \ge \eps n$, the number of components of $G'$ is at most $1/\eps$, implying that $|S| \le 1/\eps$.
In this case, the conditions on the graph $G[X_1, X_2]$ imply that either $G'$ is connected, contradiction our assumption on $S$, or $G'$ consists of two connected  components, implying that $|S| \le 1$, so $G$ is not $2$-connected.

We now assume that $\delta(G') \le \eps n$.
It follows that $|S| \ge (1/4 - 2\eps)n$.
Suppose that $|X_1'| \ge (1/4 + 3\eps)n$ and $X_2' \neq \emptyset$.
Denote by $A_1$ the set of isolated vertices in $X_1'$.
Then every vertex in $X_2' \cup (X_1' \setminus A_1)$ belongs to a component of $G'$ of size at least $\eps n$. Furthermore, $|A_1| \le (1/4 + 2\eps)n$.
It follows that the number of components is at most $|A_1| + 1/\eps \le (1/4 + 3\eps)m$, implying that $|S| \le (1/4 + 3\eps)m$.
But $|S| \ge (1/4 - 2\eps)n$, so $|A_1| \ge |S| - 1/\eps \ge (1/4 - 3\eps)n$.
Clearly, $N(A_1) \subseteq S$, and the third conditions of the lemma holds.
The case $|X_2'| \ge (1/4 + 3\eps)n$ and $X_1' \neq \emptyset$ follows similarly.

Suppose now that $X_2' = \emptyset$.
Note that $|S| < n/2$ because the number of component in $G'$ is at most $n - |S|$.
By our assumption, $X_2 \subseteq S$, so $(1/2 - \eps)n \le |X_2'| \le |S| \le n/2$.
Thus $G[X_1']$ consists of at least $(1/2 - \eps)n$ components, in particular it contains an independent set of size at least $(1/2 - \eps)n$, and the fourth condition holds.

Finally, we assume that $X_1', X_2'$ are non-empty and $|X_1'|, |X_2'| \le (1/4 + 3\eps)n$.
It follows that $|S| \ge (1/2 - 6\eps)n$ and $G'$ contains at least $|S|$ odd components, in particular $X_1' \cup X_2'$ contains an independent set $A$ of size at least $(1/2 - 6\eps)n$. Denote $A_i = X_i' \cap A$. Then $|A_i| \ge (1/4 - 9\eps)n$.

\end{proof}

\subsection{Hamilton cycles in bipartite graphs}
The following result is a stability version of a special case of Chv\'atal's theorem, Theorem \ref{thm_chvatal}, which we shall use in Section \ref{sec_red_graph_disconnected}, for the proof of Lemma \ref{lem_red_graph_disconnected_1} below.

\begin{lem}\label{lem_stability_dirac}
Let $\frac{1}{n} \ll \eps \ll 1$ and let $G$ be a balanced bipartite graph on $n$ vertices with bipartition $\{X_1, X_2\}$.
Suppose that $\delta(G)\ge (1/4 - \eps)n$ and between every two subsets of $X_1$ and $X_2$ of size at least $(1/4 - 3\eps)n$ there are at least $\eps n^2$ edges.
Then $G$ is Hamiltonian. Furthermore, there is a Hamilton path between every pair of points $x_1,x_2$ where $x_1 \in X_1$ and $x_2 \in X_2$.
\end{lem}

We prove this result by a relatively simple application of the absorbing method of of R\"odl, Ruci\'nski and Szemer\'edi \cite{luczak_rodl_szemeredi}.
In fact, all we need for this proof is Lemma \ref{lem_absorbing_paths}, asserting the existence of short absorbing paths in robust subgraphs, and the Regularity Lemma.
A graph $G$ as in Lemma \ref{lem_stability_dirac} is $(\eps, 2)$-weakly robust, thus by Lemma \ref{lem_absorbing_paths}, it is possible to find an absorbing path $P$ in $G$. 
We consider the reduced graph $\Gamma$, obtained from applying the Regularity Lemma on the graph $G \setminus V(P)$.
We deduce from the conditions of the lemma that $\Gamma$ has an almost perfect matching, implying that $G$ contains a cycle extending $P$ and spanning almost all vertices. The remaining vertices may be absorbed by $P$.

\begin{proof}[Proof of Lemma \ref{lem_stability_dirac}]

It is easy to check from the conditions that $G$ is $(\eps, 2)$ weakly robust with bipartition $\{X_1, X_2\}$.
Indeed, let $x_1 \in X_1, x_2 \in X_2$.
Recall that $|\con_{G,2}(x_1, x_2)|$ is the number of paths of length three in $G$ between $x_1$ and $x_2$. Then $|\con_{G,2}(x_1,x_2)|$ is the number of edges between $N(x_1) \setminus \{x_2\}$ and $N(x_2) \setminus \{x_1\}$, which, by the assumptions in the lemma is at least $\eps n^2$.

It follows from Lemma \ref{lem_absorbing_paths}, that for $\rho > 0$ is small enough, there exists a $\rho^2 n$-absorbing path $P$ in $G$ of length at most $\rho n$.
Namely, if $W\subseteq V(G) \setminus V(P)$ is such that $|W \cap X_1| = |W \cap X_2| \le \rho^2 n$, then $G[V(P) \cup W]$ contains a Hamilton path with the same ends as $P$.
Note that we may assume for convenience that $P$ has one end in $X_1$ and the other in $X_2$.

Pick $\eta > 0$ suitably small.
Apply the regularity lemma, Lemma \ref{lem_regularity}, with the graph $G \setminus V(P)$, the bipartition $\{X_1', X_2'\}$, where $X_i' = X \setminus V(P)$ and parameter $\eta$.
Let $G'$ be the subgraph of $G$ promised by Lemma \ref{lem_regularity} and let  $\{V_0, \ldots, V_m \}$ be the given partition.
Denote by $m_i$ the number of parts $V_j$ (where $j \ge 1$) which are contained in $X_i$ (so $m = m_1 + m_2$).
Without loss of generality, we assume that $m_1 \ge m_2$.
\begin{claim}
$m_1 - m_2 \le 2\eta m$.
\end{claim}
\begin{proof}
Denote by $n' = |V(G')| = n - |P|$.
Recall that $|V_0| \le \eta n'$, and $|V_1| = \ldots = |V_m| \ge \frac{(1 - \eta) n'}{m}$.
Since $|X_1'| = |X_2'|$, we obtain $(m_1 - m_2) |V_1| \le |V_0| \le \eta n'$.
It follows that $m_1 - m_2 \le \frac{\eta n'}{|V_1|} \le \frac{\eta m}{1 - \eta} \le 2\eta m$.
\end{proof}

Let $\Gamma$ be the $(\eta, 4\eta)$-reduced graph defined by the $m_2$ clusters $V_i$ contained in $X_2$ and by some $m_2$ clusters contained in $X_1$.
Denote by $\{Y_1, Y_2\}$ the bipartition of $\Gamma$.
\begin{claim}
$\Gamma$ is a balanced bipartite graph on $m' = 2m_2$ vertices with
$\delta(\Gamma) \ge (1/4 - 2\eps) m'$.
Furthermore, for every choice of subsets $A_i\subseteq Y_i$, where $|A_1|, |A_2| \ge (1/4 - 2\eps)m$, we have $e(\Gamma[A_1, A_2]) \ge \frac{\eps}{2} (m')^2$.
\end{claim}
\begin{proof}
Both parts of the claim follow from Condition \ref{itm_reg_deg} in Lemma \ref{lem_regularity} and the definition of $\Gamma$ as long as $\eta$ and $\rho$ are small enough.
\end{proof}

It follows from Lemma \ref{lem_hall}, that $\Gamma$ contains a perfect matching, which is a connected matching as $\Gamma$ is connected.
By Lemma \ref{lem_connected_matching}, we obtain a cycle $C$ in $G$, containing the path $P$ and spanning all but at most $6\eta n \le \rho^2 n$ vertices.
Denote $W = V(G) \setminus V(C)$.
Since $G$ is a balanced bipartite graph, we have $|W \cap X_1| = |W \cap X_2| \le \rho^2 n$.
It follows from the absorbing property of $P$ that the vertices of $W$ may be absorbed into $P$ and thus into $C$ to obtain a Hamilton cycle.

It is easy to modify the proof to obtain a Hamilton path between any given vertices $x_i \in X_i$. 

\end{proof}

\subsection{Monochromatic cycle partitions in $2$-coloured graphs with the red graph almost disconnected}

The following two lemmas, Lemmas \ref{lem_red_graph_disconnected_1} and \ref{lem_red_disconnected_2}, state that a $2$-coloured graph $G$ has the desired partition into a red cycle and a blue one if $G$ admits some restrictive structural property.
We shall use these results several times in the following sections and delay their proofs to the ends of the paper.
We prove Lemma \ref{lem_red_graph_disconnected_1} in Section \ref{sec_red_graph_disconnected} and Lemma \ref{lem_red_disconnected_2} in Section \ref{sec_red_disconnected_2}.
\begin{lem} \label{lem_red_graph_disconnected_1}
Let $\frac{1}{n} \ll \eps \ll 1$ and let $G$ be a graph on $n$ vertices with $\delta(G)\ge 3n/4$ and a $2$-colouring $E(G) = E(G_B) \cup E(G_R)$.
Suppose that $S, T \subseteq V(G)$ satisfy the following conditions.
\begin{itemize}
\item
$S,T$ are disjoint and $|S|, |T| \ge (1/2-\eps)n$.
\item
$\delta(G_B[S,T]) \ge (1/4 - \eps) n$.
\item
For every $S'\subseteq S, T' \subseteq T$ with $|S'|, |T'| \ge (1/4 - 100\eps) n$, we have $e(G_B[S', T']) \ge 25 \eps n^2$.
\end{itemize}
Then $V(G)$ may be partitioned into a red cycle and a blue one.
\end{lem}

\begin{lem}\label{lem_red_disconnected_2}
Let $\frac{1}{n} \ll \eps \ll 1$ and let $G$ be a graph on $n$ vertices with $\delta(G)\ge 3n/4$ and a $2$-colouring $E(G) = E(G_B) \cup E(G_R)$.
Suppose that there exists a partition $\{S, T, X\}$ of $V(G)$ with the following properties.
\begin{itemize}
\item
$|S|, |T| \ge (1/2 - \eps)n$.
\item
$|X| \le 2$ and if $|X| = 2$, there exists $u \in X$ such that $\deg_R(x, S) \le \eps n$ or $\deg_R(x, T) \le \eps n$.
\item
The sets $S$ and $T$ belong to different components of $G_R \setminus X$.
\end{itemize}
Then $V(G)$ may be partitioned into a red cycle and a blue one.
\end{lem}

This concludes the preliminary material needed for the proof of Theorem \ref{thm_main}.
We are now finally ready to turn to the heart of the proof.

\section{Rough structure}\label{sec_rough_struct}
In this section we make the first step towards our proof of Theorem \ref{thm_main}. 
We use the Regularity Lemma, Lemma \ref{lem_regularity}, to obtain information about the rough structure.

\begin{lem}\label{lem_rough_structure}
Let $\frac{1}{n} \ll \alpha, \frac{1}{k} \ll \eps \ll 1$ and let $G$ be a graph with $\delta(G)\ge 3n/4$.
Let $E(G) = E(G_B) \cup E(G_R)$ be a $2$-colouring of $G$.
Then one of the following assertions holds, where a robust component refers to an $(\alpha, k)$-robust component, possibly with the roles of red and blue reversed.

\begin{enumerate}
\item\label{itm_struct_spanning_non_bip}
There exists a monochromatic strongly robust blue component on at least $(1-\eps)n$ vertices.
\item\label{itm_struct_spanning_blue_bip_large_red}
There exists a weakly robust blue component of order at least $(1 - \eps / 4)n$ and a red strongly robust component of order at least $(1/2 + \eps / 2) n$.
\item\label{itm_struct_spanning_blue_bip_almost_balanced}
There exists a weakly robust blue component with bipartition $\{X_1,X_2\}$ where $|X_1|,|X_2|\ge (1/2-\eps)n$ and for each $i \in [2]$, $e(G_B[X_i]) \le \eps n^2$ and one of the following holds.

\begin{enumerate}
\item \label{itm_struct_blue_bip_balanced_red_strongly_robust}
$G_R[X_i]$ is strongly robust.

\item \label{itm_struct_blue_bip_balanced_red_weakly_robust}
$G_R[X_i]$ is weakly robust with bipartition $\{Y_{i, 1}, Y_{i, 2}\}$ satisfying $|Y_{i, j}| \ge (1/4 - \eps)n$ and $e(G_R[Y_{i, j}]) \le \eps n^2$.

\item \label{itm_struct_blue_bip_balanced_two_red}
There exists a partition $\{Y_{i, 1}, Y_{i, 2}\}$ of $X_i$ such that $|Y_{i, j}| \ge (1/4 - \eps)n$, $G_R[Y_{i, j}]$ is strongly robust and $e(G_R[Y_{i, 1}, Y_{i, 2}]) \le \eps n^2$.
\end{enumerate}

Furthermore, if Condition (\ref{itm_struct_blue_bip_balanced_red_weakly_robust}) holds for $i \in [2]$, then in addition $e(G_R[Y_{1, j}, Y_{2, j}]) \le \eps n^2$ for $j \in [2]$.

\item\label{itm_struct_two_largish_compts}
There exist a blue strongly robust component and a red robust component, each has order at least $(3/4 - \eps)n$ and together the span all but at most $\eps n$ of the vertices.

\item\label{itm_struch_only_midsized_compts}
There exist sets $X_1, X_2, Y_1, Y_2$ of order at least $(1/2 - \eps)n$ such that
\begin{itemize}
\renewcommand\labelitemi{--}
\item
$X_1, X_2$ are disjoint, $Y_1, Y_2$ are disjoint and $X_1 \cup X_2 = Y_1 \cup Y_2$.
\item
$|X_i \cap Y_j| \ge (1/4 - \eps)n$ for $i, j \in [2]$.
\item
$G_B[X_i]$ is robust and $G_R[Y_i]$ is strongly robust.
\end{itemize}
\end{enumerate}
\end{lem}

We remark that in light of the variety of extremal examples for Theorem \ref{thm_main} (see Section \ref{sec_extremal_examples}), it should not be surprising that there is a large number of cases to consider for the rough structure.
Furthermore, it is perhaps useful to note that many of the above cases describe the structure of the extremal examples we gave in Section \ref{sec_extremal_examples}.
For example, the left-hand graph in Figure \ref{figure_sharpness_1} corresponds to Case \ref{itm_struch_only_midsized_compts}, Figures \ref{figure_sharpness_2} and \ref{figure_sharpness_3} correspond to Case \ref{itm_struct_spanning_blue_bip_almost_balanced} and Figure \ref{figure_sharpness_4} corresponds to Case \ref{itm_struct_two_largish_compts}.

This lemma, technical as it seems, is a simple application of Lemmas \ref{lem_reduced_components} and \ref{lem_many_reduced_components}, which imply that monochromatic components in the reduced graph correspond to robust components in the original graph.
Before proving Lemma \ref{lem_rough_structure}, we give a brief overview of the proof.
After applying the Regularity Lemma with suitable parameters, we obtain a reduced graph $\Gamma$, which has minimum degree close to $3m / 4$ where $m = |\Gamma|$.
It is a routine check to verify that either there is a spanning monochromatic component, or there are two monochromatic components of size almost $3m / 4$ spanning the whole vertex set, or for each colour there are two almost half-sized components spanning the whole vertex set.
In the case where there is a spanning monochromatic component, further analysis is needed to show that one of the cases (\ref{itm_struct_spanning_non_bip}, \ref{itm_struct_spanning_blue_bip_large_red}, \ref{itm_struct_spanning_blue_bip_almost_balanced}) holds.

\begin{proof}
Set $\eta = 48\eps$.
Let $\Gamma$ be a $(\eta, 6\eta)$-reduced graph obtained by applying Lemma \ref{lem_regularity} to the graph $G$.
Note that $\delta(\Gamma)\ge (3/4 - 13\eta)m$, where $m = |\Gamma|$.
Without loss of generality, we assume that the largest monochromatic component is blue and denote it by $\Phi_1$.

Suppose first that $\Phi_1$ is a spanning subgraph of $\Gamma$.
If it is in addition non-bipartite, by Lemma \ref{lem_reduced_components}, there is a strongly robust blue component $F_1$ of order at least $(1 - 2\eta)n$, as in (\ref{itm_struct_spanning_non_bip}).

Thus we assume that $\Phi_1$ is bipartite, with bipartition $\{X_1, X_2\}$ where $|X_1| \ge |X_2|$.
If $|X_1| > (1/2 + 26\eta)m$, then $\delta(\Gamma[X_1])\ge (3/4 - 13\eta)m- |X_2| > |X_1|/2$.
It follows that $\Gamma[X_1]$ is a red non bipartite component.
By Lemma \ref{lem_many_reduced_components}, we obtain a weakly robust blue component $F_1$ on at least $(1 - 3\eta)n$ vertices and a red strongly robust component $F_2$ on at least $(1/2 + 23\eta)n$ vertices, as in (\ref{itm_struct_spanning_blue_bip_large_red}).

We assume now that $|X_1| \le (1/2 + 26\eta)m$. Then $(1/2 - 26\eta)m \le |X_1|,|X_2| \le (1/2 + 26\eta)m$.
Denote $\Gamma_i = \Gamma[X_i]$ for $i \in [2]$.
$\Gamma_i$ contains only red edges and $\delta(\Gamma_i)\ge (1/4 - 39\eta)m$.
Then one of the following holds for $i \in [2]$.
\begin{enumerate}
\item\label{itm_3_lem_struct}
$\Gamma_i$ is connected in red and non-bipartite.
\item\label{itm_2_lem_struct}
$\Gamma_i$ is connected and bipartite. Furthermore, it has minimum degree at least $(1/4 - 39\eta)m$.
\item\label{itm_1_lem_struct}
$\Gamma_i$ consists of two red components, each of order at least $(1/4 - 39\eta)m$.
\end{enumerate}

By the definition of a reduced graph and our choice of parameters, the number of blue edges of $G$ which are not present in $G'$ is at most $14 \eta n^2$ and similarly for the red edges.
It follows from Lemma \ref{lem_many_reduced_components} that one of the conditions in (\ref{itm_struct_spanning_blue_bip_almost_balanced}) holds.
Suppose that Condition (\ref{itm_2_lem_struct}) holds for both $\Gamma_1$ and $\Gamma_2$.
Denote the bipartition of $\Gamma_i$ by $\{Y_{i, 1}, Y_{i, 2}\}$.
If there are no edges between $Y_{1, j}$ and $Y_{2, j}$ for $j = 1, 2$, Condition (\ref{itm_struct_spanning_blue_bip_almost_balanced}) us satisfied.
Otherwise, without loss of generality, there is an edge between $Y_{1, 1}$ and $Y_{2, 1}$. Then the red graph $\Gamma_B$ is connected. But we assumed that there $\Gamma$ has no spanning non-bipartite monochromatic component, so there are no edges between $Y_{i, 1}$ and $Y_{3 - i, 2}$ for $i \in [2]$. Thus, up to relabelling of $Y_{i, j}$, Condition (\ref{itm_struct_spanning_blue_bip_almost_balanced}) holds.

We assume now that $\Phi_1$ does not span $\Gamma$.
Denote $s = |\Phi_1|$.
Suppose first that $s > (1/2 + 26\eta)m$.
Denote $U = V(\Gamma) \setminus V(\Phi_1)$.
Then every vertex $u \in U$, is incident to at least $(3/4 - 13\eta)m - (m - s) > s/2$ neighbours.
It follows that every two vertices outside of $V(\Phi_1)$ have a common red neighbour, implying that $U$ is contained in a red component $\Phi_2$ of order at least $(3/4 - 13\eta)m$.
Indeed, pick $u \in U$. Then the red neighbourhood of $u$ in $\Gamma$ is contained in $\Phi_2$ as well as $U$.

By the choice of $\Phi_1$ as the largest connected monochromatic subgraph, we have that $|\Phi_1|\ge (3/4-13\eta)m$.
The components $\Phi_1,\Phi_2$ cover $\Gamma$ and intersect in at least $(1/2 - 26\eta)n$ vertices.
By the following claim, at least one of $\Phi_1,\Phi_2$ is non-bipartite, implying that condition (\ref{itm_struct_two_largish_compts}) holds, by Lemma \ref{lem_many_reduced_components}.
\begin{claim}
At least one of the graphs $\Phi_1, \Phi_2$ is not bipartite.
\end{claim}
\begin{proof}
Suppose otherwise.
Denote $X_i = V(\Phi_i) \setminus V(\Phi_{3 - i})$ for $i \in [2]$ and $Y = V(\Phi_1) \cap V(\Phi_2)$.
Then for every $x_1 \in X_1,x_2 \in X_2$, the vertices $x_1, x_2$ are non-adjacent in $\Gamma$.
Thus, $x_1$ sends at least $(3/4 - 13\eta)m - |X_1|$ (blue) edges to $Y$ and similarly $x_2$ sends at least $(3/4 - 13\eta)m - |X_2|$ red edges into $Y$.
If both $\Phi_1$ and $\Phi_2$ are bipartite, it follows that $Y$ contains a set $A_1$ of $(3/4 - 13\eta)m - |X_1|$ vertices spanning no blue edges and a set $A_2$ of $(3/4 - 13\eta)m - |X_2|$ spanning no red edges. It follows that $Y$ contains an independent set of size at least $|A_1| + |A_2| - |Y| \ge (3/2 - 26\eta)m - (|X_1| + |X_2| + |Y|) \ge (1/2 - 26\eta)m$.
This is a contradiction to the minimum degree condition on $\Gamma$.
\end{proof}

It remains to consider the case where $s = |\Phi_1| \le (1/2 + 26\eta)m$.
An argument similar to a previous one shows that if $s < (1/2 - 26\eta)m$, every two vertices of $\Phi_1$ have a common red neighbour outside of $\Phi_1$, contradicting the choice of $\Phi_1$ as the largest monochromatic component.
Thus we have that $(1/2 - 26\eta)m \le s\le (1/2 + 26\eta)m$.
Note that we may find $u_1, u_2\in V(\Phi_1)$ which have no common red neighbour outside of $\Phi_1$ (otherwise there is a red component of order larger than $|\Phi_1|$ contradicting our choice of $\Phi_1$).
Denote by $X_i$ the set of red neighbours of $u_i$ outside of $V(\Phi_1)$.
Let $Y_i$ be the red neighbourhood of $X_i$ in $\Phi_1$.
It follows from the minimum degree condition and the order of $\Phi_1$ that $|X_i|,|Y_i|\ge (1/4 - 39\eta)m$.
Furthermore, the sets $X_1,X_2,Y_1,Y_2$ are disjoint, there are no red edges in between $X_1 \cup Y_1$ and $X_2 \cup Y_2$ and no blue edges between $X_1\cup X_2$ and $Y_1\cup Y_2$. In particular, there are no edges between $X_i$ and $Y_{3-i}$. Considering the minimum degree conditions and the size of the various sets, it follows that the blue subgraph $\Phi_2 = \Gamma_B \setminus V(\Phi_1)$ is connected.
Similarly, $\Gamma_R[X_i \cup Y_i]$ is connected. Moreover, of the four components, there cannot be both a red and a blue bipartite component. Condition (\ref{itm_struch_only_midsized_compts}) follows.

\end{proof}

\subsection*{Proof of the main theorem}

We now prove Theorem \ref{thm_main}, using Lemma \ref{lem_rough_structure} and other results which we shall state and prove in subsequent sections.

\begin{thm*}[\ref{thm_main}]
There exists $n_0$ such that if a graph on $n \ge n_0$ vertices and minimum degree at least $3n/4$ is $2$-coloured then its vertex set may be partitioned into two monochromatic cycles of different colours.
\end{thm*}

\begin{proof} [Proof of Theorem \ref{thm_main}]
Let $\frac{1}{n} \ll \eps \ll 1$ and let $G$ be a graph on $n \ge n_0$ vertices with minimum degree at least $3n / 4$ and a red and blue colouring of the edges.
By Lemma \ref{lem_rough_structure}, we may assume that one of the Cases (\ref{itm_struct_spanning_non_bip} - \ref{itm_struch_only_midsized_compts}) from the statement of the lemma hold.
It remains to conclude that in each of these cases, we may find a partition of $V(G)$ into a red cycle and a blue one.
We prove this for each of the above cases using lemmas appearing in Sections (\ref{sec_large_bip_blue_large_red} - \ref{sec_large_blue_non_bip}).

We start by resolving Case (\ref{itm_struct_spanning_blue_bip_large_red}), which is perhaps the easiest to deal with, in Lemma \ref{lem_large_blue_bip_large_red} in Section \ref{sec_large_bip_blue_large_red}.

Case (\ref{itm_struct_spanning_blue_bip_almost_balanced}) is dealt with in Sections \ref{sec_large_bip_blue_almost_balanced}, \ref{sec_large_blue_bip_two_midsized_red}.
Lemmas \ref{lem_blue_bip_balanced_two_red_not_both_bip}, \ref{lem_large_blue_bip_three_red} and \ref{lem_blue_bal_bip_four_red} share some similarities and are used to prove all possible combinations of Cases (\ref{itm_struct_blue_bip_balanced_red_strongly_robust}, \ref{itm_struct_blue_bip_balanced_red_weakly_robust}, \ref{itm_struct_blue_bip_balanced_two_red}) except for the case where Condition (\ref{itm_struct_blue_bip_balanced_red_weakly_robust}) holds for both graphs in question.
The proof of Theorem \ref{thm_main} in the latter case can be deduced from Lemma \ref{lem_blue_bal_bip_two_red_bip} and is of different nature.

Case (\ref{itm_struct_two_largish_compts}) can be resolved by Lemma \ref{lem_two_three_quarter_sized} in Section \ref{sec_two_three-quarter-sized}.
It shares some ideas with the proofs in Section \ref{sec_large_bip_blue_almost_balanced}, but requires further analysis.
Case (\ref{itm_struch_only_midsized_compts}) is dealt with by Lemma \ref{lem_four_midsized} in Section \ref{sec_four_midsized}.
Case (\ref{itm_struct_spanning_non_bip}) turns out to be hardest to deal with, thus we prove it last in Lemma \ref{lem_large_blue_non_bip} in Section \ref{sec_large_blue_non_bip}.
\end{proof}

\section{Large weakly robust blue component}\label{sec_large_bip_blue_large_red}
In this section we deal with Case (\ref{itm_struct_spanning_blue_bip_large_red}) of Lemma \ref{lem_rough_structure}.
This is perhaps the easiest case to deal with. We recommend the reader to follow the proof here carefully, since the methods appearing here will be used in later sections, often in less detail.

In order prove Theorem \ref{thm_main} in Case (\ref{itm_struct_spanning_blue_bip_large_red}), we prove the following lemma. Note that we make the further assumption that the given robust components cover all the vertices of $G$ (rather than almost all of them). This can be easily be justified by Lemma \ref{lem_robust_adding_vertices}, which states that given a robust subgraph $F$, the graph obtained by adding vertices of large degree into $F$ remains robust.

\begin{lem}\label{lem_large_blue_bip_large_red}
Let $\frac{1}{n} \ll \eps, \alpha, \frac{1}{k} \ll 1$ and let $G$ be a graph of order $n$ with $\delta(G) \ge 3n / 4$ and a $2$-colouring $E(G) = E(G_B) \cup E(G_R)$.
Suppose that $F_1, F_2$ satisfy the following conditions.
\begin{itemize}
\item
$F_1\subseteq G_B$ is $(\alpha,k)$ weakly robust with bipartition $\{X, Y\}$ and $|F_1|\ge (1-\eps)n$.
\item
$F_2\subseteq G_R$ is $(\alpha,k)$-strongly robust and $|F_2|\ge (1/2+2\eps)n$.
\item
$V(G) = V(F_1) \cup V(F_2)$.
\end{itemize}
Then $V(G)$ may be partitioned into a blue cycle and a red cycle.
\end{lem}

This case is the most straightforward of the various cases arising from Lemma \ref{lem_rough_structure}. 
We apply Lemma \ref{lem_absorbing_paths} to find vertex disjoint absorbing paths $P_i$ in $F_i$ for $i \in [2]$. Then, using the regularity lemma and the connected matching method, we find two vertex-disjoint monochromatic cycles containing the paths $P_1, P_2$ and covering almost all of the vertices.
Finally, we use the absorption property to insert the remaining vertices into the paths $P_1, P_2$ so as to obtain the desired monochromatic cycle partition.

\begin{proof}[Proof of Lemma \ref{lem_large_blue_bip_large_red}]

We use Lemma \ref{lem_absorbing_paths} to build absorbing paths for $F_1$ and $F_2$. 
More precisely, given a suitably small $\rho>0$, there exists vertex disjoint paths $P_i\subseteq F_i$ for $i \in [2]$ satisfying the following conditions, where $U = V(P_1) \cup V(P_2)$.
\begin{itemize}
\item
$|P_i|\le \rho n$.
\item
For every set $W\subseteq V(F_1)\setminus U$, with $|W\cap X|=|W\cap Y|\le \rho^2 n$, the graph $F_1[V(P_1)\cup W]$ has a Hamilton path with the same ends as $P_1$.
\item
For every set $W\subseteq V(F_2)\setminus U$ of size at most $\rho^2 n$, the graph $F_2[V(P_2)\cup W]$ has a Hamilton path with the same ends as $P_2$.
\end{itemize}
Indeed, by Lemma \ref{lem_absorbing_paths}, we may find a $\rho^2n$-weakly-absorbing path $P_1$ in $F_1$ of length at most $\rho n$.
By Lemma \ref{lem_robust_removing_vertices}, the subgraph $F_2'=F_2\setminus V(P_1)$ is robust (with suitable parameters), thus we may find a $\rho^2n$ absorbing path $P_2$ in $F_2'$ of length at most $\rho n$.

Using the regularity lemma, we shall find vertex-disjoint monochromatic cycles, containing the paths $P_1$ and $P_2$ and covering most vertices of $G$.
We then cover the remaining vertices using the absorption properties of the paths $P_1$ and $P_2$. However, the fact that $P_1$ is weakly-absorbing presents a technical difficulty which we overcome as follows.

Recall that $\{X, Y\}$ is the bipartition of $F_1$.
Without loss of generality, suppose that $|X\cap V(F_2)|\ge n/8$. Note that $|Y|\ge \alpha n$ by the minimum degree condition on $F_1[X, Y]$.
Pick subsets $A_1\subseteq X \cap V(F_2)$ and $A_2 \subseteq Y$, disjoint of $U$, such that $|A_1|=\rho^2n/2$ and $|A_2|=\rho^2n/4$, and denote $A = A_1 \cup A_2$.

Apply the Regularity Lemma, Lemma\ref{lem_regularity}, to the graph $G \setminus (U \cup A)$ with parameter $\eta=\rho^2/16$ and the cover $\{V(F_1), V(F_2)\}$.
Let $\Gamma$ be the corresponding $(\eta, 6\eta)$-reduced graph.
Note that $\delta(G \setminus (U \cup A))\ge (3/4-3\rho)n$. It follows from observation \ref{obs_min_deg_reduced_graph} that for a suitable choice of $\rho$ we have $\delta(\Gamma)\ge (3/4-\eps/2)m$, where $m=|\Gamma|$.

The blue subgraph of $\Gamma$ determined by the clusters contained in $V(F_1)$ is connected by Lemma \ref{lem_robust_regularity}. Let $\Phi_1$ be the connected component of $\Gamma_B$ containing that subgraph. Similarly, let $\Phi_2$ be the connected component of $\Gamma_R$ containing the clusters which are contained in $V(F_2)$.
Note that $\Phi_1$ and $\Phi_2$ cover $V(\Gamma)$. Furthermore, $|\Phi_1|\ge (1-3\eps/2)m$ and $|\Phi_2|\ge (1/2+3\eps/2)m$.
Let $\Gamma'$ be the union of these graphs.
\begin{claim}\label{claim_pft_matching_case_2}
$\Gamma'$ has a perfect matching.
\end{claim}

\begin{proof}
Denote  $V_1=V(\Phi_1)\cap V(\Phi_2)$ and $V_2=V(\Phi_1) \triangle V(\Phi_2)$. Note that $V_1,V_2$ partition $V(\Gamma)$ and $|V_1|\ge m/2$.

Recall that $\delta(\Gamma)\ge (3/4-\eps/2)m$.
It follows that for every $v\in V_1$, we have $\deg_{\Gamma'}(v)\ge(3/4-\eps/2)m$ because all edges of $\Gamma$ incident to $v$ are in $\Gamma'$.
Vertices not in $\Phi_1$ have blue degree at most $\frac{3}{2}\eps m$. Similarly, vertices not in $\Phi_2$ have red degree at most $(1/2-3\eps/2)m$.
It follows that for every $v \in V_2$, we have $\deg_{\Gamma'}(v)\ge(1/4+ \eps)m$.

We conclude from Theorem \ref{thm_chvatal} that $\Gamma$ has a Hamilton cycle and in particular a perfect matching.
\end{proof}

\begin{rem}
In the last claim, we implicitly assumed that $m$ is even. It is indeed possibly to make this further assumption in the Regularity Lemma. We shall make this assumption whenever convenient without stating so explicitly.  
\end{rem}

By Claim \ref{claim_pft_matching_case_2}, $\Gamma$ has a perfect matching consisting of a blue connected matching in $\Phi_1$ and a red connecting matching in $\Phi_2$. Thus, we may use Lemma \ref{lem_connected_matching} to obtain a blue cycle $C_1$ and a red cycle $C_2$ which are disjoint, each $C_i$ contains the respective absorbing path $P_i$ and together they cover all but at most $7\eta n\le \rho^2n/4$ vertices of $V(G)\setminus A$.

We now show how to absorb the leftover vertices into the cycles $C_1,C_2$.
Let $B$ be the set of vertices which are not contained in the cycles $C_1, C_2$ or in the set $A$ (so $|B| \le \rho ^2 n / 4$).
Denote $B_1 = X \cap B$, $B_2 = Y \cap B$, and $B_3 = B \setminus V(F_1)$.
Recall that $A_1\subseteq X\cap V(F_2)$ and $A_2\subseteq Y$ are disjoint of the cycles $C_1,C_2$ and have sizes $\rho^2n/2$ and $\rho^2n/4$ respectively.
It follows that $\rho^2n/4\le |A_2\cup B_2|\le \rho^2n/2$. Thus we may choose $A_1'\subseteq A_1$ such that $A_1'\cup B_1$ and $A_2\cup B_2$ are of equal size which is most $\rho^2n/2$. Recall that the path $P_1$, which is contained in $C_1$ is $\rho^2n$-absorbing in $F_1$, so these sets can be absorbed by $P_1$ and thus by $C_1$.
We remain with the vertices $(A_1\setminus A_1')\cup B_3$. There are at most $\rho^2 n$ of them and they belong to $F_2$, so we may absorb them into $P_2$ and thus into $C_2$.
This completes a partition of $V(G)$ into a red and a blue cycle.
\end{proof}

Now that we have proved Theorem \ref{thm_main} in Case (\ref{itm_struct_spanning_blue_bip_large_red}) of Lemma \ref{lem_rough_structure}, we are ready to consider harder cases.

\section{Large weakly robust blue graph with almost balanced bipartition} \label{sec_large_bip_blue_almost_balanced}

In this section we consider Condition (\ref{itm_struct_spanning_blue_bip_almost_balanced}) of Lemma \ref{lem_rough_structure}, where we have a large weakly robust blue component with an almost balanced bipartition $\{X_1, X_2\}$.
There are three possibilities for each of the red graphs $G_R[X_i]$.
In this section we focus on the case where at least one of these red subgraphs satisfies conditions (\ref{itm_struct_blue_bip_balanced_red_strongly_robust}) or (\ref{itm_struct_blue_bip_balanced_two_red}).
The remaining case, with both red graphs $G_R[X_1], G_R[X_2]$ satisfying condition (\ref{itm_struct_blue_bip_balanced_red_weakly_robust}) requires a proof of different nature and is thus dealt with in the following section.

The main idea in the various cases arising here is that if two red robust components may be joined by two vertex-disjoint red paths, then they can essentially be treated as one bigger component, at which case we may use the argument of the previous section to finish the proof.
If that is not possible, we deduce that the red graph $G_R$ may be disconnected by removing a small number of vertices into two almost half-sized subgraphs.
This case may be resolved by Lemmas \ref{lem_red_graph_disconnected_1} and \ref{lem_red_disconnected_2}.

The following lemma deals with the case where one of the red graphs in question satisfies condition (\ref{itm_struct_blue_bip_balanced_red_strongly_robust}) and the other satisfies (\ref{itm_struct_blue_bip_balanced_red_strongly_robust}) or (\ref{itm_struct_blue_bip_balanced_red_weakly_robust}).

\begin{lem} \label{lem_blue_bip_balanced_two_red_not_both_bip}
Let $\frac{1}{n} \ll \eps, \alpha, \frac{1}{k} \ll 1$ and let $G$ be a graph of order $n$ with $\delta(G) \ge 3n / 4$ and a $2$-colouring $E(G) = E(G_B) \cup E(G_R)$.
Suppose that $F_1, F_2, F_3$ satisfy the following assertions.
\begin{itemize}
\item
$F_1\subseteq G_B$ is $(\alpha,k)$ weakly robust and with bipartition $\{X, Y\}$, where $|X| , |Y| \ge (1/2 - \eps)n$ and $e(G_B[X]), e(G_B[Y]) \le \eps n^2$.
\item
$F_2 \subseteq G_R$ is $(\alpha,k)$-strongly robust and $V(F_2) = X$.
\item
$F_3 \subseteq G_R$ is $(\alpha, k)$-robust with $V(F_3) = Y$.
\end{itemize}
Then $V(G)$ may be partitioned into a blue cycle and a red cycle.
\end{lem}

To prove this lemma, we show that we may either join the two robust components $F_2, F_3$ or we can finish using Lemma \ref{lem_red_graph_disconnected_1}.

\begin{proof}[Proof of Lemma \ref{lem_blue_bip_balanced_two_red_not_both_bip}]

By Menger's theorem, one of the following holds.
\begin{enumerate}
\item \label{itm_1_lem_blue_bal_bip_two_red}
There exist two vertex-disjoint red paths $P_2,P_3$, each having one end in $F_2$ and the other in $F_3$. 
\item \label{itm_2_lem_blue_bal_bip_two_red}
There exists  $u \in V(G)$ such that $F_2 \setminus \{u\}$ and $F_3 \setminus\{u\}$ are disconnected in $G_R \setminus \{u\}$.
\end{enumerate}

In Case (\ref{itm_2_lem_blue_bal_bip_two_red}), it is easy to deduce that the required monochromatic cycle partition exists from Lemma \ref{lem_red_disconnected_2}.

It remains to consider Case (\ref{itm_1_lem_blue_bal_bip_two_red}).
The idea is to use the paths $P_2$ and $P_3$ so as to essentially connect the two components $F_2,F_3$ into one large component. We achieve this as follows.

Note that we may assume that the internal vertices of $P_2,P_3$ belong to $V(G)\setminus (V(F_2)\cup V(F_3))$. It follows that $|P_2|, |P_3| \le 2 \eps n$ and the components $F_1,F_2,F_3$ remain robust after removing the vertices of the paths $P_2,P_3$.
Note that every vertex has either large (say, at least $\alpha n$) blue degree into $F_1$, or large red degree into either $F_2$ or $F_3$.
It follows that we may extend these components to cover the remaining vertices.  Namely, using Lemma \ref{lem_robust_adding_vertices}, we obtain $(\alpha/2,k+2)$-robust components $F_1',F_2',F_3'$ which extend the given components and cover $V(G)\setminus (V(P_2)\cup V(P_3))$. Denote by $\{X',Y'\}$ the bipartition of $F_1'$.

By Lemma \ref{lem_absorbing_paths}, $F_i'$ contains a $\rho^2 n$-absorbing path $P_i'$ of length at most $\rho n$ for each $i\in [3]$. We may assume that the paths $P_1',P_2',P_3'$ are vertex disjoint. For $i\in \{2,3\}$, we may connect $P_i$ with $P_i'$ in $F_i$ to a path $Q_i$ using at most $k + 2$ additional vertices. For convenience, we denote $Q_1=P_1'$.
To conclude, the paths $Q_1\subseteq G_B$ and $Q_2,Q_3\subseteq G_R$ are vertex-disjoint paths, such that $Q_i$ is $\rho^2 n$-absorbing in $F_i'$. Furthermore, each of $Q_2,Q_3$ has one end in $F_2'$ and one in $F_3'$.

Recall that $F_1$ is weakly robust and that $F_2$ is strongly robust. It is perhaps not evident from the definitions that a strongly robust graph is weakly robust, but for our purpose, the place where the difference is important is in the definition of an absorbing path. But clearly, a weakly absorbing path is strongly robust. Thus, it suffices to consider the case where $F_3'$ is weakly robust with bipartition $\{Z_1,Z_2\}$.

Similarly to the proof in Section \ref{sec_large_bip_blue_large_red}, in order to overcome the technical issues arising when dealing with weakly robust components, we pick sets $A_1,A_2,A_3$ such that
\begin{itemize}
\item
$A_1,A_2,A_3$ are pairwise vertex disjoint and do not intersect $V(Q_1)\cup V(Q_2)\cup V(Q_3)$.
\item
$A_1\subseteq Z_1\cap Y$, $A_2\subseteq Z_2\cap Y$ and $A_3\subseteq X\cap V(F_2)$.
\item
$|A_1|=\rho^2 n /16$, $|A_2|=\rho^2 n / 4$ and $|A_3|=\rho^2 n /2$.
\end{itemize}

We now consider the $(\eta, 6\eta)$-reduced graph obtaining by applying the Regularity Lemma, Lemma \ref{lem_regularity}, to the graph obtained from $G$ by removing the vertices in the $Q_i$'s and $A_i$'s (with the cover $\{V(F_1'), V(F_2'), V(F_3')\}$).
The reduced graph $\Gamma$ consists of a large blue component $\Phi_1$ (containing almost all vertices) and two disjoint almost half-sized connected red subgraphs $\Phi_2,\Phi_3$.
It is easy to verify, similarly to the proof of Claim \ref{claim_pft_matching_case_2}, using Theorem \ref{thm_chvatal}, that $\Gamma$ has a perfect matching consisting of edges in $\Phi_1, \Phi_2$ and $\Phi_3$.

By Lemma \ref{lem_connected_matching}, there exist a blue cycle $C_1$ and a red cycle $C_2$ such that 
\begin{itemize}
\item
$C_1$ and $C_2$ are vertex-disjoint and do not intersect $A = A_1 \cup A_2 \cup A_3$.
\item
They cover all but at most $4\eta n$ vertices of $V(G)\setminus A$.
\item
$C_1$ contains the path $Q_1$ and $C_2$ contains the paths $Q_2,Q_3$.
\end{itemize}
Let us elaborate slightly more on how to obtain the required cycles by pointing out that $C_2$ may be obtained by connecting the ends of $Q_2,Q_3$ by two paths, one in $F_2'$ and the other in $F_3'$.

Let $B$ be the set of vertices which do not belong to the cycles $C_1, C_2$ or to $A$ (so $|B| \le  7\eta n \le \rho^2 n / 16$) and denote
\begin{align*}
&B_1 = B \cap Z_1,\quad B_2 = B \cap Z_2 \\
&B_3 = (B \cap X')\setminus (B_1 \cup B_2) \\
&B_4 = (B \cap Y')\setminus (B_1 \cup B_2) \\
&B_5 = B \setminus (B_1 \cup \ldots \cup B_4).
\end{align*}
We perform the following steps in order to absorb $B$.
\begin{itemize}
\item
We have $|B_2|\le \rho^2 n /16 \le |A_1\cup B_1| \le \rho^2 n /8$ and $|A_2|=\rho^2 n /4$.
Thus we may choose $A_2'\subseteq A_2$ such that $|A_1\cup B_1|=|A_2'\cup B_2|$.
The vertices $A_1\cup B_1\cup A_2'\cup B_2$ can be absorbed into $Q_3$.
\item
Similarly, we may choose $A_3'\subseteq A_3$ such that $|(A_2\setminus A_2')\cup B_3|=|A_3'\cup B_4|$.
The vertices $(A_2\setminus A_2')\cup B_3\cup A_3'\cup B_4$ can be absorbed into $Q_1$.
\item
The remaining vertices $(A_3\setminus A_3') \cup B_5$ can be absorbed into $Q_2$.
\end{itemize}
This completes the partition of $V(G)$ into a blue cycle and a red one.

\end{proof}

The following lemma deals with the case where one of the red graphs in question satisfies Condition (\ref{itm_struct_blue_bip_balanced_two_red}) in Lemma \ref{lem_rough_structure} and the other satisfies one of the other two conditions.

\begin{lem}\label{lem_large_blue_bip_three_red}
Let $\frac{1}{n} \ll \eps, \alpha, \frac{1}{k} \ll 1$ and let $G$ be a graph of order $n$ with $\delta(G) \ge 3n / 4$ and a $2$-colouring $E(G) = E(G_B) \cup E(G_R)$.
Suppose that $F_1, F_2, F_3, F_4$ satisfy the following assertions.
\begin{itemize}
\item
$F_1\subseteq G_B$ is $(\alpha,k)$ weakly robust and with bipartition $\{X, Y\}$, where $|X| , |Y| \ge (1/2 - \eps)n$.
\item
$F_2 \subseteq G_R$ is $(\alpha,k)$-robust and $V(F_2) = X$.
\item
$F_3, F_4 \subseteq G_R$ are $(\alpha, k)$ strongly robust components of order at least $(1/4 - \eps)n$ whose vertex sets partition $Y$.
Furthermore, $e(G_R[V(F_3), V(F_4)]) \le \eps n^2$.
\end{itemize}
Then $V(G)$ may be partitioned into a blue cycle and a red cycle.
\end{lem}

The proof of this lemma is similar to the previous one. Either $F_2$ may be joined to one of $F_3, F_4$ or we may finish using Lemma \ref{lem_red_graph_disconnected_1}.

\begin{proof}[Proof of Lemma \ref{lem_large_blue_bip_three_red}]

Similarly to the previous case, two possibilities arise.
\begin{enumerate}
\item \label{itm_1_lem_blue_bal_bip_three_red}
There exist vertex-disjoint red paths $P_1, P_2$ with one end in $F_2$ and either both have the other end in $F_3$ or both have the other end in $F_4$. 
\item \label{itm_2_lem_blue_bal_bip_three_red}
There exists a set $U$ of size at most $2$ such that $F_1 \setminus U$ and $(F_2 \cup F_3) \setminus U$ are disconnected in $G_R \setminus U$.
\end{enumerate}

In Case (\ref{itm_2_lem_blue_bal_bip_three_red}), the required monochromatic cycle partition exists by Lemma \ref{lem_red_graph_disconnected_1}. Indeed, we may find $S \subseteq X \setminus U$ and $T \subseteq Y \setminus U$ such that the three conditions in the lemma hold, for some parameter $\eta = \eta(\eps)$.
In particular, the third condition holds because most vertices in $Y$ have degree at most $(1/4 + \sqrt{\eps})n$ in $G[Y]$, and thus have degree at least $(1/2 - \sqrt{\eps})n$ into $X$.

Case (\ref{itm_1_lem_blue_bal_bip_three_red}) may be dealt with similarly to the previous proof of Lemma \ref{lem_blue_bip_balanced_two_red_not_both_bip}, we omit further details.
\end{proof}

The following lemma deals with the remaining case, where both red graphs satisfy condition (\ref{itm_struct_blue_bip_balanced_two_red}) in Lemma \ref{lem_rough_structure}.

\begin{lem} \label{lem_blue_bal_bip_four_red}
Let $\frac{1}{n} \ll \eps, \alpha, \frac{1}{k} \ll 1$ and let $G$ be a graph of order $n$ with $\delta(G) \ge 3n / 4$ and with a $2$-colouring $E(G) = E(G_B) \cup E(G_R)$.
Suppose that $F_1, F_2, F_3, F_4, F_5$ satisfy the following conditions.
\begin{itemize}
\item
$F_1\subseteq G_B$ is $(\alpha,k)$ weakly robust and with bipartition $\{X, Y\}$, where $|X| , |Y| \ge (1/2 - \eps)n$.
\item
$F_2, F_3 \subseteq G_B$ are $(\alpha, k)$ strongly robust components of order at least $(1/4 - \eps)n$ whose vertex sets partition $X$.
\item
$F_4, F_5 \subseteq G_B$ are $(\alpha, k)$ strongly robust components of order at least $(1/4 - \eps)n$ whose vertex sets partition $Y$.
\item
$e(G_R[V(F_2), V(F_3)]), e(G_R[V(F_4), V(F_5)]) \le \eps n^2$.
\end{itemize}
Then $V(G)$ may be partitioned into a blue cycle and a red cycle.
\end{lem}

To prove this lemma we follow similar ideas to the previous results in this section. We show that either at least three of the four components $F_2, F_3, F_4, F_5$ may be joined or we may finish using Lemma \ref{lem_red_graph_disconnected_1} or Lemma \ref{lem_red_disconnected_2}.

\begin{proof}[Proof of Lemma \ref{lem_blue_bal_bip_four_red}]
We consider four cases.
In order to be able to distinguish between them, we define a graph $H$ on vertex set $\{2, 3, 4, 5\}$ with an edge $(i, j)$, where $i \in \{2, 3\}, j \in \{4, 5\}$ if $e(G_R[F_i, F_j]) \ge \eps n^2$ (so $H$ is a bipartite graph on four vertices with bipartition $\{[2, 3], [4, 5]\}$.
Clearly, one of the following conditions.
\begin{enumerate}
\item \label{itm_graph_H_two_path}
$H$ contains a path of length $2$.
\item \label{itm_graph_H_two_matching}
$H$ consist of two vertex-disjoint edges.
\item \label{itm_graph_H_one_edge}
$H$ has exactly one edge.
\item \label{itm_graph_H_no_edges}
$H$ is the empty graph. 
\end{enumerate}
In Case (\ref{itm_graph_H_two_path}), without loss of generality suppose that $(2, 4), (2, 5) \in E(H)$.
This means that $e(G_R[F_2, F_4]), e(G_R[F_2, F_5]) \ge \eps n^2$.
In particular, there exist four vertex-disjoint edges $e_1, e_2 \in G_R[F_2, F_4], e_3, e_4 \in G_R[F_2, F_5]$.
We deduce that the components $F_2, F_4, F_5$ may be joined to form a red component of size at least $(3/4 - 3\eps)n$.
We may now finish the proof of Lemma \ref{lem_blue_bal_bip_four_red} similarly to the previous lemmas in this section.

We now suppose that Case (\ref{itm_graph_H_two_matching}) holds. Without loss of generality, $E(H) = \{(2, 4), (3, 5)\}$.
If there are two vertex-disjoint red paths $P_1, P_2$ between $V(F_2) \cup V(F_4)$ and $V(F_3) \cup V(F_5)$, we conclude that the four component $F_2, F_3, F_4, F_5$ may be joined (note that to join $F_2$ to $F_4$ and $F_3$ to $F_5$ we use edges between them which can be chose so as to not intersect the given paths $P_1, P_2$).
Otherwise, there is a vertex $u \in V(G)$ such that $G_R \setminus \{u\}$ is disconnected, with $F_2, F_4$ in one component and $F_3, F_5$ in another.
The proof of Lemma \ref{lem_blue_bal_bip_four_red} can be completed by Lemma \ref{lem_red_disconnected_2}.

We now assume that Case (\ref{itm_graph_H_one_edge}) holds.
Without loss of generality, $e(G_R[V(F_2), V(F_4)]) \ge \eps n^2$.
Consider the graph $G' = G_R[V(F_2) \cup V(F_4), V(F_3) \cup V(F_5)]$.
If it contains a matching of size at least $5$, then the one of $F_2, F_4$ may be joined to one of $F_3, F_5$. Since there are many edges between $F_2$ and $F_4$, it follows that we may form a red component using $F_2, F_4$ and one of $F_3, F_5$.
The proof may be completed as before.
Thus we assume that the above graph has no matching of size $5$.
It follows that we may remove four vertices from $G'$ so as to disconnect $F_2$ and $F_4$ from $F_3$ and $F_5$.
Thus there exists a set $U \subseteq V(G)$ of size at most $3 \eps n$ such that $F_2$ and $F_4$ are disconnected from $F_3$ and $F_5$.
Lemma \ref{lem_blue_bal_bip_four_red} now follows from Lemma \ref{lem_red_graph_disconnected_1} (note that most vertices in $V(F_3) \cup V(F_5)$ have red degree at most about $n / 4$).

Finally, we consider Case (\ref{itm_graph_H_no_edges}).
Denote $U = V(F_2) \cup V(F_3) \cup V(F_4) \cup V(F_5)$. Let $W$ be the union of $V(G) \setminus U$ with the set of vertices in $V(G)$ which have degree at least $2 \sqrt{\eps} n$ in $G_R[V(F_2) \cup V(F_3), V(F_4) \cup V(F_5)]$.
Since $H$ is the empty graph, we have $|U| \le 5\sqrt{\eps}n$, thus we may complete the proof by Lemma \ref{lem_red_graph_disconnected_1}.
\end{proof}

In order to finish the proof of Theorem \ref{thm_main} under the assumption that Condition (\ref{itm_struct_spanning_blue_bip_almost_balanced}) from Lemma \ref{lem_rough_structure}, we need to consider the case where both graphs in question satisfy Condition (\ref{itm_struct_blue_bip_balanced_red_weakly_robust}).
This is done in the next section, Section \ref{sec_large_blue_bip_two_midsized_red}.

\section{Almost balanced large weakly robust blue component and two half-sized weakly robust red components}\label{sec_large_blue_bip_two_midsized_red}

In this section we consider Case (\ref{itm_struct_spanning_blue_bip_almost_balanced}) from Lemma \ref{lem_rough_structure}, where both graphs $G_R[X_1]$ and $G_R[X_2]$ satisfy condition (\ref{itm_struct_blue_bip_balanced_red_weakly_robust}).
Namely, we have a large weakly robust blue component with an almost balanced bipartition $\{X_1, X_2\}$. Furthermore, $G_R[X_i]$ is weakly robust for $i \in [2]$ with an almost balanced bipartition $\{Y_{i, 1}, Y_{i, 2}\}$ such that $e(G_R[Y_{i, 1}, Y_{i, 2}]) \le \eps n^2$.

\begin{lem} \label{lem_blue_bal_bip_two_red_bip}
Let $\frac{1}{n} \ll \eps \ll 1$ and let $G$ be a graph of order $n$ with $\delta(G) \ge 3n / 4$ and with a $2$-colouring $E(G) = E(G_B) \cup E(G_R)$.
Suppose that there exists four disjoint sets $Y_{i, j}$, $i, j \in [2]$ with the following properties.
\begin{itemize}
\item
$|Y_{i, j}| \ge (1/4 - \eps)n$ for $i, j \in [2]$.
\item
$e(G[Y_{i, j}]) \le \eps n^2$.
\item
$e(G_B[Y_{i, 1}, Y_{i, 2}]) \le \eps n^2$ for $i \in [2]$.
\item
$e(G_R[Y_{1, j}, Y_{2, j}]) \le \eps n^2$ for $j \in [2]$.
\end{itemize}
\end{lem}

We notice that a graph with the given conditions has a rather specific structure. Namely, the sets $Y_{i, j}$ span few edges, whereas the graphs $G_R[Y_{i, 1}, Y_{i, 2}]$ and $G_B[Y_{1, j}, Y_{2, j}]$ are almost complete.
By Lemma \ref{lem_red_graph_disconnected_1}, we conclude that we may finish the proof unless say $G_R[Y_{1, 1}, Y_{2, 2}]$ and $G_B[Y_{1, 2}, Y_{2, 1}]$ are almost complete.
In the latter case we construct the required partition into a red cycle and a blue one ``by hand''.

\begin{proof} [Proof of Lemma \ref{lem_blue_bal_bip_two_red_bip}]

The conditions imply that for some $\eta = \eta(\eps)$, we can find disjoint sets $S_1, S_2, T_1, T_2$ such that the following holds.

\begin{itemize}
\item
$|S_i|,|T_i|\ge (1/4-\eta)n$ for $i\in [2]$.
\item
$\delta(G_R[S_1,S_2])\ge (1/4-\eta)n$ and $\delta(G_R[T_1,T_2])\ge (1/4-\eta)n$.
\item
$\delta(G_B[S_i,T_i])\ge (1/4-\eta)n$ for $i\in [2]$.
\item
$e(G[S_i]), e(G[T_i]) \le \eta n^2$.
\end{itemize}

Denote $S = S_1 \cup S_2$ and $T = T_1 \cup T_2$.
Consider the graph $G_B[S, T]$.
If for every $S' \subseteq S$ and $T' \subseteq T$ with $|S'|, |T'| \ge (1/4 - 200\eta)n$ we have $e(G_B[S', T']) \ge 50\eta n^2$, Lemma \ref{lem_red_graph_disconnected_1} implies that $G$ may be partitioned into a blue cycle and a red one.
Thus we may assume that there exist subsets $S' \subseteq S$ and $T' \subseteq T$ of size $(1/4 - 200 \eta)n$ such that $e(G_B[S', T']) \le 50\eta n^2$.

\begin{claim}
Either $|S_1 \cap S'|, |T_2 \cap T'| \le 10 \sqrt{\eta}n$ or $|S_2 \cap S'|, |T_1 \cap T'| \le 10 \sqrt{\eta}$.
\end{claim}
\begin{proof}
Suppose to the contrary that $|S_1 \cap S'|, |T_1 \cap T'| \ge 10 \sqrt{\eta}n$.
For every $u \in S_1$, the number of vertices in $T_1$ which are not blue neighbours of $u$ is at most $4 \eta n$. If $S' \cap S_1$ and $T' \cap T_1$ both have size at least $10 \sqrt{\eta}$, we deduce $e(G_B[S', T']) \ge 90\eta n$, a contradiction.
\end{proof}
By the above claim, without loss of generality, we may assume that $|S_2 \cap S'|, |T_1 \cap T'| \le 10 \sqrt{\eta}n$, so $G_B[S_1, T_2]$ is almost empty, and $G_R[S_1, T_2]$ is almost complete.
Similar arguments imply that we may assume that $G_B[S_2, T_1]$ is almost complete.

We deduce that there exists $\rho = \rho(\eps)$ such that we may find disjoint sets $A_1, A_2, A_3, A_4$ satisfying the following conditions.
\begin{itemize}
\item
$|A_1|, |A_2|, |A_3|, |A_4| \ge (1/4 - \rho)n$.
\item
$\delta(G_B[A_1, A_2]), \delta(G_B[A_2, A_3]), \delta(G_B[A_3, A_4]) \ge (1/4 - \rho)n$.
\item
$\delta(G_R[A_1, A_3]), \delta(G_R[A_1, A_4]), \delta(G_R[A_2, A_4]) \ge (1/4 - \rho)n$.
\end{itemize}

For the time being, we assume that $n$ is even.
We obtain a partition $\{A_1', A_2', A_3', A_4'\}$ of $V(G)$ by adding each vertex $u \in V(G) \setminus (A_1 \cup A_2 \cup A_3 \cup A_4)$, to one of the sets $A_1, A_2, A_3, A_4$ as follows.
\begin{align*}
\text{If}\quad
\begin{array}{l}
\deg_B(u, A_2) \ge n/32 \\
\deg_B(u, A_3) \ge n/32 \\ 
\deg_R(u, A_2) \ge n/32 \\ 
\deg_R(u, A_3) \ge n/32 \\
\end{array} 
\quad
\text{ add } 
u 
\text{ to }
\quad
\begin{array}{l}
A_3 \\
A_2 \\
A_4 \\
A_1
\end{array}
\end{align*}
Note that every vertex will be added to one of the $A_i$'s.

Denote $m_i = |A_i|$.
We will find values $k_1, k_2, k_3$ and $l_1, l_2,l _3$ such that there exist a partition of $V(G)$ into a blue cycle $C_1$ and a red cycle $C_2$ with the following properties.
\begin{align*}
|V(C_1) \cap A_i| = 
\left\{
\begin{array}{ll}
k_1 & i = 1 \\
k_1 + k_2 & i = 2 \\
k_2 + k_3 & i = 3 \\
k_3 & i = 4
\end{array}
\right.
\qquad
|V(C_2) \cap A_i| = 
\left\{
\begin{array}{ll}
l_1 + l_2 & i = 1 \\
l_3 & i = 2 \\
l_1 & i = 3 \\
l_2 + l_3 & i = 4
\end{array}
\right.
\end{align*}
To that end, we find blue paths $P_1, P_2, P_3, P_4$ forming a blue cycle $C_1 = P_1P_2P_3P_4$ such that the following assertions hold. 
\begin{itemize}
\item
$P_1 \in G_B[A_1', A_2']$, its ends are in $A_2'$ and it has $k_1$ vertices in $A_1'$.
\item
$P_2, P_4 \in G_B[A_2', A_3']$, both have one end in $A_2'$ and the other in $A_3'$ and together they have $k_2 + 1$ vertices in $A_2'$.
\item
$P_3 \in G_B[A_3', A_4']$, its ends are in $A_3'$ and it has $k_3$ vertices in $A_4'$.
\end{itemize}

Similarly, we find red paths $Q_1, Q_2, Q_3, Q_4$ forming a red cycle $C_2 = Q_1Q_2Q_3Q_4$ such that $C_1$ and $C_2$ partition $V(G)$ and the following assertions hold.
\begin{itemize}
\item
$Q_1 \in G_R[A_1', A_3']$, its ends are in $A_1'$ and it has $l_1$ vertices in $A_3'$.
\item
$Q_2, Q_3 \in G_R[A_1', A_4']$, both have one end in $A_1'$ and the other in $A_4'$ and together they have $l_2 + 1$ vertices in $A_1'$.
\item
$Q_4 \in G_R[A_2', A_4']$, its ends are in $A_2'$ and it has $l_3$ vertices in $A_2'$.
\end{itemize}

The values $k_1, k_2, k_3$ and $l_1, l_2, l_3$ clearly need to satisfy the following system of equations.
\begin{align*}
\begin{array}{l}
m_1 = k_1 + l_1 + l_2 \\
m_2 = k_1 + k_2 + l_3 \\
m_3 = k_2 + k_3 + l_1 \\
m_4 = k_3 + l_2 + l_3
\end{array}
\end{align*}
Which may be solved as follows.
\begin{align*}
&k_1 = k_3 + \frac{1}{2}(m_1 + m_2 - m_3 - m_4) \\
&l_1 = l_3 + \frac{1}{2}(m_1 - m_2 + m_3 - m_4) \\
&k_2 = -(k_3 + l_3) + m_4 \\ 
&l_2 = -(k_3 + l_3) + m_3 - \frac{1}{2}(m_1 - m_2 + m_3 - m_4).
\end{align*}
Since $n$ is even, if $k_3, l_3$ are integers then so are $k_1, k_2, l_1, l_2$.
Note that  the following inequalities hold.
\begin{align*}
\begin{array}{l}
k_3 - 2\rho n \le k_1 \le k_3 + 2\rho n \\
l_3 - 2\rho n \le l_1 \le l_3 + 2\rho n
\end{array}
\qquad
\begin{array}{l}
k_2 \ge (1/4 - \rho)n - (k_3 + l_3) \\
l_2 \ge (1/4 - 3\rho)n - (k_3 + l_3).
\end{array}
\end{align*}
We pick $l_3 = k_3 = 12\rho n$. It follows that $10 \rho n \le l_1, k_1 \le 14\rho n$.
We now choose the paths $P_1, P_2, P_3$ as follows.
Let $P_2$ be any edge $u_2u_3 \in G_B[A_2, A_3]$. Greedily pick paths $P_1$ in $G_B[A_1, A_2]$ starting with $u_2$ and ending in some $u_1 \in A_2$, such that $P_1$ contains $k_1$ vertices from $A_1$. Similarly, let $P_3$ be a path in $G_B[S_3, A_4]$ with $k_3$ vertices from $A_4$ and ends $u_3$ and $u_4$ where $u_4$ is some vertex in $A_4$.

We now construct $Q_1, Q_2, Q_3$ as follows (we implicitly ensure that these paths are disjoint of the paths $P_1, P_2, P_3$).
Let $Q_2$ be any edge $v_2v_3 \in G_R[A_1, A_4]$.
Pick $Q_1$ to be a path in $G_R[A_1', A_3]$ containing all vertices $A_1' \setminus A_1$, with ends $v_2$ and $v_1 \in A_1$, and with $l_1$ vertices from $A_3'$.
We remark that such a path exists. Indeed, we have $|A_1' \setminus A_1| \le 5\rho n$ and any two vertices in $A_1' \setminus A_1$ may be connected using at most three additional vertices from $A_1 \cup A_3$. Thus we may find a path with at most $10\rho n$ vertices from $A_1$ starting with $v_2$ and containing the vertices $A_1' \setminus A_1$. We extend it arbitrarily to the desired length.
Similarly, pick a path $Q_3 \in G_R[A_2, A_4']$ containing the vertices $A_4' \setminus A_4$, with ends $v_3$ and $v_4 \in A_4$ and with $l_3$ vertices in $A_2$.

Denote by $U$ the set of inner vertices in the paths $P_1P_2P_3$ and $Q_1Q_2Q_3$ and denote $A_i'' = A_i' \setminus U$ for $i \in [4]$.
By the definition of the paths $P_i$ and $Q_i$, $i \in [3]$ we obtain the following equalities.
\begin{align*}
&|A_1''| = m_1 - k_1 - l_1 = l_2 \\
&|A_2''| = m_2 - k_1 - l_3 = k_2 \\
&|A_3''| = m_3 - k_3 - l_1 = k_2 \\
&|A_4''| = m_4 - k_3 - l_3 = l_2.
\end{align*}
It is easy to verify that $G_B[A_2'', A_3'']$ contains a Hamilton path $P_4$ with ends $u_1, u_4$ and that $G_R[A_1'', A_4'']$ contains a Hamilton path $Q_4$ with ends $v_1, v_4$.

It remains to consider the case where $n$ is odd.
If there exists a vertex $u$ which has blue neighbours $v_1 \in A_1 \cup A_3$ and $v_2 \in A_2 \cup A_4$, we consider the graph $G \setminus \{u\}$ and partition it into a red cycle and a blue path with ends $v_1, v_2$. This may be done by the same arguments as when $n$ is even.
Thus we assume that no vertex has blue neighbour in both $A_1 \cup A_3$ and $A_2 \cup A_4$. By symmetry, we may also assume that no vertex has red neighbours in both $A_1 \cup A_2$ and $A_3 \cup A_4$.
It follows that every vertex has neighbours in at most three of the sets $A_1, A_2, A_3, A_4$, so it sends at least $(1/4 - 2\rho)n$ edges to each of these three sets.
We extend the sets $A_1, A_2, A_3, A_4$ as follows. For each $u \in V(G) \setminus A_1 \cup A_2 \cup A_3 \cup A_4$, 
\begin{align*}
\text{if}\quad
\begin{array}{l}
\deg_B(u, A_2), \deg_R(u, A_3) \ge (1/4 - 2\rho)n \\
\deg_B(u, A_1), \deg_R(u, A_4) \ge (1/4 - 2\rho)n \\ 
\deg_B(u, A_4), \deg_R(u, A_1) \ge (1/4 - 2\rho)n \\ 
\deg_B(u, A_3), \deg_R(u, A_2) \ge (1/4 - 2\rho)n \\
\end{array} 
\quad
\text{ add } 
u 
\text{ to }
\quad
\begin{array}{l}
A_1 \\
A_2 \\
A_3 \\
A_4
\end{array}
.
\end{align*}
It is easy to verify that the following claim holds.
\begin{claim}
The sets $A_1', A_2', A_3', A_4'$ partition $V(G)$ and satisfy the following conditions.
\begin{align*}
&\delta(G_B[A_1', A_2']), \delta(G_B[A_2', A_3']), \delta(G_B[A_3', A_4']) \ge (1/4 - 2\rho)n.\\
&\delta(G_R[A_1', A_3']), \delta(G_R[A_1', A_4']), \delta(G_R[A_2', A_4']) \ge (1/4 - 2\rho)n.
\end{align*}
\end{claim}
Without loss of generality, suppose that $|A_1'| \ge n/4$. Then in fact, $|A_1'| > n/4$, because $n$ is odd.
Hence $G[A_1']$ contains an edge $uv$. If it is blue, $u$ has blue neighbours in both $A_1'$ and $A_2'$. If it is red, $u$ contains red neighbours in both $A_1'$ and $A_4'$.
We may continue as before, to partition $G \setminus \{u\}$ into a monochromatic cycle and a monochromatic path.

\end{proof}

The proof of Lemma \ref{lem_blue_bal_bip_two_red_bip} completes the proof of our main Theorem under the assumption that Case (\ref{itm_struct_spanning_blue_bip_almost_balanced}) from Lemma \ref{lem_rough_structure} holds.

 \section{Two monochromatic robust components of size almost $3n/4$}\label{sec_two_three-quarter-sized}

In this section, we consider Case (\ref{itm_struct_two_largish_compts}) from Lemma \ref{lem_rough_structure}.
Similarly to Section \ref{sec_large_bip_blue_large_red}, we may assume that the given robust components cover $V(G)$.

\begin{lem} \label{lem_two_three_quarter_sized}
Let $\frac{1}{n} \ll \eps, \alpha, \frac{1}{k} \ll 1$ and let $G$ be a graph of order $n$ with $\delta(G) \ge 3n / 4$ and a $2$-colouring $E(G) = E(G_B) \cup E(G_R)$.
Suppose that $F_1, F_2$ satisfy the following assertions.
\begin{itemize}
\item
$F_1\subseteq G_B$ is $(\alpha,k)$ strongly robust and $|F_1| \ge (3/4 - \eps)n$.
\item
$F_2\subseteq G_R$ is $(\alpha,k)$-robust and $|F_2| \ge (3/4 - \eps)n$.
\item
$V(G) = V(F_1) \cup V(F_2)$.
\end{itemize}
Then $V(G)$ may be partitioned into a blue cycle and a red cycle.
\end{lem}

We proceed as before, building absorbing paths, and considering the reduced graph on the remaining vertices, where we have a blue component and a red one, each with almost $3/4$ of the vertices. If a perfect matching can be found using the edges in these component, we continue as before to obtain the required partition into cycles. If no such perfect matching exists, we conclude that the graph $G$ satisfies some structural conditions which enable us to either find the required partition ``by hand'', or to join two components in a similar way to previous cases.

\begin{proof}[Proof of Lemma \ref{lem_two_three_quarter_sized}]
Let $Q_1,Q_2$ be disjoint $\rho^2 n$-absorbing paths in $F_1,F_2$ respectively of length at most $\rho n$.
As pointed out in Section \ref{sec_large_bip_blue_almost_balanced} we may assume that $F_2$ is weakly robust. Denote its bipartition by $\{X, Y\}$.

Without loss of generality, suppose that $|X \cap V(F_1)|\ge n/8$ and $|Y| \ge \alpha n$.
Fix sets $A_1\subseteq X \cap V(H_1)$ of order $\rho^2 n/2$ and $A_2\subseteq Y$ of order $\rho^2 n /4$.
Finally, apply Lemma \ref{lem_regularity} to the graph $G$ with the vertices in $Q_1, Q_2, A_1, A_2$ removed, with the cover $\{V(F_1), V(F_2)\}$ and with small enough $\eta$. Consider the $(\eta, 6\eta)$-reduced graph $\Gamma$.

Note that, assuming $\rho, \eta$ are small enough, we have $\delta(\Gamma)\ge (3/4 - \eps)m$, where $m = |\Gamma|$. Furthermore, there is a blue component $\Phi_1$ and a red component $\Phi_2$ which cover $V(\Gamma)$ and are of order at least $(3/4 - 2\eps)m$ each.
We consider the subgraph $\Gamma'$ of $\Gamma$ spanned by the blue edges in $\Phi_1$ and the red ones in $\Phi_2$.

Consider the following claim.
\begin{claim}
One of the following conditions holds.
\begin{enumerate}
\item
$\Gamma'$ has a perfect matching.
\item \label{itm_2_claim_two_largish_compts}
The following holds for some $\eta, \beta, l$ depending only on $\eps, \alpha, k$.
There exist subsets $V_1\subseteq V(F_1)\setminus V(F_2)$, $V_2\subseteq V(F_2)\setminus V(F_1)$ and $V_0 \subseteq V(F_1) \cap V(F_2)$ such that 
\begin{itemize}
\renewcommand\labelitemi{--}
\item
$|V_1|, |V_2| \ge (1/4 - \eta)n$ and the graphs $G_R[V_1], G_B[V_2]$ are $(\beta, l)$ strongly robust.
\item
$|V_0| \ge (1/2 - \eta)n$ and $G_B[V_1 \cup V_0]$ and $G_R[V_2 \cup V_0]$ are $(\beta, l)$-robust, with at least one of them being strongly-robust.
\end{itemize}
\end{enumerate}

\end{claim}

\begin{proof}
Let $U_1 = V(\Phi_1) \setminus V(\Phi_2)$, $U_2 = V(\Phi_2) \setminus V(\Phi_1)$ and $U_0 = V(\Phi_1 \cap \Phi_2)$.
Note that the vertices in $U_1 \cup U_2$ have degree at least $(1/2 - 4\eps)m$ in $\Gamma'$, whereas $|U_0|\ge (1/2 - 4\eps)m$ and the vertices in $U_3$ have degree at least $(3/4-2\eps)m$ in $\Gamma'$.
If $\Gamma'$ has no prefect matching, it follows from Theorem \ref{thm_chvatal} that $|U_1 \cup U_2|\ge m/2$ and that the set $W$ of vertices $U_1 \cup U_2$ with degree at most $(1/2 + \eps)m$ has size at least $(1/2 - 4\eps)m$.
Denote $W_1 = W \cap U_1$ and $W_2 = W \cap U_2$.

We claim that $|W_1|, |W_2| \ge (1/4 - 13\eps)n$.
Indeed, the vertices in $W_1$ have blue degree at most $(1/2 + \eps)m$, so they have red degree at least $(1/4 - 5\eps)m$. Note that the red neighbourhood of $W_1$ is contained in $U_1$. It follows that $|U_1| \ge (1/4 - 5\eps)n$ and similarly $|U_2| \ge (1/4 - 5\eps)n$. Since $|U_0| \ge (1/2 - 4\eps)m$, we have that $|U_1|, |U_2| \le (1/4 + 9\eps)m$. Thus $|W_1|, |W_2| \ge (1/4 - 13\eps)m$.
It is easy to verify that $\Gamma_R[W_1]$ and $\Gamma_B[W_2]$ are connected and non bipartite (in fact, they are almost complete).

Let $W_0$ be the intersection of the blue neighbourhood of $W_1$ in $U_0$ and the red neighbourhood of $W_2$ in $U_0$.
By the definition of the sets $U_i$, there are no edges between $U_1$ and $U_2$, there are no red edges between $U_1$ and $U_0$ and no blue edges between $U_2$ and $U_0$.
It follows that the each vertex in $U_1$ has at least 
$(1/2 - 10\eps)m$ blue neighbours in $U_0$, and by the analogous argument for $U_2$, we have $|W_0| \ge (1/2 - 20\eps)m$.
Note that $\Gamma_B[W_0 \cup W_2]$ and $\Gamma_R[W_0 \cup W_2]$ are connected and it is not hard to see that at least one of them is non-bipartite.
It follows from Lemma \ref{lem_many_reduced_components} that condition (\ref{itm_2_claim_two_largish_compts}) holds.
\end{proof}

Similarly to the proofs in Section \ref{sec_large_bip_blue_almost_balanced}, one of the following holds.
\begin{enumerate}
\item \label{itm_two_three_quarters_1}
There are two vertex-disjoint blue paths between $V_2$ and $V_1 \cup V_0$.
\item \label{itm_two_three_quarters_2}
There are two vertex-disjoint red paths between $V_1$ and $V_2 \cup V_0$.
\item \label{itm_two_three_quarters_3}
There exists a set $X$ of size at most $2$ such that the sets $V_2$ and  $V_1 \cup V_0$ are not connected in $G_B \setminus X$ and the sets $V_1$ and $V_2 \cup V_0$ are not connected in $G_R \setminus X$.
\end{enumerate}

In the first two cases we proceed as in Section \ref{sec_large_bip_blue_almost_balanced} to join say the two blue component to form an almost spanning component, which together with the large red component ($V_2 \cup V_0$) may be used to find the required cycle partition.
Thus we assume that the Case (\ref{itm_two_three_quarters_3}) holds.

Consider the graph $G' = G \setminus X$.
Let $U_1$ be the component of $V_1$ in $G'_R$ and let $U_2$ be the blue component of $V_2$ in $G'_B$.
Then $e(G[U_1, U_2]) = 0$, so $|U_1|, |U_2| \le n/4 - 1$.

Denote $W = V(G') \setminus (U_1 \cup U_2)$ (so $|U| \ge n/2$).
We define $U_1', U_2'$ as follows.
For each $x \in X$, if $\deg_B(x, W) \ge n/8$, put $x$ into $U_1$. Otherwise we have $\deg_R(x, W) \ge n/8$, and we put $x$ into $U_2$.

Note that $e(G[W]) \ge n^2 / 16$. We may assume without loss of generality that $e(G_B[W]) \ge n^2 / 32$.
Denote $|W| = |U_1'| + |U_2'| + k$ and note that $0 \le k \le 5 \eta n$.

\begin{claim} \label{claim_partition_two_largish_compts}
There exists $\theta = \theta(\eps)$ such that $W$ may be partitioned into sets $W_1, W_2$ satisfying the following conditions.
\begin{itemize}
\item
$|W_1| = |U_1'| + k$ and $|W_2| = |U_2'|$.
\item
$G_B[W_1]$ contains a blue path of length $k$.
\item
The graphs $G_B[U_1', W_1]$ and $G_R[U_2', W_2]$ have minimum degree at least $n / 32$ and all but at most $\theta n$ vertices have degree at least $(1/4 - \theta)n$.
\end{itemize}
\end{claim}

It is easy to conclude from Claim \ref{claim_partition_two_largish_compts} that $G$ may be partitioned into a blue cycle and a red one.
Indeed, the graph $G_R[U_2', W_2]$ is Hamiltonian (e.g.~by Corollary \ref{cor_chvatal_bip}).
We claim that $G_B[U_1' \cup W_1]$ is also Hamiltonian, completing the partition of $V(G)$ into a red cycle and a blue one. Indeed, take any path $P$ in $G_B[W_1]$ of length $k$. Denote its ends by $u, v$ and let $W_1'$ be the set obtained from $W_1$ by removing the inner vertices of $P$.
It is easy to see (e.g.~by Corollary \ref{cor_chvatal_bip}) that $G_B[U_1', W_1']$
has a Hamilton path with ends $u, v$.
It remains to prove Claim \ref{claim_partition_two_largish_compts}.

\begin{proof}[Proof of Claim \ref{claim_partition_two_largish_compts}]
Recall that $(1/4 - \eta)n \le |U_1|, |U_2| \le n/4 - 1$.
Denote by $W'$ the set of vertices which have at least $\sqrt{\eta}n$ non-neighbours in either $U_1$ or $U_2$.
Note that each vertex in $U_1$ has at most $\eta n$ non-neighbours in $W$. It follows that $|W'| \le \sqrt{\eta}n$.

Let $W_1'$ be the set of vertices in $W'$ with (blue) degree at least $n / 16$ into $U_1'$, and let $W_2' = W' \setminus W_1'$. So the vertices in $W_2'$ have degree at least $n / 16$ into $U_2'$.

Note that $e(G_B[W \setminus W']) \ge n^2 / 64$. It follows from Theorem \ref{thm_erdos_gallai} that $W \setminus W'$ contains a blue path $P$ of length $k$.

Pick any partition $\{W_1, W_2\}$ with the following properties.
\begin{itemize}
\item
$|W_1| = |U_1'| + k$ and $|W_2| = |U_2'|$.
\item
$W_1' \cup V(P) \subseteq W_1$ and $W_2' \subseteq W_2$.
\end{itemize}
Such a partition satisfies the required conditions of Claim \ref{claim_partition_two_largish_compts} with $\theta = \sqrt{\eta}$. 
\end{proof}

\end{proof}

\section{Four half-sized robust components}\label{sec_four_midsized}
In this section we consider Case (\ref{itm_struch_only_midsized_compts}) from Lemma \ref{lem_rough_structure}.
\begin{lem} \label{lem_four_midsized}
Let $\frac{1}{n} \ll  \alpha, \frac{1}{k} \ll \eps \ll 1$ and let $G$ be a graph of order $n$ with $\delta(G) \ge 3n / 4$ and a $2$-colouring $E(G) = E(G_B) \cup E(G_R)$.
Suppose that $F_1, F_2, F_3, F_4$ satisfy the following conditions.
\begin{itemize}
\item
$F_1, F_2 \subseteq G_B$ are vertex-disjoint $(\alpha, k)$-robust components on at least $(1/2 - \eps)n$ vertices.
\item
$F_3, F_4 \subseteq G_R$ are vertex-disjoint $(\alpha, k)$-strongly robust components on at least $(1/2 - \eps)n$ vertices.
\item
$V(F_1) \cup V(F_2) = V(F_3) \cup V(F_4)$.
\item
$|V(F_i) \cap V(F_j)| \ge (1/4 - \eps)n$ for $i \in [2], j \in [3, 4]$.
\end{itemize}
Then $V(G)$ may be partitioned into a blue cycle and a red one.
\end{lem}

We follow similar ideas to previous sections.
If there exist four vertex-disjoint paths, two of which are blue and connect $F_1$ with $F_2$ and two are red and connect $F_3$ with $F_4$, then we may continue as in previous sections, by essentially having two robust components, one red and one blue, which both span almost all the vertices.
The main effort in this case goes into showing that if such paths do not exist, the desired partition may be found by Lemma \ref{lem_red_disconnected_2}.

\begin{proof}[Proof of \ref{lem_four_midsized}]
We extend the components $F_i$ as follows.
For every vertex $v$ not in any of the components, if there exist $i\in \{1,2\},j\in \{3,4\}$ such that $v$ sends at least $\alpha n$ blue edges to $F_i$ and at least $\alpha n$ red edges to $F_j$, we add $v$ to $F_i$ and $F_j$. Note that if no such $i,j$ exist, then $v$ either blue degree at least $(3/4 - 3\eps)n$ or red degree at least $(3/4 - 3\eps)n$.

Note that the obtained components satisfy the conditions above (though with relaxed parameters $\alpha, k$ in the definition of robustness and with say $2\eps$ instead of $\eps$). We abuse notation by denoting the modified components by $F_1,F_2,F_3,F_4$.
So in addition to the above conditions, we have that every vertex not in $V(F_1)\cup V(F_2)$ has either blue degree or red degree at least $(3/4 - 3\eps)n$.
 
We claim that one of the following assertions holds.
\begin{enumerate}
\item \label{itm_1_four_mid_sized}
There exist vertex-disjoint paths $P_1, P_2, P_3, P_4$ such that $P_1, P_2$ are blue paths from $F_1$ to $F_2$ and $P_3, P_4$ are red paths from $F_3$ to $F_4$.
\item \label{itm_2_four_mid_sized}
There exist two vertices $u, v$ such that $F_1, F_2$ belong to different connected components of $G_B \setminus \{u, v\}$.
Furthermore, $v$ sends at most $\eps n$ blue edges to either $F_1$ or $F_2$.
\item \label{itm_3_four_mid_sized}
There exist two vertices $u, v$ such that $F_3, F_4$ belong to different connected components of $G_R \setminus \{u, v\}$.
Furthermore, $v$ sends at most $\eps n$ red edges to either $F_3$ or $F_4$.
\end{enumerate}

Condition (\ref{itm_1_four_mid_sized}) implies that we may connect $F_1$ and $F_2$ using $P_1, P_2$ to obtain a large robust blue component, and similarly we may connect $F_3$ and $F_4$ using the paths $P_3, P_4$ to obtain a large strongly robust red component.
We may continue as in Section \ref{sec_large_bip_blue_almost_balanced} to obtain the desired partition of $V(G)$ into a red cycle and a blue one.

If one of Conditions (\ref{itm_2_four_mid_sized}, \ref{itm_3_four_mid_sized}) holds, we may find the desired partition into a red cycle and a blue one by Lemma \ref{lem_red_disconnected_2}.

It remains to prove that indeed, one of the above three cases holds.
We call a vertex \emph{blue} if it sends at least $\alpha n$ blue edges to both $F_1$ and $F_2$. Similarly, a vertex is \emph{red} if it sends at least $\alpha n$ red edges to both $F_3$ and $F_4$.
We show that either one of the above three conditions holds, or there are at least four vertices which are either blue or red.

If true, we conclude that one of the three conditions above holds. 
It is easy to verify that if there exist four vertices, two of them red and two blue, we may find paths as Condition (\ref{itm_1_four_mid_sized}).
Thus we may assume that there is at most one red vertex.
This implies that either Condition (\ref{itm_3_four_mid_sized}) holds, or there are three vertex disjoint red paths $P_1, P_2, P_3$ between $F_3$ and $F_4$. We may assume that the inner vertices of these path are not in $V(F_3) \cup V(F_4)$.
If there exist four blue vertices $u_1, u_2, u_3, u_4$, without loss of generality, $u_1, u_2 \notin V(P_1) \cup V(P_2)$ and we may find two vertex-disjoint blue paths between $F_1$ and $F_2$ which are disjoint of $P_1, P_2$ using the blue vertices $u_1, u_2$ (we can take them to be of length $2$ and be centred at $u_1$ and $u_2$).
If there exist three blue vertices $u_1, u_2, u_3$ and a different red vertex $v$, without loss of generality, the path $P_1$ has length $2$, is centred at $v$ and avoids $u_1, u_2, u_3$.
We may further assume that $P_2$ does not contain $u_1, u_2$.
It follows that we may find two vertex-disjoint blue paths, disjoint of $P_1, P_2$, between $F_1$ and $F_2$.

It remains to show that indeed, if Conditions (\ref{itm_1_four_mid_sized} - \ref{itm_3_four_mid_sized}) do not hold, there are at least four vertices which are either red or blue.
Clearly, each vertex $v \in V(G) \setminus (F_1 \cup F_2)$ is either red or blue.
Let $l$ be the number of these vertices.
Denote 
\begin{align*}
& A_1 = V(F_1) \cap V(F_3) \\
& A_2 = V(F_1) \cap V(F_4) \\ 
& A_3 = V(F_2) \cap V(F_3) \\
& A_4 = V(F_2) \cap V(F_4).
\end{align*}
Consider the following claim.
\begin{claim} \label{claim_red_blue_vertices}
Suppose that $|A_i| \ge n/4 - 1 + k$ where $1 \le k \le 4$. Then either one of the Conditions (\ref{itm_1_four_mid_sized} - \ref{itm_3_four_mid_sized}) from above holds or $A_i$ contains at least $k$ vertices which are either blue or red.
\end{claim}
Using the claim, if the above conditions do not hold and there are at most three vertices which are either blue or red, we have $l \le 3$ and the total number of vertices is at most $4(n/4 - 1) + (3 - l)  + l \le n - 1$, a contradiction. It remains to prove the claim.

\begin{proof}[Proof of Claim \ref{claim_red_blue_vertices}]
Without loss of generality, $i = 1$.
Note that if $G[A_1, A_4]$ has a matching of size $7$, Condition (\ref{itm_1_four_mid_sized}) holds.
Indeed, let $M$ be such a matching. If in such a matching at least two edges are red and at least two are blue, condition (\ref{itm_1_four_mid_sized}) holds, since each edge connects $F_1$ and $F_2$ as well as $F_3$ and $F_4$.
Thus we may assume that at most one edge is red.

Recall that there exist vertex-disjoint red paths $P_1, P_2$ between $F_3$ and $F_4$ (otherwise Condition (\ref{itm_3_four_mid_sized}) holds). We assume that the inner vertices of $P_1$ and $P_2$ do not belong to $V(F_3 \cup F_4) = V(F_1 \cup F_2)$, so each of the paths $P_1$ and $P_2$ intersects at most two edges of $M$.
It follows that there exist two blue edges $e_1, e_2 \in M$ which are disjoint of $P_1$ and $P_2$, and Condition (\ref{itm_1_four_mid_sized}) holds.
It remains to consider the case where $G[A_1, A_4]$ has no matching of size $7$.

We deduce that there is a set $U \subseteq A_1 \cup A_4$ of at most six vertices which intersects each of the edges of $G[A_1, A_4]$.
Note that by the minimum degree conditions, every vertex in $A_4$ has at least $k$ neighbours in $A_1$. Hence, $|U \cap A_1| \ge k$. Denote $A_4' = A_4 \setminus U$. Then
\begin{align*}
e(G[A_1, A_4']) \ge k(|A_4| - 6).
\end{align*}
We conclude that at least $k$ vertices in $U \cap A_1$ have at least $n/25$ neighbours in $A_4$.
Indeed, otherwise we have
\begin{align*}
e(G[A_1, A_4']) \le (k - 1)|A_4| + 6n / 25 < k(|A_4| - 6),
\end{align*}
a contradiction.

Each of the vertices in $A_1$ with at least $n/25$ neighbours in $A_4$ is either red or blue. E.g.~if it sends at least $\alpha n$ blue edges to $A_4$ is blue.\end{proof}
\end{proof}

\section{Large strongly robust blue component} \label{sec_large_blue_non_bip}

In this section, we resolve Case (\ref{itm_struct_spanning_non_bip}) from Lemma \ref{lem_rough_structure}, which is the last remaining case.

\begin{lem}\label{lem_large_blue_non_bip}
Let $\frac{1}{n} \ll \eps, \alpha, \frac{1}{k} \ll 1$ and let $G$ be a graph of order $n$ with $\delta(G) \ge 3n / 4$ and a $2$-colouring $E(G) = E(G_B) \cup E(G_R)$.
Suppose that $F$ is a blue $(\alpha, k)$-strongly robust component on at least $(1 - \eps)n$ vertices.
Then $V(G)$ may be partitioned into a blue cycle and a red cycle.
\end{lem}

In our proof of Lemma \ref{lem_large_blue_non_bip}, we extend $F$ to include all vertices which send a fairly large number of blue edges into $F$, and denote the set of remaining vertices by $Z$.
We consider two cases, according to the size of $Z$.

The case where $Z$ is relatively small turns out to be harder. We apply the Regularity Lemma on $F$ and prove a structural result on the reduced graph, focusing on ways to obtain perfect matchings.
In each of the cases for the structure of the reduced graph, we can
partition almost all of the vertices into a red cycle and a blue one.
The vertices of $Z$ may be covered using their large degree and the leftover vertices of $F$ may be absorbed as usual.

In the case where $Z$ is large, the reduced graph can easily be seen to have a perfect matching consisting of a connected blue matching and a connected red matching.
We have to be slightly more careful than usual when obtaining cycles from the connected matching so as to cover $Z$. The leftover vertices of $F$ can be absorbed as usual.

\begin{proof} [Proof of Lemma \ref{lem_large_blue_non_bip}]

Let $F_1$ be the graph obtained by adding to $F$ the vertices in $G$ with at least $\alpha n$ blue neighbours in $F$.
By Lemma \ref{lem_robust_adding_vertices}, $F_1$ is $(\alpha^3/2, k+2)$ strongly robust.
Denote $Z = V(G)\setminus V(F_1)$. Note that 
$$\deg_R(z) \ge (3/4 - \eps - \alpha)n\ge(3/4 - 2\eps)n \text{ for every }z \in Z.$$
It follows that every two vertices in $Z$ have at least $(1/2 - 4\eps)n$ red neighbours in common.
We consider two cases according to the size of $Z$.
Pick $\beta > 0$ by Lemma \ref{lem_robust_removing_vertices} so that $F_1$ remains $(\alpha^3/4, k + 2)$ strongly robust after removing at most $\beta n$ vertices.

\subsection*{Case 1: \normalfont \normalsize $|Z|\ge \beta n$}

Let $Q$ be a $\rho^2 n$-absorbing path in $F_1$ of length at most $\rho n$.
Pick a suitably small $\eta$, and let $G'$ be the subgraph obtained by applying Lemma \ref{lem_regularity} with the graph $G \setminus V(Q)$, the partition $\{V(F_1),Z\}$ and parameters $\eta$ and $d = 6\eta$.
Let $\Gamma$ be the corresponding reduced graph.
As usual, $\delta(\Gamma) \ge (3/4 - 2\eps)m$, where $m=|\Gamma|$.

Let $\Phi_1$ be the blue subgraph spanned by the clusters which are contained in $V(F_1)$. Then $|\Phi_1| \ge (1 - 2\eps)m$ and $\Phi_1$ is connected by Lemma \ref{lem_robust_regularity}. The vertices in clusters  contained in $Z$ have red degree at least $(3/4 - 4\eps)m$ in $\Gamma$ (note that since $|Z|\ge \beta  n \ge 2\eta n$, there are such clusters). In particular, $\Gamma$ contains a red component $\Phi_2$ of order at least $(3/4 - 4\eps)m$.
It is easy to check that $\Phi_1 \cup \Phi_2$ has a perfect matching, e.g.~by Theorem \ref{thm_chvatal}.

Let $U_1$ be the set of vertices in $G'$ belonging to clusters of the blue matching and let $U_2$ be the set of vertices in $G'$ belonging to clusters of the red matching.
We obtain the required partition into a blue cycle and a red cycle as follows.

Suppose first that $|U_2|\le 3n/8$.
Fix two vertices $z_1,z_2 \in Z$.
Let $W$ be a set of vertices in $G'$ containing $3\eta$ of the vertices of each cluster in $V(F_1)$.
Clearly, $z_1,z_2$ each have many neighbours in the set of vertices belonging to the clusters of $\Phi_2$. Thus, by Lemma \ref{lem_connected_matching}, there is a red path $P_1$ in $G'$ between $z_1$ and $z_2$ spanning at least $(1 - 6\eta)$ of the vertices of $U_2 \setminus W$ and using at most $m^2$ other vertices.

Denote by $Z'$ the set of vertices in $Z$ which are not covered by this path. Note that $|Z'| \le 9\eta n$. Furthermore, every two vertices of $Z'\cup \{z_1,z_2\}$ have at least $n/16$ common red neighbours in $V(G') \setminus U_2$.
It is thus possible to find a path in $G' \setminus (V(P_1) \cup U_1)$ between $z_1$ and $z_2$, containing $Z$ and using at most $200 \eta$ of the vertices of each cluster.
Indeed, such a path can be constructed greedily. Suppose $Z' = \{z_3, \ldots, z_t\}$.
For $i = 2, \ldots, t$ we find a common red neighbour of $z_i, z_{i+1}$ (where subscripts are taken modulo $t$) which was not used before and which does not belong to a cluster in $U_2$ with at least $200\eta$ of its vertices already used. Note that this would give the required red path, completing $P_1$ to a red cycle $C$.
In each step, there are at least $n/16$ possible choices, out of which at most $9\eta n$ were already used, and at most $9\eta n/200\eta < n/20$ belong to clusters for which at least $200 \eta$ of its vertices are used. Thus it is possible to choose a suitable vertex.

We now construct a blue cycle which is disjoint of $C$, contains $Q$ and misses at most $210\eta$ of the vertices vertices of each cluster in $U_1$, using Lemma \ref{lem_connected_matching} (we use the vertices of $W$ to connect cluster pairs, as described after the statement of Lemma \ref{lem_connected_matching}).
The missing vertices may be absorbed by $Q$, completing the desired cycle partition.

Suppose now that $|U_2| \ge 3n/8$.
By Lemma \ref{lem_connected_matching}, there exist a blue cycle $C_1$ and a red cycle $C_2$ which are disjoint and $C_i$ covers all but at most $9 \eta$ of the vertices of $U_i$ for $i \in [2]$.
In particular, the red cycle has length at least $5n/16$.
Let $Z'$ be the set of vertices of $Z$ which are not covered by either of the cycles.
We show how to obtain a red cycle $C_2'$ such that $Z \subseteq V(C_2')\subseteq V(C_2) \cup Z$ and $|V(C_2') \setminus V(C_2)| \le 60 \eta n$.
To that end, we claim that the vertices of $Z'$ can be inserted one by one, such that in each stage at most $20$ of the original vertices of $C$ are removed, none of them from $Z$. If $z$ cannot be inserted as explained, the number of red neighbours of $z$ in the cycle is at most $40|Z| + n/20$.
But every vertex $z \in Z'$ has at least $(1 / 16 - 3\eps)n$ red neighbours in the cycle obtained from $C_2$, as long as it has length at least $(5 / 16 - 60\eta)n$, implying that $z$ may indeed be inserted.
There are at most $20|Z'| + 9\eta n \le \rho^2 n$ vertices missing from $V(C_1)\cup V(C_2')$, all of them from $V(F_1)$. They can be absorbed by $Q$.

\subsection*{Case 2: \normalfont \normalsize $|Z|\le \beta n$}
If $|Z|\ge 3$, fix three vertices $z_1, z_2, z_3\in Z$.
Let $P$ be a red path containing $Z\setminus\{z_1, z_2, z_3\}$ (but avoiding $z_1, z_2, z_3$) with at most $\beta n$ vertices of $F_1$ (note that $P$ may be constructed greedily).
If $|Z|\le 3$, we take $P$ to be empty.
Let $F_2 = F_1\setminus V(P)$.
$F_2$ is $(\alpha/4, k+2)$ strongly robust by the choice of $\beta$.
Apply Lemma \ref{lem_absorbing_paths} to obtain a $\rho^2 n$-absorbing path $Q$ in $F_2$ of length at most $\rho n$, and denote $F_3 = F_2\setminus V(Q)$. 
$F_3$ is $(\alpha^3/8, k+2)$ strongly robust, by Lemma \ref{lem_robust_removing_vertices}.
Furthermore, $|F_3|\ge (1 - 2\eps)n$ and the vertices of $Z$ have at least $(3/4 - 4\eps)n$ red neighbours in $F_3$.

Apply the Regularity Lemma, Lemma \ref{lem_regularity}, to the graph $G[V(F_3)]$ with a suitably small parameter $\eta$.  Let $\Gamma$ be the corresponding $(\eta, 4\eta)$-reduced graph.
We have $\delta(\Gamma) \ge (3/4 - 2\eps)m$ and by Lemma \ref{lem_robust_regularity}, $\Gamma_B$ is connected.

We shall use the following proposition.
\begin{prop} \label{prop_large_blue_cases_matching}
One of the following assertions holds.
\begin{enumerate}
\item
$\Gamma_B$ has a perfect matching.
\item
There exists a red component $\Phi$ on at least $(1/2 - 50\eps)m$ vertices such that $\Gamma_B \cup \Phi$ has a perfect matching.
\item
There exist disjoint subsets $X_1, X_2 \subseteq V(G)$ of size at least $(1/4 - 500\eps)m$ such that
\begin{itemize}
\renewcommand\labelitemi{--}
\item
$X_1 \cup X_2$ is independent in $\Gamma_B$.
\item
$\Gamma_R[X_i]$ is connected and $\Gamma_B \cup \Gamma_R[X_i]$ has a perfect matching for $i \in [2]$.
\end{itemize}
\item
There exist disjoint subsets $X_1, X_2, Y_1, Y_2$ of size at least $(1/4 - 500\eps)m$ such that 
\begin{itemize}
\renewcommand\labelitemi{--}
\item
$\Gamma_R[X_i \cup Y_i]$ is connected and non-bipartite for $i \in [2]$.
\item
$\Gamma_B[X_1 \cup X_2]$ and $\Gamma_B[Y_1 \cup Y_2]$ are connected.
\end{itemize}
\end{enumerate}
\end{prop}

Before proving Proposition \ref{prop_large_blue_cases_matching}, we show how to complete the proof of Lemma \ref{lem_large_blue_non_bip} in this case using the proposition.

\subsubsection*{$\Gamma_B$ has a perfect matching}
Suppose that $\Gamma_B$ has a perfect matching.
Recall that $\Gamma_B$ is connected, so this matching is connected.
We complete $Z \cup V(P)$ to a red cycle $C_1$ using at most four additional vertices of $F_3$.
By Lemma \ref{lem_connected_matching}, there exists a blue cycle $C_2$, disjoint of $C_1$ which extends the absorbing path $Q$ and contains all but at most $6\eta n \le \rho^2 n$ vertices of $F_3$. The remaining vertices can be absorbed by $Q$ to obtain a blue cycle $C_2'$.
The cycles $C_1,C_2'$ form the required cycle partition.

\subsubsection*{An almost half-sized red component whose union with $\Gamma_B$ has a perfect matching}
Suppose that $\Phi$ is a red component of size at least $(1/2 - 50\eps)m$ and $\Gamma_B \cup \Phi$ has a perfect matching.
Define a red path $P'$ as follows.
If $|Z| \ge 2$, extend $P$ to a path $P'$ containing $Z$ with ends $z_1, z_2 \in Z$ (we may do this using at most three additional vertices).
If $|Z| = 1$, take $P'$ to be the path $(z)$ where $Z = \{z\}$.
If $Z = \emptyset$, take $P'$ to be the empty path.
Note that every vertex of $Z$ sends many red edges to the set of vertices contained in the clusters defined by $V(\Phi)$.
By Lemma \ref{lem_connected_matching}, there exist vertex-disjoint cycles $C_1, C_2$ such that $C_1$ is blue and contains the absorbing path $Q$ and $C_2$ is red and contains the path $P'$.
Furthermore, the cycles $C_1, C_2$ cover all but at most $\rho^2 n$ vertices belonging to $F_3$, which may be absorbed by $Q$, completing the desired partition into a red cycle and a blue one.

\subsubsection*{Two almost quarter-sized red components whose union with $\Gamma_B$ has a perfect matching}
Let $X_1, X_2 \subseteq V(\Gamma)$ be disjoint sets of size at least $(1/4 - 500\eps)m$ satisfying the following conditions.
\begin{itemize}
\item
$X_1 \cup X_2$ is independent in $\Gamma_B$.
\item
$\Gamma_R[X_i]$ is connected and $\Gamma_B \cup \Gamma_R[X_i]$ has a connected matching for $i \in [2]$.
\end{itemize}
Note that we may assume that $\Gamma_R[X_1 \cup X_2]$ is not connected, since otherwise we may proceed as in the previous case.

Let $U_i$ be the set of vertices contained in the clusters in $X_i$.
Note that $|U_i| \ge  (1/4 - 501\eps)n$.
We define a path $P'$ as follows.
If $|Z| \ge 3$, without loss of generality, $z_1, z_2$ send at least $2\eta n$ red edges into $U_1$. Extend $P$ to a path $P'$ containing $Z$ with ends $z_1, z_2$ (using at most three additional vertices).
If $|Z| = 1$, denote $Z = \{z\}$ and take $P' = (z)$, and suppose without loss of generality that $z$ has at least $2 \eta n$ red neighbours in $U_1$.
If $Z = \emptyset$, take $P'$ to be the empty path.
As before, since $\Gamma_B \cup \Gamma_R[X_1]$ has a perfect matching, Lemma \ref{lem_connected_matching} implies that there exist disjoint cycles $C_1, C_2$ such that $C_1$ is blue and contains $Q$ and $C_2$ is red and contains $P'$, and together they cover all but at most $\rho ^2 n$ vertices of $F_3$, which may be absorbed by $C_1$.

It remains to consider the case where $|Z| = 2$. Denote $Z = \{z_1, z_2\}$.
If for some $i \in [2]$, both $z_1, z_2$ have at least $2 \eta n$ red neighbours in $U_i$, we may continue as before by taking $P'$ to be a path of length $2$ connecting $z_1, z_2$.
Thus we assume that $$\deg_R(z_1, U_2), \deg_R(z_2,U_1) \le 2\eta n \le \eps n.$$
Recall that $\deg_R(z_i) \ge (3/4 - 4\eps)n$ for $i \in [2]$.
It follows that $|U_1|, |U_2| \le (1/4 + 5\eps)n$.
Let $W$ be the set of common red neighbours of $z_1$ and $z_2$ in $V(G) \setminus (U_1 \cup U_2)$.
Then $|W| \ge (1/2 - 20\eps)n$.

Suppose that there exists a vertex $w_1 \in W$ with at least $4\eta n$ red edges into $U_1 \cup U_2$. Without loss of generality, $w_1$ has at least $2\eta n$ red neighbours in $U_1$. Take $P' = (w_1 z_2 w_2 z_1)$ for some $w_2 \in W$ and continue as before (when $|Z| \neq 2$) to conclude that the desired partition exists.

We now assume that every vertex in $W$ has at most $\eps n$ red neighbours in $U_1 \cup U_2$.
Furthermore, by the definition of the reduced graph, every vertex in $U_1 \cup U_2$ has at most $9\eta n \le \eps n$ red neighbours in $W$.
Note that since $e(\Gamma[X_1, X_2]) = 0$, the graph $G_B[U_1 \cup U_2, W]$ is almost complete.
It follows that we may apply Lemma \ref{lem_red_graph_disconnected_1} with parameter $1002 \eps$, to conclude that $V(G)$ may be partitioned into a red cycle and a blue one. 

\subsubsection*{Four half-sized monochromatic components}
Suppose that $X_1, X_2, Y_1, Y_2$ are disjoint sets of size at least $(1/4 - 500\eps)m$ with the following properties.
\begin{itemize}
\item
$\Gamma_R[X_i \cup Y_i]$ is connected and non-bipartite for $i \in [2]$.
\item
$\Gamma_B[X_1 \cup X_2]$ and $\Gamma_B[Y_1 \cup Y_2]$ are connected.
\end{itemize}
We conclude from Lemma \ref{lem_many_reduced_components} that there exist sets $U_1, U_2, W_1, W_2 \subseteq V(G)$ of order at least $(1/2 - \theta)n$, where $\theta = 502\eps$, such that
\begin{itemize}
\item
$U_1, U_2$ are disjoint, $W_1, W_2$ are disjoint and $U_1 \cup U_2 = W_1 \cup W_2$.
\item
$|U_i \cap W_j| \ge (1/4 - \theta)n$ for $i, j \in [2]$.
\item
$G_B[U_i]$ is $(\gamma, l)$-robust and $G_R[W_i]$ is $(\gamma, l)$-strongly robust for $i \in [2]$, where $\gamma = \gamma(\eps, \alpha, k)$ and $l = l(\eps, \alpha, k)$.
\end{itemize}
By Lemma \ref{lem_four_midsized}, $V(G)$ may be partitioned into a red cycle and a blue one.
This completes the proof of Lemma \ref{lem_large_blue_non_bip}, under the assumption that Proposition \ref{prop_large_blue_cases_matching} holds. 
We prove it in the following subsection.

\end{proof}

\subsection*{Proof of Proposition \ref{prop_large_blue_cases_matching}}
In this subsection, we prove Proposition \ref{prop_large_blue_cases_matching}.
We shall consider four cases according to the sizes of the red components in $\Gamma$.
In each of these cases, we apply Lemmas \ref{lem_matching_bip_technical} or \ref{lem_matching_tripartite_density} to gain additional structural information about the graph $\Gamma$ in case $\Gamma_B$ has no perfect matching (otherwise we are done).
This information will enable us to show that one of the conditions in the proposition holds.

\begin{proof}[Proof of Proposition \ref{prop_large_blue_cases_matching}]
The following claim reduced the proof to the following four cases.
\begin{claim} \label{claim_prop_large_blue_cases_matching}
One of the following conditions holds.
\begin{enumerate}
\item \label{itm_prop_no_large_red}
$\Gamma_R$ has no components of order at least $(1/4 - 4\eps)m$.
\item \label{itm_prop_half-size_red}
$\Gamma_R$ contains a component of order at least $(1/2 - 50\eps)m$.
\item \label{itm_prop_tripartition}
There exists a tripartition $\{X_1, X_2, X_3\}$ of $V(\Gamma)$ such that $|X_i| \le (1/2 - 24\eps)m$ and every red component is contained in one of the red sets $X_i$.
\item \label{itm_prop_four_quarter-sized}
$\Gamma_R$ contains four vertex-disjoint components of order at least $(1/4 - 36\eps)m$.
\end{enumerate}
\end{claim}

\begin{proof}
Let $s_1 \ge \ldots \ge s_l$ be the orders of components in $\Gamma_R$.
We may assume that $(1/4 - 4\eps)m \le s_1 \le (1/2 - 50\eps)m$, otherwise one of the first two conditions holds.

If $s_1 + s_2 \ge (1/2 + 24\eps)m$, we take the tripartition of $V(\Gamma)$ corresponding to the partition $\{[1], [2], [3, \ldots, l]\}$ of $[s]$.
Thus we can assume that $s_1 + s_2 \le (1/2 + 24\eps)n$.

If $s_3 \le 50\eps m$, we may find the desired partition as follows.
Note that $s_2 \le (1/4 + 12\eps)m$ and pick $i$ to be the minimal such that $s_2 + \ldots + s_i > 3m/8$.
Since $s_i \le s_3 \le 50\eps m$, we have $s_2 + \ldots + s_i \le (3/8 + 50\eps)m$ and $s_{i + 1} + \ldots + s_l \le (3/8 - 54\eps)m$.
It follows that the partition $\{[1], [2, i], [i + 1, l]\}$ satisfies the requirements.
Thus, we may assume that $s_3 \ge 50\eps m$.

If $s_1 + s_2 \le (1/2 - 24\eps)m$, we have $s_1 + s_2 \ge (1/4 + 46\eps)m$.
Let $i$ be minimal such that $s_3 + \ldots + s_i \ge (1/4 - 12\eps)m$. Since $s_i \le s_2 \le (1/4 - 12\eps)m$, we have $s_3 + \ldots + s_i \le (1/2 - 24\eps)m$.
The partition $\{[2], [3, i], [i + 1, l]\}$ satisfies the requirements.
We may now assume that $(1/2 - 24\eps)m \le s_1 + s_2 \le (1/2 + 24\eps)m$. 

Since $s_3 \ge 50\eps m$, we have $s_1 + s_2 + s_3 \ge (1/2 + 24\eps)m$.
If $s_2 + s_3 \le (1/2 - 24\eps)m$, the partition $\{[1], [2, 3], [4, l]\}$ satisfies the requirements.
Otherwise, $s_2 + s_3 \ge (1/2 - 24\eps)m$. It follows that $s_1 \ge s_2 \ge (1/4 - 12\eps)m$, hence $s_1 \le (1/4 + 36\eps)m$ and $s_3 \ge (1/4 - 60\eps)m$.
If $s_3 + s_4 \le (1/2 - 24\eps)m$, we may take the partition $\{[1], [2] \cup [5, l], [3,4]\}$.
Otherwise, we have $s_4 \ge (1/4 - 36\eps)m$, and the required four components exist.

\end{proof}
We prove Proposition \ref{prop_large_blue_cases_matching} in each of the four cases in Claim \ref{claim_prop_large_blue_cases_matching}.
The second case, where there is a large red component, turns out to be hardest and we leave it to the end of the proof.

\subsubsection*{No large red components}
In Case (\ref{itm_prop_no_large_red}) of the previous claim, we have $\deg_B(\Gamma) \ge m/2$, implying that $\Gamma_B$ has a perfect matching.

\subsubsection*{Tripartition}
In Case (\ref{itm_prop_tripartition}) of Claim \ref{claim_prop_large_blue_cases_matching}, there exists a tripartition $\{X_1, X_2, X_3\}$ of $V(\Gamma)$ such that $|X_i| \le (1/2 - 24\eps)m$ and every red component is contained in one of the red sets $X_i$.
We assume that $\Gamma_B$ contains no perfect matching.
By Lemma \ref{lem_matching_tripartite_density}, without loss of generality, there exist subsets $Y_1 \subseteq X_1, Y_2 \subseteq X_2$ such that $|Y_1|, |Y_2| \ge (1/4 - 10\eps)m$ and $Y_1 \cup Y_2$ is independent in $\Gamma_B$.

It follows that 
\begin{equation} \label{eqn_Y_i_large_min_deg}
\delta(\Gamma_R[Y_i]) \ge |Y_i| - 12\eps n.
\end{equation}
In particular, $\Gamma_R[Y_i]$ is connected.
We show that $\Gamma_B \cup \Gamma_R[Y_i]$ has a perfect matching for $i \in [2]$.
Suppose to the contrary that $\Phi = \Gamma_B \cup \Gamma_R[Y_1]$ has no perfect matching.
By Lemma \ref{lem_matching_tripartite_density}, it follows that there exist subsets $Z_i \subseteq X_i$ and $Z_j \subseteq X_j$ of size at least $(1/4 - 10 \eps)n$ such that $Z_1 \cup Z_2$ is independent in $\Phi$, for some $1 \le i \neq j \le 3$.

Suppose first that say $i = 2$.
Then the intersection of $Y_2$ and $Z_2$ is non-empty. 
Let $u \in Y_2 \cup Z_2$.
Note that $u$ has at most $(1/4 + 2\eps)m$ non-neighbours in $X_1 \cup X_3$, at least $(1/4 - 10\eps)m$ of which belong to $Y_1$, so $|Z_1 \cap Y_1| \ge (1/4 - 22 \eps)m$, a contradiction to (\ref{eqn_Y_i_large_min_deg}).

It remains to consider the case where $i = 1$, $j = 3$.
As before, we may pick $u \in Y_1 \cap Z_1$. But $u$ has at least $|Z_3| + |Y_2| \ge (1/2 - 20\eps)m$ non-neighbours in $\Gamma$, a contradiction.

\subsubsection*{Almost equipartition into four parts}
In Case (\ref{itm_prop_four_quarter-sized}) of \ref{claim_prop_large_blue_cases_matching}, $\{X_1, X_2, X_3, X_4\}$ is a partition of $V(G)$ such that every red component is contained in one of the parts $X_i$ and $|X_i| \ge (1/4 - 36\eps)m$. The following claim finishes the proof of Proposition \ref{prop_large_blue_cases_matching} in this case.
\begin{claim}
One of the following assertions holds.
\begin{enumerate}
\item
$\Gamma_B$ has a perfect matching.
\item
For some $1 \le i < j \le 4$, there exist subsets $Y_i \subseteq X_i$ and $Y_j \subseteq X_j$ of size at least $(1/4 - 500\eps)m$ such that
\begin{itemize}
\renewcommand\labelitemi{--}
\item
$\Gamma_R[Y_l]$ is connected for $l \in \{i, j\}$.
\item
$Y_i \cup Y_j$ is independent in $\Gamma_B$.
\item 
$\Gamma_B \cup \Gamma_R[Y_l]$ has a perfect matching for $l \in \{i, j\}$.
\end{itemize} 
\end{enumerate}
\end{claim}
\begin{proof}
Consider Lemma \ref{lem_matching_bip_technical} with the graph $\Gamma_B$, the bipartition $\{X_1 \cup X_2, X_3 \cup X_4\}$ and parameter $72 \eps$.
Assuming that $\Gamma_B$ has no perfect matching, it is easy to verify that the first three conditions cannot hold. It follows that $\Gamma_B$ has an independent set $Y$ of size at least $(1/2 - 432\eps)m$.
This implies that $\delta(\Gamma_R[Y]) \ge (1/4 - 434\eps)m$, implying that $\Gamma_R[Y]$ is the union of two red connected subgraphs of order at least $(1/4 - 434\eps)m$.
Without loss of generality, $Y \subseteq X_1 \cup X_2$. Denote $Y_i = Y \cap X_i$.
If $|X_1| + |X_2| \ge |X_3| + |X_4|$, the graph $\Gamma_B \cup \Gamma_R[Y_i]$ has a perfect matching for $i \in [2]$, by Lemma \ref{lem_matching_bip_technical}.
Otherwise, $\Gamma_B[X_1 \cup X_2, X_3 \cup X_4]$ is almost complete, and by Lemma \ref{lem_matching_bip_technical}, if $\Gamma_B$ has no perfect matching, $X_3 \cup X_4$ contains an independent set $Y$ of size at least $(1/2 - 432\eps)m$.
Denoting $Y_i = Y \cap X_i$ for $i \in [3, 4]$, we conclude as before that $|Y_i| \ge (1/4 - 434\eps)m$,  $\Gamma_R[Y_i]$ is connected and $\Gamma_B \cup \Gamma_R[Y_i]$ has a perfect matching for $i \in [3, 4]$.
\end{proof}

\subsubsection*{Large red component}
We now consider Case (\ref{itm_prop_half-size_red}) of Claim \ref{claim_prop_large_blue_cases_matching}, which the last remaining case.
Let $\Phi_1$ be a red component on at least $(1/2 - 50\eps)m$ vertices and denote $X_1 = V(\Phi_1)$, $X_2 = V(\Gamma) \setminus X_1$ and $\Gamma_1 = \Gamma_B \cup \Phi_1$.

The following claim is a simple application of Theorem \ref{thm_chvatal}.
\begin{claim}\label{claim_prop_large_red}
One of the following assertions holds.
\begin{enumerate}
\item
$\Gamma_1$ has a perfect matching.
\item
$X_2$ contains a red component of order at least $(1/2 - 4\eps)m$.
\item
$|X_2| \ge m/2$, and $\Gamma_R[X_2]$ contains two red components of order at least $(1/4 - 52\eps)m$.
\end{enumerate}
\end{claim}
\begin{proof}
Suppose that $\Gamma_1$ does not have a perfect matching.
In particular, $\Gamma_1$ has no Hamilton cycle.
Note that
\begin{align*}
\deg_{\Gamma_1}(u) \ge
\left\{
\begin{array}{ll}
(3/4 - 2\eps)m & u \in X_1 \\
(3/4 - 2\eps)m - (n - |X_1|) & u \in X_2.
\end{array}
\right.
\end{align*}
It follows from Theorem \ref{thm_chvatal} that either there are at least $m/2$ vertices of degree at most $(1/2 + 50\eps)m$ or there are at least $(1/4 + 2\eps)m$ vertices of degree at most $(1/4 + 2\eps)m$.

Suppose that the former holds and let $Y$ be the set of vertices of degree at most $(1/2 + 50\eps)m$ in $\Gamma_1$.
Then $\deg_R(u, X_2) \ge (1/4 - 52\eps)m$ for every $u \in Y$.
It follows that either $\Gamma_R[Y]$ is connected, or it consists of two connected components of order at least $(1/4 - 52\eps)m$.

Suppose now that the latter holds, and let $Y$ be the set of vertices of degree at most $(1/4 + 2\eps)m$ in $\Gamma_1$.
Then $\deg_R(u, X_2) \ge (1/2 - 4\eps)m$ for every $u \in Y$, and in particular, $X_2$ contains a red component of order at least $(1/2 - 4\eps)m$.

\end{proof}

In order to prove Proposition \ref{prop_large_blue_cases_matching}, we may assume that $\Gamma_1$ has no perfect matching.
We consider first the third case of Claim \ref{claim_prop_large_red}, where $|X_2| \ge m/2$ and $X_2$ contains two red components of size at least $(1/4 - 52\eps)m$.
Denote by $Y_1$ and $Y_2$ the vertex sets of these red component.
Since $\Gamma_1 = \Gamma_B \cup \Gamma_R[X_1]$ has no perfect matching, it follows from Lemma \ref{lem_matching_bip_technical} that $X_2$ contains an independent set of size at least $(1/2 - 50\eps)m$, implying that $\Gamma[Y_1, Y_2]$ is almost empty and $\Gamma[X_1, X_2]$ is almost complete.
It is easy to conclude from Lemma \ref{lem_matching_bip_technical} that $\Gamma_B \cup \Gamma_R[Y_i]$ has a perfect matching for $i \in [2]$.

In order to complete the proof of Proposition \ref{prop_large_blue_cases_matching} in this case, it remains to consider the case where $\Gamma_1$ has no perfect matching and $X_2$ contains a red component of order at least $(1/2 - 4\eps)m$.
The following claim completes the proof of Proposition \ref{prop_large_blue_cases_matching}.

\begin{claim} \label{claim_large_blue_two_midsized_red}
One of the following conditions holds.
\begin{enumerate}
\item
Either $\Gamma_1$ or $\Gamma_2$ has a perfect matching.
\item \label{itm_2_claim_large_blue_two_midsized_red}
There exist disjoint sets $A_i, B_i \subseteq X_i$ of size at least $(1/4 - 500\eps)m$ such that 
\begin{itemize}
\renewcommand\labelitemi{--}
\item
$\Gamma_R[A_i \cup B_i]$ is connected and non bipartite for $i \in [2]$.
\item 
$\Gamma_B[A_i \cup B_{3 - i}]$ is connected for $i \in [2]$
\end{itemize}
\end{enumerate}
\end{claim}

\begin{proof}
Without loss of generality, we assume that $|X_1| \ge |X_2|$.
We assume that both $\Gamma_1$ and $\Gamma_2$ have no perfect matchings and make use of Lemma \ref{lem_matching_bip_technical} (with parameter $50 \eps$).
We start by considering $\Gamma_1$.
It is easy to verify that $\Gamma_1$ is $2$-connected.
Furthermore, since the vertices of $X_1$ have degree at least $(3/4 - 2\eps)$, it follows that the two last conditions cannot hold (because they imply that there are many vertices in $X_1$ of degree at most approximately $m/2$).
We conclude that the following holds.
\begin{align*}
& \text{There exists an independent set $A_2 \subseteq X_2$ of size at least $(1/4 - 200\eps)m$ such that} \\
& |N_{\Gamma_1}(A_2)| \le (1/4 + 150\eps)m.
\end{align*}

Similarly, if $\Gamma_2$ has no perfect matching, one of the two following conditions holds.
\begin{enumerate}
\item \label{itm_1_large_blue_mid_sized_red}
$X_1$ contains an independent set $Y_1$ of size at least $(1/2 - 50\eps)m$.
\item \label{itm_2_large_blue_mid_sized_red}
There exists an independent set $A_1 \subseteq X_1$ of size at least $(1/4 - 200\eps)m$ such that $|N_{\Gamma_2}(A_1)| \le (1/4 + 150\eps)m$.
\end{enumerate}

Suppose that Condition (\ref{itm_1_large_blue_mid_sized_red}) above holds.
Then every vertex in $A_2$ has at least $(1/2 - 152\eps)m$ red neighbours (in $\Gamma$) in $X_2$.
Denote 
\begin{align*}
& B_1 = Y_1 \cap N_{\Gamma_1}(A_2) \\
& A_1 = Y_1 \setminus N_{\Gamma_1}(A_2) \\
& Y_2 = N_{\Gamma_R}(A_2) \\
& B_2 = Y_2 \cap N_{\Gamma_2}(A_1)
\end{align*}
It is easy to verify that Condition \ref{itm_2_claim_large_blue_two_midsized_red} from Claim \ref{claim_large_blue_two_midsized_red} holds, completing the proof in this case.

We now suppose that Condition (\ref{itm_2_large_blue_mid_sized_red}) holds.
So we have sets $A_i \subseteq X_i$ of size at least $(1/4 - 200\eps)m$ such that $|N_{\Gamma_{3 - i}}(A_i)| \le (1/4 + 150\eps)m$. We write $N(A_i)$ for $N_{\Gamma_{3 - i}}(A_i)$ as a shorthand.

If $A_1 \nsubseteq N(A_2)$, denote $A_1' = A_1 \setminus N(A_2)$ and pick $u \in A_1'$. Then $N(u, X_2) \subseteq N_{\Gamma_2}(A_1')$ and $N(u, X_2) \cap A_2 = \emptyset$.
Since $\deg(u, X_2) \ge (1/4 - 6\eps)$, we conclude $|N_{\Gamma_2}(A_1) \setminus A_2| \ge (1/4 - 6\eps)m$.
In particular, $A_2 \nsubseteq N(A_1)$, so by the same argument, $|N_{\Gamma_1}(A_2') \setminus A_1| \ge (1/4 - 5\eps)m$, where $A_2' = A_2 \setminus N(A_1)$.
Define 
\begin{align*}
B_i = N_{\Gamma_i}(A_{3 - i}') \cap N_{\Gamma_R}(A_i').
\end{align*}
It is easy to check that Condition (\ref{itm_2_claim_large_blue_two_midsized_red}) in Claim \ref{claim_large_blue_two_midsized_red} holds.

Finally, if $A_1 \subseteq N(A_2)$ and $A_2 \subseteq N(A_1)$ denote $B_i = X_i \setminus (N(A_1) \cup N(A_2))$.
There are no edges in $\Gamma$ between $A_1$ and $B_2$ or between $A_2$ and $B_1$.
It is easy to conclude from here that Condition (\ref{itm_2_claim_large_blue_two_midsized_red}) in Claim \ref{claim_large_blue_two_midsized_red} holds.
\end{proof}

\end{proof}
The proof of Lemma \ref{lem_large_blue_non_bip} concludes our proof of Theorem \ref{thm_main}.
We remind the reader that Lemmas \ref{lem_red_graph_disconnected_1} and \ref{lem_red_disconnected_2} were used several times in the proof, and we have yet to proved them.
The next two sections, Sections \ref{sec_red_graph_disconnected} and \ref{sec_red_disconnected_2} are devoted to the proofs of Lemmas \ref{lem_red_graph_disconnected_1} and \ref{lem_red_disconnected_2} respectively.

\section{Proof of Lemma \ref{lem_red_graph_disconnected_1}} \label{sec_red_graph_disconnected}
In this section, we prove Lemma \ref{lem_red_graph_disconnected_1}. Before turning to the proof, we remind the reader of the statement.

\begin{lem*} [\ref{lem_red_graph_disconnected_1}]
Let $\frac{1}{n} \ll \eps \ll 1$ and let $G$ be a graph on $n$ vertices with $\delta(G)\ge 3n/4$ and a $2$-colouring $E(G) = E(G_B) \cup E(G_R)$.
Suppose that $S, T \subseteq V(G)$ satisfy the following conditions.
\begin{itemize}
\item
$S,T$ are disjoint and $|S|, |T| \ge (1/2-\eps)n$.
\item
$\delta(G_B[S,T]) \ge (1/4 - \eps) n$.
\item
For every $S'\subseteq S, T' \subseteq T$ with $|S'|, |T'| \ge (1/4 - 100\eps) n$, we have $e(G_B[S', T']) \ge 25 \eps n^2$.
\end{itemize}
Then $V(G)$ may be partitioned into a red cycle and a blue one.
\end{lem*}

The main tool we use in the proof is Lemma \ref{lem_stability_dirac}, which is a stability version of a special case of Chv\'atal's theorem, Theorem \ref{thm_chvatal}.
Our aim would be to find a short red cycle $C$ and a short blue path $P$, whose removal from $G$ leaves a balanced bipartite graph.
We then apply Lemma \ref{lem_stability_dirac} to deduce that $P$ may be extended to a blue cycle with vertex set $V(G) \setminus V(C)$.

\begin{proof}[Proof of Lemma \ref{lem_red_graph_disconnected_1}]
We start by modifying the sets $S,T$ as follows.
\begin{align*}
&S_1 = S \cup \{v \in V(G) \setminus (S \cup T) : \deg_B(v, T) \ge 24\eps n \}. \\
&T_1 = S \cup \{v \in V(G) \setminus (S_1 \cup T) : \deg_B(v, S) \ge 24\eps n \}.\\
&X = V(G) \setminus (S_1 \cup T_1).
\end{align*}
\begin{rem}
The vertices in $X$ have red degree at least $(3/4 - 50 \eps)n$ in $G$.
\end{rem}

We will find a red cycle and a blue path with one end in $S$ and one in $T$, which are disjoint, cover $X$ and their removal from $G$ leaves a balanced bipartite graph with a large number of vertices. We will then use the following claim to obtain the required partition into a blue cycle and a red one.

\begin{claim} \label{claim_leftover_blue_grpah_ham}
Let $S' \subseteq S_1, T' \subseteq T_1$ be such that $|S'| = |T'| \ge (1/2 - 12\eps)n$.
Then $G_B[S', T']$ is Hamiltonian. Furthermore, for every $s \in S'$ and $t \in T'$, the graph $G_B[S',T']$ contains a Hamilton path with ends $s$ and $t$. 
\end{claim}

\begin{proof}
Denote $G' = G_B[S', T']$ and $Y = (S' \cup T') \setminus (S \cup T)$.
We claim that there exists a path $P$ of length at most $12\eps n$ whose vertex set contains $Y$. Indeed, we may construct $P$ greedily, by adding a vertex of $Y$ one at a time.
Suppose that we want to add the vertex $y_1 \in Y$ to a path $P$ in $G'$ of length at most $12\eps n$, one of whose ends is $y_2 \in Y$.
We may pick $z_1, z_2 \in ((S \cup T) \cap V(G')) \setminus V(P)$ such that $z_i$ is a neighbour of $y_i$ in $G'$.
We have $N_{G'}(z_i) \ge (1/4 - 11\eps)n$.
In particular, by the third assumption of the lemma, there exists a path of length at most $2$ between $N_{G'}(z_1)$ and $N_{G'}(z_2)$.
Thus we may add $y_1$ to $P$ using at most five additional vertices.
Using this process, we obtain the desired path $P$, containing the vertices of $Y$.
Denote by $s, t$ the ends of $P$ and assume that $s \in S, t \in T$ (we may need to extend $P$ slightly).

Let $G''$ be the graph obtained from $G'$ by removing the inner vertices of $P$ and denote $n'' = |G''|$ and $\eta = 25\eps$.
Then $\delta(G'') \ge (1/4 - \eta)n''$.
Let $S' \subseteq S, T' \subseteq T$ be subsets of size at least $(1/4 - 3\eta)n'' \ge (1/4 - 75\eps)(1 - 20\eps)n \ge (1/4 - 80\eps)n$.
Then by the assumptions of Lemma \ref{lem_red_graph_disconnected_1}, $e(G'[S', T']) \ge \eta n^2$.
By Lemma \ref{lem_stability_dirac}, it follows that $G'$ is Hamiltonian. The same argument may be used to show that $G'$ contains a Hamilton path with ends $s, t$ for every $s \in S', t \in T'$.
\end{proof}

We consider several cases, depending on the size of $X$ and the behaviour of the vertices in $X$.
\subsection*{Case 1: \normalfont \normalsize $X = \emptyset$}
Without loss of generality, $|S_1| \ge |T_1|$. Denote $k = |S_1| - |T_1|$. 
Suppose first that $k$ is even. We use the following claim.

\begin{claim}\label{claim_existence_even_cycles}
The graph $G[S_1]$ either contains a blue path of length $k$, or it contains a red cycle of length $k$.
\end{claim} 

\begin{proof}
Suppose that $G_B[S_1]$ has no path of length $k$.
It follows from Erd\H{o}s and Gallai's theorem, Theorem\ref{thm_erdos_gallai}, that $e(G_B[S_1]) \le k n/2 \le \eps n^2$ (note that $k \le 2\eps n$).
Setting $\eta = \sqrt{\eps}$, there are at most $\eta n$ vertices in $S_1$ of blue degree at least $\eta n$ in $S_1$.
Denote by $U$ the set of vertices in $S_1$ with blue degree at most $\eta n$.
Then $\delta(G_R[U]) \ge (1/4 - 2\eta)n$.

We show that $G_R[U]$ contains a cycle of length $k$. We may assume $k \ge 4$ because this assertion is trivial for $k = 0, 2$.
Pick $u \in U$ and denote $A = N_R(u) \cap U$ (so $|A| \ge (1/4 - 2\eta) n$).
If $G_R[A]$ has a path of length $k - 2$, together with the vertex $u$ it forms a red cycle in $S$ of length $k$.

Thus, we assume that $G_R[A]$ contains no path of length $k - 2$. 
It follows from Theorem \ref{thm_erdos_gallai} that $e(G_R[A]) \le |A| \cdot \eps n$.
We deduce that at most $|A|/2$ vertices in $A$ have red degree at least $4\eps n$ in $A$.
In particular, we may pick a set $B$ of $k / 2$ vertices in $A$ with red degree at most $4\eps n$ in $A$.
For every $v \in B$ we have $\deg_R(v, U \setminus A) \ge (1/4 - 4\eps - \eta)n$.
It follows that every two vertices in $B$ have at least say $n/8$ common red neighbours in $U \setminus A$.
In particular, if $B = \{b_1, \ldots, b_{k/2} \}$, we may pick distinct $c_1, \ldots, c_{k/2} \in U \setminus A$ such that $( b_1, c_1, \ldots, b_{k/2}, c_{k/2} )$ is a red cycle in $S$ of length $k$.
\end{proof}

By Claim \ref{claim_existence_even_cycles}, either $S_1$ contains a blue path $P$ of length $k$ or it contains a red cycle $C$ of length $k$.
In the first case, it is easy to verify that $P$ may be extended to a Hamilton cycle of $G_B$ by Claim \ref{claim_leftover_blue_grpah_ham}. Indeed, consider the bipartite graph $G_B[S_1 \setminus U, T_1]$ where $U$ is the set of inner vertices of $P_1$. This graph is almost balanced (namely the first set has one more vertex than the other), so it contains a Hamilton path whose ends are the ends of $P$.
In the second case, we may conclude directly from Claim \ref{claim_leftover_blue_grpah_ham} that the graph $G_B \setminus V(C)$ is Hamiltonian.

We now suppose that $k$ is odd.
If $S_1$ contains a blue path $P$ of length $k$, we continue as before.
Otherwise, if $S_1$ contains a blue edge $uv$, we may find a red cycle $C$ in $S_1 \setminus \{u, v\}$ of length $k - 1$, by the argument of Claim \ref{claim_leftover_blue_grpah_ham}.
It follows that the graph $G_B \setminus V(C)$ is Hamiltonian.
Finally, if $S_1$ has no blue edges, we have $\delta (G_R[S_1]) \ge |S_1| - n/4 > |S_1|/2$, since $|S_1| \ge (n + 1) /2$.
It follows from Bondy's theorem (\ref{thm_bondy}) that $G_R[S_1]$ is pancyclic, in particular it contains a cycle $C$ of length $k$. We proceed as before to conclude that $G_B \setminus V(C)$ is Hamiltonian.

\subsection*{Case 2: \normalfont \normalsize $|X| = 1$}
Denote $X = \{x\}$.
Again, we assume that $|S_1| \ge |T_1|$ and denote $k = |S_1| - |T_1|$.
Suppose first that $k$ is odd.
We shall use the following claim, whose proof is similar to the proof of Claim \ref{claim_existence_even_cycles}.
\begin{claim}\label{claim_existence_odd_cycles_through_x}
Either $G[S_1]$ has a blue path $P$ of length $k$, or $G[S_1 \cup \{x\}]$ has a red cycle $C$ of length $k + 1$ containing $x$.
\end{claim}

\begin{proof}
If $G[S_1]$ contains no blue path of length $k$, we consider the set $A = N_R(x, S_1)$.
Recall that $|A| \ge (1/4 - 25\eps)n$.
As in Claim \ref{claim_existence_even_cycles}, we conclude that either $G[A]$ contains a red path of length $k - 1$, or $G[A, S_1 \setminus A]$ contains a red path of length $k - 1$ with both ends in $A$.
\end{proof}
If $G[S_1]$ contains a path of length $k$, the graph $G_B \setminus \{x\}$ is Hamiltonian and we may take $(x)$ to be the red cycle. Otherwise, let $C$ be a red cycle of length $k + 1$ in $S_1 \cup \{x\}$ containing $x$.

We now consider the case where $k$ is even.
Note that if $k = 0$, we can take $(x)$ to be the red cycle and $G_B[S_1, T_1]$ is Hamiltonian. Thus, we assume that $k \ge 2$.
Denote $A = N_R(x, S_1)$.
\begin{claim} \label{claim_A_has_no_red_edges}
One of the following conditions holds.
\begin{enumerate}
\item
$G[S_1]$ has a blue path of length $k$.
\item
$G[S_1 \cup \{x\}]$ contains a red cycle $C$ of length $k$ going through $x$ and a blue edge $e$ which is disjoint of $C$.
\item
$G[S_1 \cup \{x\}]$ has a red cycle $C$ of length $k + 1$ going through $x$.
\item
$G[S_1]$ has no blue edges and $G[A]$ has no red edges.
\end{enumerate}
\end{claim}

\begin{proof}
It is easy to conclude, as in Claims \ref{claim_existence_even_cycles} and \ref{claim_existence_odd_cycles_through_x}, that if $G[S_1]$ has at least one blue edge,  one of the first two conditions holds.
Thus we assume that $G[S_1]$ has no blue edges.
Assuming the fourth condition does not hold, we may further assume that $G[A]$ has a red edge $uv$.
We prove that the third condition holds under these assumptions.
As before, we may assume that $e(G_R[A]) \le \eps n^2$, because otherwise $G[A]$ contains a red path of length $k - 1$ and the third condition holds.

Note that we have $\delta(G_R[S_1]) \ge |S_1| - n/4 \ge n/4$.
Thus we may construct a red path in $G[S_1]$ on $k - 2$ vertices $P = (v = v_1, \ldots, v_{k - 2})$.
If there exists a vertex $v_{k - 1} \in S_1$ which is a common red neighbour of $v_{k - 2}$ and $x$, we obtain the red cycle $C = (v_1, \ldots, v_{k - 1}, x, u)$.
Otherwise, the sets $A' = A \setminus (V(P) \cup \{u\})$ and $B = N_R(v_{k-2}, S) \setminus (V(P) \cup \{u\})$ are disjoint.
Note that $|A'|, |B| \ge (1/4 - 60\eps)n$.
Since $e(G_R[A']) \le \eps n^2$ and $\delta(G_R[S_1]) \ge n/4$, we conclude that $G_R[A', B]$ contains an edge $v_{k-1}, v_{k}$ (where $v_{k-1} \in B, v_{k} \in A'$).
It follows that $(v_1, \ldots v_{k}, x)$ is a red cycle in $S_1$ of length $k$.
\end{proof}

In each of the first three conditions of Claim \ref{claim_A_has_no_red_edges}, we may proceed as before to conclude that $G$ has the desired partition into a red cycle and a blue one.
Thus we may assume that $S_1$ spans no blue edges and $A$ spans no red edges.
In particular, $\deg_R(x, S_1) \le n/4$, hence $x$ has at least two blue neighbours in $S_1$, since $\deg(x, S_1) \ge |S_1| - (n/4 - 1) \ge n/4 + 1/2 + k/2 \ge n/4 + 3/2$.
We consider the cases $k = 2$, $k = 4$ and $k \ge 6$ separately.

Suppose first that $k = 2$. Let $u, v \in S_1$ be two blue neighbours of $x$ in $S_1$.
Let $C_1$ be a (red) cycle consisting of a single vertex in $S_1 \setminus \{u, v\}$.
We may find a blue cycle spanning $V(G) \setminus V(C_1)$ by Claim \ref{claim_leftover_blue_grpah_ham}.

Suppose now that $k = 4$.
We have $|S_1| = n/2 + 3/2$.
Thus $x$ has at least three blue neighbours $u, v, w \in S_1$.
Pick an edge $ab$ in $S_1$ such that $a$ and $b$ are distinct from $u, v, w$.
Since $\deg_R(a, S_1) \ge |S_1| - n/4 \ge n/4 + 3/2$, we conclude that $a$ and $b$ have a common red neighbour $c \in S_1$.
Let $C_1$ be the red triangles $(abc)$.
Without loss of generality $c \neq u, v$.
We proceed as before, to show that the graph $G_B \setminus \{a, b, c\}$ is Hamiltonian.

It remains to consider the case $k \ge 6$.
Fix $u, v$ to be blue neighbours of $x$ in $S_1$ and denote $S_2 = S_1 \setminus \{u, v\}$. Note $|S_2| \ge n/2 +1/2$ and $\delta (G_R[S_2]) \ge |S_2| - n/4 > |S_2|/2$.
It follows from Theorem \ref{thm_bondy}, that $G_R[S_2]$ is pancyclic. In particular, it contains a red cycle of length $k-1$.
We proceed as before.

From now on, we may assume that $|X| \ge 2$.
Set $\eta = 2\sqrt{\eps}$.
\subsection*{Case 3: \normalfont \normalsize there exist $x_1, x_2 \in X$ with $\deg_R(x_1, S), \deg_R(x_2, T) \ge (1/4 + 10\eta) n$}

Denote $X = \{x_1, \ldots, x_r\}$.
Recall that $\deg_R(x_i) \ge (3/4 - 50\eps)n$.
Thus we may pick distinct $y_i \in V(G) \setminus X$ such that $y_i$ is a common red neighbour of $x_i$ and $x_{i + 1}$ for $i \in [2, r - 1]$.
Denote $Y  =\{y_2, \ldots, y_{r - 1} \}$.
Let $S_2 = S_1 \setminus Y$ and $T_2 = T_1 \setminus Y$.

\begin{claim} \label{claim_existence_paths}
The graph $G[S_2]$ either contains a blue path of length $5\eps n$ or for every $2 \le l \le 5\eps n$ it contains a red path of length $l - 1$ whose one end is a red neighbour of $x_1$ and the other is a red neighbour of $x_r$.

Similarly, the graph $G[T_1]$ either contains a blue path of length $5 \eps n$ or for every $2 \le l \le 5\eps n$ it contains a path of length $l - 1$ a red neighbour of $x_1$ as one end and a red neighbour of $x_2$ as the other end.
\end{claim}

\begin{proof}
We prove the first part of Claim \ref{claim_existence_paths}, the second part of follows similarly.
Suppose that $G[S_2]$ has no blue path of length $5 \eps n$.
It follows by Theorem \ref{thm_erdos_gallai} that $e(G[S_2]) \le 3\eps n^2$.
Denote by $U$ the set of vertices in $S'$ with blue degree at most $\eta n$ (where $\eta = 2\sqrt{\eps}$).
Then $|S_2 \setminus U| \le \eta n$ and $\delta(G_R[U]) \ge (1/4 - 3\eta)n$.
Pick $u \in N_R(x_r, U)$.
Greedily construct a red path $P = (u = u_1, \ldots, u_{l - 1})$ in $U$.
Denote $A = N_R(x_1, U \setminus V(P))$ and $B = N_R(u_{l-1}, U \setminus V(P))$.
As $|Y| + |P| \le \eta n$, we have $|A| \ge (1/4 + 6\eta)n$ (by the assumption on $x_1$) and $|B| \ge (1/4 - 4\eta)n$.
Thus $|A \cap B| \ge |A| + |B| - |U| \ge \eta n$.
It follows that we may pick $u_l \in A \cap B$.
The path $(x_r, u_1, \ldots, u_l, x_1)$ satisfies the requirements of Claim \ref{claim_existence_paths}.

\end{proof}

Without loss of generality, we assume that $|S_2| \ge |T_2|$.
Denote $k = |S_2| - |T_2|$, so $k \le 4\eps n$.
It is easy to conclude from Claim \ref{claim_existence_paths} that we may find vertex-disjoint paths $P_1 \in G_B[S_2]$ and $Q_1 \in G_R[S_2]$ such that the following holds.
\begin{itemize}
\item
Either $P_1$ has length $k$ and $Q_1$ is a singleton, or $P_1$ is the empty path and $Q_1$ has length $k$.
\item
One end of $Q_1$ is a red neighbour of $x_1$ and the other end is a red neighbour of $x_r$ (if $Q_1$ is a singleton, then it is a common red neighbour of $x_1, x_r$).
\end{itemize}
Indeed, if $G[S_2]$ contains a blue path of length $5\eps n$, we may find such $P_1, Q_1$ where $P_1$ has length $k - 1$ and $Q_1$ is any common red neighbour of $x_1, x_r$ in $S_2$ (note that such a common neighbour exists).
Otherwise, $G[S_2]$ contains a red path of length $k - 1$ whose ends are a red neighbour of $x_1$ and a red neighbour of $x_2$.

Similarly, we may pick $P_2, Q_2$ to be a blue and a red path in $G[T_2]$ as follows.
\begin{itemize}
\item
Either $P_2$ is the empty path and $Q_2$ has length $1$ or $P_2$ has length $1$ and $Q_2$ is a singleton.
\item
$Q_2$ has one end which is a red neighbour of $x_1$ and the other is a red neighbour of $x_2$.
\end{itemize}
We take $C$ to be the red cycle $(Q_1 x_1 Q_2 x_2 y_2 x_3 \ldots x_l)$.
The paths $P_1, P_2$ may be extended to a Hamilton cycle of $G_B \setminus V(C)$ by Claim \ref{claim_leftover_blue_grpah_ham}.

Without loss of generality, we may now assume the following. 
\begin{equation}\label{eqn_deg_x_to_T}
\text{$\deg_R(x, T) \le (1/4 + 10 \eta)n$ for every $x \in X$}.
\end{equation}
It follows that $\deg(x, S) \ge (1/2 - 11\eta)n$ for every $x \in X$.
\subsection*{Case 4: \normalfont \normalsize some $x_1, x_2 \in X$ have at least $3\eta n$ common red neighbours in $T$}

We proceed similarly to the previous case.
For $2 \le i \le r - 1$, pick $y_i$ to be a common red neighbour of $x_i$ and $x_{i + 1}$ such that the $y_i$'s are distinct and do not belong to $X$.

Let $S_2 = S_1 \setminus Y$ and $T_2 = T_1 \setminus Y$.
Note that $\deg_R(x_i, S_1) \ge (1/2 - 11\eta)n$ for $i \in [r]$ by (\ref{eqn_deg_x_to_T}).
Consider the set $D$ of common red neighbours of $x_1, x_r$ in $S'$. Then $|D| \ge (1/2 - 23\eta)n$. Clearly, $D$ has either a blue path or a red one of length $5\eps n$.
\begin{claim} \label{claim_existence_even_path_case_4}
The graph $G[T_2]$ either has a blue path of length $5\eps n$ or it contains a red path of length $l$ with ends which are neighbours of $x_1, x_2$ respectively, for every even $2 \le l \le 5\eps n$.
\end{claim}

\begin{proof}
Suppose that $G[T_2]$ has no blue path of length $5\eps n$.
Denote by $U$ the set of vertices of $T_2$ with at most $\eta n$ blue neighbours.
Then $|T_2 \setminus U| \le \eta n$ and $\delta(G_R[U]) \ge (1/4 - 3\eta)n$.
Denote $A = N_R(x_1, U)$ and $B = N_R(x_2, U)$.
Note that $|A|, |B| \ge (1/4 - 2\eta)n$ and $U \cap A \cap B \neq \emptyset$, by the assumptions on $x_1$ and $x_2$, thus we may pick $u_1 \in U \cup A \cup B$. 

If $|A \cup B| \ge (1/4 + 10 \eta) n$, we may find a red path of length $l$ as follows.
Greedily pick a red path $P = (u_1, \ldots, u_l)$ in $U$.
As in Claim \ref{claim_existence_paths}, there exists $u_{l + 1} \in (A \cup B) \backslash V(P)$ which is a red neighbour of $u_l$.
The path $(u_1, \ldots, u_{l + 1})$ satisfies the requirements.

If $|A \cup B| \le (1/4 + 10 \eta)n$, it follows that $|A \cap B| \ge |A| + |B| - |A \cup B| \ge (1/4 - 14\eta)n$.
If $A \cap B$ contains a red path of length $l$, we are done.
Otherwise, we may continue as in Claim \ref{claim_existence_even_cycles} to conclude that the graph $G_R[A \cap B, U \setminus (A \cap B)]$ has a path of length $l$ with ends in $A \cap B$.
\end{proof}

Denote $k = |S_2| - |T_2|$.
Pick $2 \le l_1, l_2\le 5\eps n$ such that $l_2$ is even and $l_1 - l_2 = k$.
Similarly to the previous case, we may pick vertex-disjoint paths $P_1 \in G_B[S_2]$, $Q_1 \in G_R[S_2]$, $P_2 \in G_B[T_2]$ and $Q_2 \in G_R[T_2]$ with the following properties.
\begin{itemize}
\item
$Q_1$ has ends which are a red neighbour of $x_1$ and a red neighbour of $x_r$. Similarly, $Q_2$ has ends which are a red neighbour of $x_2$ and a red neighbour of $x_1$.
\item
Either $P_i = \emptyset$ and $Q_i$ has length $l_i$ or $P_i$ has length $l_i$ and $Q_i$ is a singleton, for $i \in [2]$.
\end{itemize}
As before, we take $C$ to be the red cycle $(Q_1 x_1 Q_2 x_2 y_2 x_3 \ldots x_r)$. The paths $P_1, P_2$ may be extended to a Hamilton cycle in $G_B \setminus V(C)$.

Note that if $|X| \ge 3$, one of Cases 3 and 4 holds (perhaps with the roles of $S$ and $T$ reversed).
Thus we may assume that $|X| = 2$, and denote $X = \{x_1, x_2\}$.
\subsection*{Case 5: \normalfont \normalsize $\deg_R(x_i, T) \le (1/4 + 10\eta)n$ and $|N_R(x_1, T) \cap N_R(x_2, T)| \le 3\eta n$}

Denote $$\text{$A = N_R(x_1, T_1)$, $B = N_R(x_2, T_1)$ and $D = N_R(x_1, S_1) \cap N_R(x_2, S_1)$}.$$
Note that $|D| \ge (1/2 - 23\eta)n$.

If $|S_1| \ge |T_1| + 2$, denote $k = |S_1| - |T_1|$,  so $2 \le k \le 2\eps n$.
As usual, $G[D]$ either contains a blue path $P$ or a red path $Q$ of length $k - 2$.
In the former case, pick $C$ to be a $4$-cycle consisting of $x, y$ and two vertices in $D \setminus V(P)$.
In the latter case, extend $Q$ to a cycle $C$ through $x, y$ using an additional vertex of $D$. For convenience denote $P = \emptyset$.
In both cases, the path $P$ may be extended to a Hamilton cycle of $G_B \setminus V(C)$.
From now on, we may assume the following.
\begin{equation*}
|S_1| \le |T_1| + 1.
\end{equation*}

Denote $k = |T_1| - |S_1|$ (so $-1 \le k \le 2\eps n$).
A \emph{path forest} is a collection of vertex disjoint paths.
If $G_B[T_1]$ contains a path forest $H$ with $k + 2$ edges, we can finish the proof as follows. Pick $C$ to be any red $4$-cycle consisting of $x_1, x_2$ and two vertices from $D$. Then $G_B \setminus V(C)$ has a Hamilton cycle extending $H$.
This can be seen by connecting the paths of $H$ with paths in $G_B[S_1, T_1]$ of length at most $6$ and using Claim \ref{claim_leftover_blue_grpah_ham}.
Thus we may assume the following. 
\begin{align}
\text{$G_B[T_1]$ has no path forest with $k + 2$ edges. In particular, $e(G_B[T_1]) \le \eps n^2$}.
\label{eqn_no_large_path_forest}
\end{align}

Suppose that $G_R[T_1]$ contains a path $P$ with one end in $A$ and the other in $B$, of length $l$, where $0, k \le l \le 3\eps n$ (by a path of length $0$ we mean a single vertex). Then we may find the desired cycle partition as follows.
$G[D]$ contains a path $Q$ of length $l - k$ which is either red or blue.
If $Q$ is red, take $C$ to be the red cycle $(x_1 P x_2 Q)$, and the remaining graph $G_B \setminus V(C)$ is Hamiltonian.
If $Q$ is blue, take $C = (x_1 P x_2 u)$ where $u \in D \setminus V(Q)$.  The leftover graph $G_B \setminus V(C)$ has a Hamilton path extending $Q$.
Thus from now on we assume the following.
\begin{equation}\label{eqn_red_A-B_paths}
\text{$G_R[T_1]$ has no path of length $l$, where $0, k \le l \le 3\eps n$, with ends in $A$ and $B$}. 
\end{equation}
In particular, $e(G_R[A, B \setminus A]) \le \eps n^2$, 
implying that $G_R[A], G_R[B]$ are almost complete (recall that $A \cap B$ have a small intersection).
Suppose that $x_1$ has two blue neighbours $u, v \in T_1$.
It follows from the previous assumption that $G_R[B]$ contains a path $P$ on $k + 1$ vertices. Form a red cycle $C$ by adding the vertex $x_2$ to $P$.
Then the graph $G_B \setminus V(C)$ is Hamiltonian, since the graph $G_B \setminus (V(C) \cup \{x_1\})$ has a Hamilton path with ends $u, v$.
Thus we may assume that both $x_1$ and $x_2$ have at most one blue neighbour in $T_1$.
In particular, 
\begin{equation}\label{eqn_size_A_B}
|A|, |B| \ge |T_1| - n/4.
\end{equation}

We can now finish the proof in case $|S_1| - |T_1| \in \{0, 1\}$. Note that we have $e(G_B[T_1]) \le 2$ by Assumption (\ref{eqn_no_large_path_forest}). 
It is easy to conclude, using Assumption (\ref{eqn_size_A_B}) that one of the following conditions holds.
\begin{itemize}
\item
$A \cap B \neq \emptyset$.
\item
$A \cap B = \emptyset$ and $e(G_R[A, B]) \ge 1$.
\item
There exists a vertex $u \in T_1$ which has red neighbours in both $A$ and $B$.
\end{itemize}
It follows that there exists a red path in $T_1$ with one end in $A$ and the other in $B$ of length at most $2$, contradiction Assumption (\ref{eqn_red_A-B_paths}).
Thus the desired monochromatic cycle partition exists.

We may now assume $k = |T_1| - |S_1| \ge 1$.
Denote $Y = \{u \in T_1: \deg_B(u, T_1) \ge \eta n\}$.
Then 
\begin{align}\label{eqn_upperbound_Y}
|Y| \le (k + 1)/2
\end{align}
because otherwise $G_B[T_1]$ contains at least $(k + 2)/2$ vertex-disjoint paths of length $2$, contradicting Assumption (\ref{eqn_no_large_path_forest}).

Denote $A' = A \setminus Y$ and $B' = B \setminus Y$.
If there exists a path of length at most $2$ in $G_R[T_1]$ with one end in $A'$ and the other in $B'$, this path may be extended to a red path between $A'$ and $B'$ of length $k + 2$, contradicting assumption (\ref{eqn_red_A-B_paths}).
Thus we assume the following. 
\begin{align}
& A' \cap B' = \emptyset. \label{eqn_A'_intersect_B'_empty}\\
& e(G_R[A', B']) = 0. \label{eqn_A'-B'_edges} \\
& \text{No vertex in $T_1$ has red neighbours in both $A'$ and $B'$} \label{eqn_A'_B'_have_no_common_red_neighs}. 
\end{align}
It is not hard to reach a contradiction from here, thus finishing the proof.
By (\ref{eqn_A'_intersect_B'_empty}), $A \cap B \subseteq Y$.
In particular,
\begin{equation}\label{eqn_size_Y}
|Y| \ge |A \cap B| \ge |A| + |B| - |T_1| \ge |T_1| - n/2 = (k - 2)/2.
\end{equation}
Also, by (\ref{eqn_upperbound_Y}),
\begin{equation}\label{eqn_size_A'_B'}
|A'|, |B'| \ge |T_1| - n/4 - |Y| \ge n/4 - 3/2.
\end{equation}
If $|Y| \le (k - 1)/2$, it follows that $|A'|, |B'| \ge n/4 - 1/2$. By the minimum degree condition, we have that $\delta(G[A', B']) \ge 1$, and it is easy to deduce that $G[A', B']$ has a (blue) path forest on at least four edges.
By (\ref{eqn_size_Y}), we may complete this into a blue path forest in $T_1$ with at least $k + 2$ edges, a contradiction to (\ref{eqn_no_large_path_forest}).

Thus we assume that $|Y| \ge k/2$. It follows from (\ref{eqn_no_large_path_forest}) that 
\begin{equation}\label{eqn_blue_edges_outside_Y}
e(G_B \setminus Y) \le 1.
\end{equation}
If $|A'| > n/4 - 1$, the graph $G[A', B']$ has at least two (blue) edges, a contradiction.
Thus we assume that $|A'|, |B'| \le n/4 - 1$.
It follows that $|A \cup B| \le |A'| + |B'| + |Y| < |T_1|$.
Pick $u \in T_1 \setminus (A \cup B)$.
By (\ref{eqn_size_A'_B'}, \ref{eqn_A'-B'_edges}, (\ref{eqn_blue_edges_outside_Y})), all but at most two vertices of $A' \cup B'$ are connected to $u$, contradicting (\ref{eqn_A'_B'_have_no_common_red_neighs}).

\end{proof}

\section{Proof of Lemma \ref{lem_red_disconnected_2}}
\label{sec_red_disconnected_2}

In this section, we prove Lemma \ref{lem_red_disconnected_2}. We first remind the reader of the statement.

\begin{lem*} [\ref{lem_red_disconnected_2}]
Let $\frac{1}{n} \ll \eps \ll 1$ and let $G$ be a graph on $n$ vertices with $\delta(G)\ge 3n/4$ and a $2$-colouring $E(G) = E(G_B) \cup E(G_R)$.
Suppose that there exists a partition $\{S, T, X\}$ of $V(G)$ with the following properties.
\begin{itemize}
\item
$|S|, |T| \ge (1/2 - \eps)n$.
\item
$|X| \le 2$ and if $|X| = 2$, there exists $u \in X$ such that $\deg_R(x, S) \le \eps n$ or $\deg_R(x, T) \le \eps n$.
\item
The sets $S$ and $T$ belong to different components of $G_R \setminus X$.
\end{itemize}
Then $V(G)$ may be partitioned into a red cycle and a blue one.
\end{lem*}
The idea of the proof is as follows. By Lemma \ref{lem_red_graph_disconnected_1}, we may assume that there exist subsets $S' \subseteq S, T' \subseteq T$ of size almost $n/4$ such that $G[S', T']$ is close to being empty, implying that the subgraphs $G_B[S', T \setminus T']$ and $G_B[S \setminus S', T']$ are almost complete.
We aim, similarly to the proof of Lemma \ref{lem_red_graph_disconnected_1} to find a red cycle $C$, and two blue paths $P_1, P_2$, whose removal from $G$ leaves two balanced bipartite subgraphs of the aforementioned graphs. We then find Hamilton paths in the remainder subgraphs which together with $P_1, P_2$ form a blue cycle with vertex set $V(G) \setminus V(C)$.
We remark that we run into some technical difficulties when $X$ is non-empty.
 
\begin{proof}[Proof of Lemma \ref{lem_red_disconnected_2}]

We abuse notation and slightly change the definition of $S$, $T$ and $X$ as follows.
Denote 
\begin{equation}\label{eqn_defn_x_y}
X = \{x, y\}
\end{equation}
(if $|X| < 2$, add vertices to $X$ arbitrarily).
Recall that by the conditions of Lemma \ref{lem_red_disconnected_2}, we may assume that $\deg_R(x, S) \le \eps n$ or $\deg_R(x, T) \le \eps n$.
If the former holds, we move $x$ from $X$ to $T$, otherwise we move $x$ from $X$ to $S$.
Similarly, if $\deg_B(y, S) \ge n / 32$ we move $y$ from $X$ to $T$ and otherwise, if $\deg_R(y, T) \le n / 32$, we put $y$ in $S$.
After the modification, we have either $X = \emptyset$, or $X = \{y\}$ and $\deg_B(y) \le n / 16$.
Furthermore, the only red edges in $G[S, T]$ are adjacent to $x$ or $y$.

We may assume that there exist subsets $S' \subseteq S, T' \subseteq T$ of size at least $(1/4 - 100\eps)n$ such that $e(G[S', T']) \le 25\eps n^2$, because otherwise the proof can be completed immediately by Lemma \ref{lem_red_graph_disconnected_1}.
It is easy to deduce the following claim (we omit the exact details of the proof).
\begin{claim}
The following holds for some $\eta = \eta(\eps) \ge \eps$.
There exist partitions $\{S_1, S_2\}$ of $S$ and $\{T_1, T_2\}$ of $T$ with the following properties.
\begin{itemize}
\item
$|S_1|, |S_2|, |T_1|, |T_2| \ge (1/4 - \eta) n$.
\item
All but at most $\eta n$ vertices of $S_2 \cup T_1$ have degree at most $\eta n$ in $G_B[S_2, T_1]$.
\item
The graphs $G_B[S_i, T_i]$ have minimum degree at least $n / 64$. Furthermore, all but at most $\eta n$ vertices in these graphs have degree at least $(1/4 - \eta)n$.
\item
All but at most $\eta n$ vertices in the graphs $G[S_1, S_2]$ and $G[T_1, T_2]$ have degree at least $(1/4 - \eta) n$.
\end{itemize}
\end{claim}

We shall also use the following claim which may be easily verified by Theorem \ref{thm_chvatal}.
\begin{claim} \label{claim_red_disconnected_2_ham}
Let $S' \subseteq S_i$ and $T' \subseteq T_i$ be sets of equal size such that $|(S_i \cup T_i) \setminus (S' \cup T')| \le 10 \eta n$.
Then $H = G_B[S', T']$ is Hamiltonian. Furthermore, for every $s \in S', t \in T'$, $H$ contains a Hamilton path with ends $s$ and $t$.
\end{claim}

Without loss of generality, we assume $|S| \ge |T|$. Denote $k = |S| - |T|$.
We consider two cases according to the size of $X$.
\subsection*{Case 1: \normalfont \normalsize $X = \emptyset$}

Let $A$ be the set of red neighbours of $x$ in $G[S, T]$ and $B$ the set of red neighbours of $y$ in $G[S, T]$.
Then 
\begin{equation}\label{eqn_size_A_B}
\text{$|A| \le n/16$ and $|B| \le (1/2 - 1/64)n$}.
\end{equation}

If $k \le 1$, we may find a partition of $V(G)$ into a blue cycle and a red one as follows.
Pick any red cycle $C_1$ in $S$ of length $k$, so $C_1$ is either the empty set or a vertex.
Denote $S' = S \setminus V(C_1)$ and consider the balanced bipartite graph $H = G_B[S', T]$.
Then
\begin{align*}
\deg_{H_B}(u) \ge \left\{ \begin{array}{ll}
n / 64 & u \in \{x, y\}\\
(n - k)/2 - (n/4 - 1) - 2 = n/4 - k/2 - 1 & u \in A\\
n/4 - k/2 & u \in B \setminus A\\
n/4 - k/2 + 1 & \text{otherwise} 
\end{array} \right .
\end{align*}
It follows by Theorem \ref{thm_chvatal} that $H$ is Hamiltonian (note that $B$ is a subset of either $S$ or $T$), implying that the desired partition of $V(G)$ into a red cycle and a blue exists.
\subsubsection*{Case 1.1: \normalfont \normalsize $|T_1| < |S_1|$ and $|T_2|  < |S_2|$}
Recall that $k \ge 2$, so
\begin{equation} \label{eqn_size_S1}
|S| = (n + k)/2 \ge n/2 + 1.
\end{equation}

\begin{claim} \label{claim_two_matching}
The graph $G_B[S_1, T_2] \cup G_B[S_2, T_1]$ contain a (blue) matching of size $2$.
\end{claim}
\begin{proof}
If $|S_2| > n/4 + 1$, every vertex in $T_1$ has at least three neighbours in $S_2$, thus every vertex in $T_1 \setminus A$ has at least two blue neighbours in $S_2$, implying that we may find the desired matching. 
Similarly, if $|S_1| > n/4 + 1$, the graph $G_B[S_1, T_2]$ contains a matching of size $2$.
Thus we may assume that $|S_1|, |S_2| \le n/4 + 1$, implying that $|S_1|, |S_2| \ge n/4$.

Without loss of generality, $|S_1| \ge n/4, |S_2| > n/4$.
Hence every vertex in $T_1 \setminus A$ has at least two neighbours in $S_2$. If $y \notin S_2$, these two neighbours are both blue, and we may find the required matching in $G_B[S_2, T_1]$. If $y \in S_2$, we conclude similarly that both graphs $G[S_1, T_2]$ and $G[S_2, T_1]$ contain at least one blue edge, implying that the desired matching exists.
\end{proof}

Without loss of generality, we shall assume the following. 
\begin{equation}
\text{$s_1t_1, s_2t_2 $ is a blue matching in $G[S_2, T_1]$, where $s_1, s_2 \in S_2$ and $t_1, t_2 \in T_1$}.
\end{equation}
It is easy to verify that one of the following holds.
\begin{enumerate}
\item\label{itm_1_claim_S1}
$G[S_1]$ contains a blue path of length $5\eta n$.
\item\label{itm_2_claim_S1}
$G[S_1, T_2]$ contains a blue path of length $5 \eta n$.
\item\label{itm_3_claim_S1}
$G[S_1, S_2]$ contains a blue path of length $5 \eta n$.
\item\label{itm_4_claim_S1}
$G[S_1]$ contains a red cycle of length $l$ for every $l \le 5\eta n$, and $G[S_2]$ contains a blue path of length $5\eta n$.
\item\label{itm_5_claim_S1}
For every $l_1, l_2 \le 5\eta n$, $G[S_1, S_2]$ contains a red cycle with $l_1$ vertices from $S_1$ and $l_2$ vertices from $S_1$.
\end{enumerate}
Indeed, if the first three conditions do not hold, by Theorem \ref{thm_erdos_gallai}, the graphs $G_R[S_1]$ and $G_R[S_1, S_2]$ are almost complete, so either the third condition holds or $G_R[S]$ is almost complete.

In Case (\ref{itm_1_claim_S1}), we conclude that $G[S_1 \cup T_1]$ contains a blue Hamilton path $Q_1$ with ends $t_1, t_2$. 
Indeed, let $P_1$ be a path in $G_B[S_1]$ on $|S_1| - |T_1| + 2$ vertices (note that this quantity is smaller than $5\eta n$).
By Claim \ref{claim_red_disconnected_2_ham}, we may extend $P_1$ to a Hamilton path of $G_B[S_1\cup T_1]$ with ends $t_1$ and $t_2$.
Now consider $G[S_2]$. Note that $G[S_2]$ is almost complete (as $G[S_2, T_1]$ is almost empty), thus it either contains a blue path of length $|S_2| - |T_1| - 2$ or a red cycle of length $|S_2| - |T_1| - 1$.
In any case, we may partition $S_2 \cup T_2$ into a red cycle $C$ and a blue path $Q_2$ with ends $s_1, s_2$ (where in the former case the red cycle is empty).
By joining $Q_1$ and $Q_2$ using the edges $s_1t_1$ and $s_2t_s$, we obtain a blue cycle on the vertices $V(G) \setminus V(C)$.

In Case (\ref{itm_2_claim_S1}), let $P_1$ be a path in $G_B[S_1, T_2]$ with ends $s \in S_1, t \in T_2$ with exactly $|S_1| - |T_1| + 1$ vertices from $S_1$. 
Denote by $U$ the set of inner vertices of $P_1$ and let $S_1' = S_1 \setminus U$ and $T_2' = T_2 \setminus U$.
By Claim \ref{claim_red_disconnected_2_ham} the graph $G_B[S_1', T_1]$ has a Hamilton path $Q_1$ with ends $s, t_1$.
Consider the graph $G[S_2 \cup T_2']$. 
We conclude as before that $G[S_1, T_1']$ can be partitioned into a blue path with ends $s_1, t$ and a red cycle, completing the required partition of $V(G)$ into a blue cycle and a red one.

We assume that Condition (\ref{itm_2_claim_S1}) does not hold, implying that $G[S_1, T_2]$ is almost empty and so $G[T_2]$ is almost complete.
Suppose that Condition (\ref{itm_3_claim_S1}) holds.
Let $P_1$ be a path in $G_B[S_1, S_2]$ with ends $s_3 \in S_1, s_4 \in S_2$ and exactly $|S_1| - |T_1| + 1$ vertices from $S_1$.
Define $U$ to be the set of inner vertices of $P_1$ and let $S_i' = S_i \setminus U$.
It can be shown as before that $G_B[S_1', T_1]$ has a Hamilton path with ends $s_3, t_1$ and that $G_B[S_2', T_2]$ can be partitioned into a blue path with ends $s_1,s_4$ and a red cycle. Note that it may happen that $|T_2| \ge |S_2'|$ in which case the red cycle is contained in $T_2$ (here we use the assumption that $G[T_2]$ is almost complete).

Now suppose that Condition (\ref{itm_4_claim_S1}) holds.
It follows that $G[S_1 \cup T_1]$ may be partitioned into a blue path with ends $t_1, t_2$ and a red cycle, and that $G[S_2 \cup T_2]$ contains a blue Hamilton path with ends $s_1, s_2$.

Finally, if Condition (\ref{itm_5_claim_S1}) holds, let $C$ be a red cycle consisting of $|S_1| - |T_1| + 1$ vertices from $S_1$ and $|S_2| - |T_2| - 1$ vertices from $S_2$.
As before, the graphs $G[S_1 \setminus V(C), T_1]$ and $G[S_2 \setminus V(C), T_2]$ contain blue Hamilton paths with ends $t_1, t_2$ and $s_1, s_2$ respectively. 

\subsubsection*{Case 1.2: \normalfont \normalsize $|S_1| = |T_1|$ or $|S_2| = |T_2|$}

Suppose that $|S_1| = |T_1|$.
Similarly to Claim \ref{claim_two_matching}, we claim that $G[S_2, T_1]$ contains a blue matching of size $2$.
Indeed, if either $|S_2| > n/4 + 1$ or $|T_1| = |S_1| > n/4 + 1$ we may find the required matching as in the proof of the claim.
Otherwise, we have $n/4 \le |S_2|, |S_1| \le n/4 + 1$ and without loss of generality $|S_2| > n/4$ and $|T_1| = |S_1| \ge n/4$, and again we may find the required matching as in Claim \ref{claim_two_matching}.

Similarly, if $|S_2| = |T_2|$, it follows that $G[S_1, T_2]$ contains a blue matching of size $2$ and we proceed as before.

\subsubsection*{Case 1.3: \normalfont \normalsize $|S_1| < |T_1|$}

In order to obtain the required balanced subgraph of $G[S_1, T_1]$, we use the following claim.
\begin{claim} \label{claim_paths_S2_T1}
There exist vertex-disjoint paths $Q_1$, $Q_2$ in $G_B[S, T]$ satisfying the following properties.
\begin{itemize}
\item
$|Q_1|, |Q_2| \le 20\eta n$.
\item
$Q_i$ has one end in $S_2$ and the other in $T_1$.
\item
Denote $U = V(Q_1) \cup V(Q_2)$. Then 
\begin{align*}
|U \cap T_1| - |U \cap S_1| = |T_1| - |S_1| + 1.
\end{align*}
\end{itemize}
\end{claim}

\begin{proof}
Consider the bipartite graph $H = G_B[T_1, S']$ where $S' = S \setminus \{u\}$ for some fixed $u \in S_1$.
Recall that $|S'| = |S| - 1 \ge n/2$ (\ref{eqn_size_S1}).
We show that we there exist two edge-disjoint matchings $M_1, M_2$ of size $|T_1| - |S_1| + 1$ in $G[S_2, T_1]$ whose union contains no cycles.
To that end, we show first that $H$ contains a matching saturating $T_1$, by showing that $H$ satisfies Hall's condition, namely that for every $W \subseteq T_1$, we have $|N_H(W)| \ge |W|$.

Recall that $x, y$ are the vertices that were in $X$ originally (\ref{eqn_defn_x_y}), and $A, B$ are their red neighbourhoods in $G[S, T]$.
We consider four ranges for the size of $W$.
\begin{itemize}
\item
$|W| \le 2$. Here Hall's condition holds trivially because the minimum degree of a vertex from $T_1$ is larger than $2$.
\item
$3 \le |W| \le n/4 - 1$.
Recall that every vertex in $T_1$, except for possibly $x$ and $y$, has degree at least $n/2 - (n/4 + 1) = n/4 - 1$ in $H$.
Thus, in this case we have $$|N_H(W)| \ge n/4 - 1 \ge |W|.$$
\item
$n/4 - 1 < |W| \le n/4$.
By the lower bound on the size of $A$, there exists $w \in W \setminus (A \cup \{x, y\})$.
$w$ has at most one red neighbour in $S'$ (namely $y$), thus $$|N_H(W)| \ge \deg_B(w, S') \ge n/2 - n/4 = n/4 \ge |W|.$$
\item
$|W| > n/4$.
In this case, every vertex in $S' \setminus (A \cup \{x, y\})$ has a neighbour in $W$, so $$|N_H(W)| \ge |S'| - (n/16 + 2) \ge |T_1| \ge |W|.$$
\end{itemize}
It follows that there exists a matching in $G[S_1' \cup S_2, T_1]$ saturating $T_1$. We take $M_1$ to be a sub-matching in $G_B[T_1, S_2]$ of size $|T_1| - |S_1| + 1 \le 5\eta n$.

Consider the graph $H'$ obtained from $H$ by removing the edges of $H$ spanned by $V(M_1)$.
It is easy to check by a similar analysis that $H'$ contains a matching saturating $T_1$. Indeed, if $1 \le |W| \le 5\eta n + 2$, we clearly have $|N_H(W)| \ge |W|$.
If $5 \eta n + 2 < |W| \le n/4 - 1$, we have $|N_H(W)| \ge n/4 - 1 \ge |W|$.
If $n/4 - 1 < |W| \le n/4$, we may pick a vertex $w \in W \setminus (A \cup V(M_1) \cup \{x, y\}$ and continue as above.
Finally, if $|W| > n/4$, every vertex in $S' \setminus (V(M_1) \cup A \cup \{x, y\})$ has a blue neighbour in $W$, so $|N_H(W)| > |W|$.
So there exists a matching $M_2$ in $G_B[T_1, S_2]$, edge-disjoint of $M_1$, of size $|T_1| - |S_1| + 1$.

Let $F$ be the graph $(V(G), E(M_1) \cup E(M_2))$.
Note that by the choice of $H'$, $F$ contains no cycles(in fact it contains no paths of length larger than $3$), hence it is a collection of vertex-disjoint paths.

We obtain paths $Q_1, Q_2$ as follows.
Start with the collection of paths in $F$.
In each stage, if there are two paths with an end in $S_2$, connect them using a path in $G_B[S_2, T_2]$ of length at most $4$, such that it remains vertex-disjoint of all other paths constructed so far.
Similarly, if there are two paths with an end in $T_1$ we join them using a path of length at most $4$ in $G_B[S_1, T_1]$.
Note that whenever there are at least three paths, we may continue the process.
We stop when exactly two paths remain.

Denote $U = V(Q_1) \cup V(Q_2)$, and let $r_i$ be the number of vertices in $T_1$ of degree $i$ in $F$, where $i \in [2]$.
Then
\begin{align*}
|U \cap T_1| - |U \cap S_1| 
& = r_2 + (r_1 - 2)/2 + 2 = 
(2r_2 + r_1)/2 + 1 \\
& = |E(F)|/2 + 1 = 
|T_1| - |S_1| + 2.
\end{align*}
By the definition of $Q_1, Q_2$, the set of edges of $G_B[S_2, T_1]$ which are contained in one of $Q_1, Q_2$ is exactly $E(F)$, so in particular, it has even size.
It follows that either both paths $Q_1, Q_2$ have one end in $S_2$ and the other in $T_1$ or one of them has both ends in $S_2$ and the other has both ends in $T_1$.
Recall that by the way the paths $Q_1, Q_2$ were constructed, the first and last edges in each of them are in $G_B[S_2, T_1]$.
Hence, by either removing the first and last edges from one of the paths or by removing the first edge from each of them, we obtain paths $Q_1', Q_2'$ satisfying the requirements of Claim \ref{claim_paths_S2_T1}.
\end{proof}

Let $Q_1, Q_2$ be paths as in Claim \ref{claim_paths_S2_T1}.
Denote by $s_i \in S_2, t_i \in T_1$ the ends of $Q_i$, let $U'$ be the set of inner vertices of these paths, and let $S_i' = S_i \setminus U'$ and $T_i' = T_i \setminus U'$.
It is easy to verify that $G_B[S_1', T_1']$ has a Hamilton path with ends $t_1, t_2$ and that $G[S_2' \cup T_2']$ may be partitioned into a blue path with ends $s_1, s_2$ and a red cycle (note that $|S_2'| \ge |T_2'|$, and we use the fact that $G[S_2]$ is almost complete).

\subsubsection*{Case 1.4: \normalfont \normalsize $|S_2| < |T_2|$}
Here we have $|S_1| \ge |T_1| + 2$.
If $G[S_1, T_2]$ contains a blue path of length $10\eta n$, we may proceed as in Condition (\ref{itm_2_claim_S1}) from Case 1.1.
Otherwise, $G[S_1, T_2]$ has few edges, so $G[S_1]$ contains many edges and we may proceed as in the previous case.

\subsection*{Case 2: \normalfont \normalsize $X = \{y\}$}
We start by showing that if $k \le 2$, we may partition $V(G)$ into a red cycle and a blue one.
Consider $A = N_R(y, S)$ and recall that $|A| \ge n /16$.
If $k = 0$, we pick the red cycle $C_1$ to be $(y)$.
If $k = 1$, we pick the red cycle $C_1$ to be a red edge $(yz)$, where $z \in A$.
If $k = 2$ and $A$ contains a red edge $uv$, we pick the red cycle $C_1$ to be the triangle $(uvy)$.

In each of these cases, the remainder graph $H = G[S', T]$, where $S' = S \setminus V(C_1)$ is a balanced bipartite graph.
To obtain the desired partition of $V(G)$ into a red cycle and a blue one, we show that $H_B$ is Hamiltonian.
Indeed, denote $B = N_{H_R}(x)$. Then
\begin{align*}
\deg_{H_B}(u) \ge \left\{ \begin{array}{ll}
(n - k - 1)/2 - (n/4 - 1) - n/16 = 3n/16 + 1/2 - k/2 & v = x\\
(n - k - 1)/2 - (n/4 - 1) - 1 = n/4 - k/2 - 1/2 & v \in B\\
(n - k - 1)/2 - (n/4 - 1) = n/4 - k/2 + 1/2 & \text{otherwise} 
\end{array} \right .
\end{align*}
It follows from Theorem \ref{thm_chvatal} that indeed, $H_B$ is Hamiltonian.

It remains to consider the case where $k = 2$ and $A$ contains no red edges.
Suppose first that $y$ has two blue neighbours in $S$. Consider the graph $H = G[S', T']$, where $S' = S \setminus \{u\}$ for some $u \in S$ and $T' = T \cup \{y\}$.
We claim that $H_B$ is Hamiltonian. The desired partition may thus be obtained by letting $(u)$ be the red cycle.
Indeed, let $B = N_{H_R}(x)$. Then
\begin{align*}
\deg_{H_B}(v) \ge \left\{ \begin{array}{ll}
2 & v = y \\
(n - 1)/2 - n/4  - n/16 = 3n/16 - 1/2 & v = x\\
(n - 1)/2 - (n/4 + 1) = n/4 - 3/2 & v \in B\\
(n - 1)/2 - n/4 = n/4 - 1/2 & v \in S' \setminus (B \cup \{x, y\}) \\
(n - 1)/2 - (n/4 - 1) = n/4 + 1/2 & v \in T' \setminus (B \cup \{x, y\})
\end{array} \right .
\end{align*}
By Theorem \ref{thm_chvatal}, $H_B$ is indeed Hamiltonian.

We may now assume that $y$ has at most one blue neighbour in $S$. It follows that $\deg_R(y, S) = |A| \ge |S| - n/4 = n/4 + 1/2$. Hence, $G[A]$ contains an edge $uv$. Recall that we assumed that $A$ contains no red edges, thus $uv$ is a blue edge.
Let $w \in A \setminus \{u, v\}$.
Consider the graph $H = G[S', T]$ where $S' = S \setminus \{w\}$.
We claim that $H_B$ contains a blue Hamilton path with ends $u, v$.
Indeed, let $H'$ be the graph obtained from $G_B[S', T]$ by adding a vertex $z$ to $T$ which is connected only to $u, v$.
Clearly, $H$ contains a Hamilton path with ends $u, v$ if and only if $H'$ is Hamiltonian. As in the previous argument, it follows from Theorem \ref{thm_chvatal} that $H'$ is Hamiltonian.

From now on, we assume that $k \ge 3$, so $|S| \ge n/2 + 1$.
Denote 
\begin{equation} \label{eqn_defn_A_i_B_i}
\text{$A_i = N_R(y, S_i)$ and $B_i = N_R(y, T_i)$}.
\end{equation}
Recall that $\deg_B(y) \le n / 16$. It follows that at most one of the sets $A_1, A_2, B_1, B_2$ can have size at most $n / 16$.

\subsubsection*{Case 2.1: \normalfont \normalsize $e(G_B[S_1, T_2]) \ge 10\eta n^2$}

Suppose first that $|A_2| \ge n / 16$. 
This case can be treated similarly to previous cases where $X = \emptyset$.
If $|S_1| > |T_1|$, we may continue similarly to Condition (\ref{itm_2_claim_S1}) from Case 1.1.
If $|S_1| = |T_1|$ we continue similarly to Case 1.2 and if $|S_1| < |T_1|$, we continue as in Case 1.3.
The only difference in the arguments is that when considering the graph $G[S_2' \cup T_2']$ we partition it into a blue path with suitable ends and a red path contained in $A_2 = N_R(y, S_2)$ rather than a red cycle (by the lower bound on $A_2$, we have that $G[A_2]$ is almost complete, hence we may indeed do so).

We now suppose that $|A_2| \le n / 16$.
This implies that $|A_1|, |B_1|, |B_2| \ge n / 16$.
Suppose first that $G_B[S_1, S_2]$ contains a path of length $10 \eta n$.
We may pick vertex disjoint paths $P_1 \subseteq G_B[S_1, S_2]$ and $P_2 \subseteq G_B[S_1, T_2]$ of length at most $10 \eta n$ such that the following assertions hold.
\begin{itemize}
\item
$P_1$ has ends $s_1 \in S_1, s_2 \in S_2$ and $P_2$ has ends $s_3 \in S_1, t_3 \in T_2$.
\item
$|S_i \setminus V(P_1)| < |T_i|$ for $i \in [2]$.
\item
Denote by $U$ the set of the inner vertices from the paths $P_1, P_2$ and $S_i' = S_i \setminus U$, $T_i' = T_i \setminus U$.
Then $|S_1'| = |T_1'|$ (and thus $|S_2'| < |T_2'|$).
\end{itemize}
Indeed, to construct the paths $P_1, P_2$, first pick $P_1 \in G_B[S_1, S_2]$ to be long enough so that the second condition holds. Then by Theorem \ref{thm_erdos_gallai}, the graph $G[S_1, T_2] \setminus V(P_1)$ contains a path of length $10 \eta n$ (recall that we assume $e(G_B[S_1, T_2]) \ge 10\eta n^2$), and we may take $P_2$ to be a subpath satisfying the third condition.
We continue as usual, by finding a Hamilton path in $G_B[S_2', T_2']$ with ends $s_2, t_3$ and partitioning $G[S_1', T_1']$ into a blue path with ends $s_1, s_3$ and a red path contained in $B_1$.

We now suppose that $G_B[S_1, S_2]$ does not have a path of length $10 \eta n$, so $G_R[S_1, S_2]$ is almost complete.
It follows that we may pick $u ,v \in A_1$ such that the set $D = N_R(u, S_2) \cap N_R(v, S_2)$ has size at least $n / 8$.
Define $S_1' = S_1 \setminus \{u, v\}$ and $S' = S \setminus \{u, v\}$.
We shall now continue as before, partitioning the graph $G[S' \cup T]$ into a blue cycle and a red path with ends in $D$.

If $|S_1'| > |T_1|$, we continue as in condition (\ref{itm_2_claim_S1}) of Case 1.1, picking say two paths in $G_B[S_1', T_2]$ with one end in $S_1'$ and the other in $T_2$, whose removal from $G_B[S_1', T_1]$ leaves a balanced graph. Again, the difference is that when we consider the remainder of the graph $G[S_2 \cup T_2]$, we partition it into a blue path with suitable ends and a red path contained in $D$.
Note that in order for this path to complete the path $(uyv)$ into a cycle, it has to contain at least one vertex. Recall that $|S| \ge |T| + 3$ so this is indeed possible.

If $|S_1'| = |T|$, we proceed similarly. It is easy to check that $G_B[S_2, T_1]$ contains an edge, since either $|S_2| \ge n/4 - 1/2$ or $|S_1'| = |T_1| \ge n/4 - 1/2$.
Thus we may pick one edge from $G_B[S_1', T_2]$ and another from $G_B[S_2, T_1]$ and continue as before.

Finally, we need to consider the case $|S_1'| < |T_1|$.
Here we can follow the argument of Case 1.3.
Note that we need to ensure that a version of Claim \ref{claim_paths_S2_T1} holds for the graph $G[S', T]$.
Indeed, take $H = G_B[S'', T_1]$, where $S''$ is obtained from $S'$ by removing a vertex from $S_1'$.
We claim that $H$ has a matching saturating $T_1$.
Indeed, let $W \subseteq T_1$. We consider four ranges for the size of $W$.
\begin{itemize}
\item
$|W| = 1$. Clearly, $|N_H(W)| \ge |W|$.
\item
$2 \le |W| \le n/4 - 2$.
Since any vertex of $W$ other than $x$ has at most one red neighbour in $G[S'', T]$, we have
$$|N_H(W)| \ge |S''| - n/4 = |S| - n/4 - 3 \ge n/4 - 2 \ge |W|.$$
\item
$n/4 - 2 < |W| \le n/4 - 1$. There exists a vertex in $W$ which has no red neighbours in $G[S'', T]$, thus $|N_H(W)| \ge n/4 - 1 \ge |W|$.
\item
$|W| > n/4 - 1$. Many vertices of $S''$ have at least one neighbour in $W$, implying that $|N_H(W)| > |T_1| \ge |W|$.
\end{itemize}
Thus we may pick a matching $M_1$ from in $G_B[S_2, T_1]$ of size $|T_1| - |S_1'| + 1$. Repeating almost the same argument, we deduce that the graph $H'$ obtained from $H$ by removing all edges spanned by $V(M)$, has a matching saturating $T_1$. From here we may continue as in Claim \ref{claim_paths_S2_T1} to obtain paths $Q_1, Q_2$.
Finally, as explained before, when considering the graph remaining from $G[S_2, T_2]$, we partition it into a blue path with suitable ends and a red path, on at least one vertex, contained in $D$.

From now on we may assume that $e(G_B[S_1, T_2]) \le 10\eta n^2$, so $S_1, S_2$ become practically interchangeable.

\subsubsection*{Case 2.2: \normalfont \normalsize $G_B[S_1, S_2]$ has a path of length $20 \eta n$}

Recall that $A_i = N_R(y, S_i)$ and $B_i = N_R(y, T_i)$.

If $|S_1| > |T_1|$ and $|A_2|, |B_2| \ge n / 16$, we continue as in Condition (\ref{itm_3_claim_S1}) in Case 1.1.
The difference is that when considering the graph remaining from $G[S_2 \cup T_2]$ we partition it into a blue path with the given ends and a red path (which may be empty) contained in either $A_2$ or $B_2$.
If $|S_1| = |T_1|$ and $|A_2|, |B_2| \ge n / 16$, it is easy to check that $G_B[S_2, T_1]$ is non empty, so we may proceed as before.
The case $|S_2| \ge |T_2|$ and $|A_1|, |B_1| \ge n /16$ follows analogously.

Since at most one of the sets $A_1, A_2, B_1, B_2$ has size at most $n / 16$, it remains to consider the case where $|S_1| < |T_1|$ or $|S_2| < |T_2|$.
Without loss of generality, $|S_1| < |T_1|$, so $|S_2| > |T_2|$. If $|A_1|, |B_1| \ge n /16$, we are done. Thus we may assume $|A_2| \ge n / 16$.
But then we may proceed as in Case 1.3, to partition $V(G)$ into a blue cycle and a red path contained in $A_2$.

From now on, we may assume that $G_B[S_1, S_2]$ has no path of length $20 \eta n$, implying that $G_R[S_1, S_2]$ is almost complete.
Also, without loss of generality, $$|A_2| \ge n / 16.$$

\subsubsection*{Case 2.3: \normalfont \normalsize $|S_1| \ge |T_1| + 1$ and $|S_2| \ge |T_2| + 3$}
It is easy to check that in this case $G[S_1, T_2] \cup G[S_2, T_1]$ contains a blue matching of size $2$.
Denote the matching by $\{e_1, e_2\}$ and without loss of generality suppose $e_i$ has ends $s_i \in S_2, t_i \in T_1$.
If $G[S_1]$ has a blue path of length $20 \eta n$, the graph $G_B[S_1 \cup T_1]$ has a Hamilton path with ends $t_1, t_2$.
We proceed as before, to partition $G[S_2 \cup T_2]$ into a blue path with ends $s_1, s_2$ and a red path contained in $A_2$.

We may thus assume that $G_R[S_1]$ is almost complete.
If $e(G_B[S_2]) \ge 10\eta n^2$, pick $u, v \in A_2$ such that the set $D = N_R(u, S_1) \cap N_R(v, S_1)$ has size at least $n/8$.
Denote $S_2' = S_2 \setminus \{u, v\}$.
It is easy to verify that $G_B[S_2', T_2]$ has a Hamilton path with ends $s_1, s_2$.
We form a red cycle by picking a red path $P_2$ in $D$ of length $|S_1| - |T_1|$ and joining it to $(uyv)$ (note that this is indeed a cycle since $|S_1| - |T_1| \ge 1$). The graph $G_B[S_1 \setminus V(P_2), T_1]$ has a Hamilton path with ends $t_1, t_2$, completing the required partition.

Finally, if $e(G_B[S_2]) \le 10\eta n^2$, we have that $G_R[S]$ is almost complete, hence we may pick a red cycle $C$ in $S \cup \{y\}$ with the following properties.
\begin{itemize}
\item
$y \in V(C)$.
\item
$|V(C) \cap S_1| = |S_1| - |T_1| + 1$.
\item
$|V(C) \cap S_2| = |S_2| - |T_2| - 1$.
\end{itemize}
It is easy to verify that the graph $G_B \setminus V(C)$ is Hamiltonian.

\subsubsection*{Case 2.4: \normalfont \normalsize $|S_1| = |T_1|$, $|S_2| = |T_2| + 2$ or $|S_2| = |T_2| + 1$}
These cases may be dealt with similarly to the previous one.
Denote $H_1 = G_B[S_1, T_2]$, $H_2 = G_B[S_2, T_1]$ and $H = H_1 \cup H_2$.

If $|S_1| = |T_1|$, if we may pick a $2$-matching in $H$ with at least one edge from $H_2$, we may continue as in the previous case.
Namely, we need to show that $H_2$ is non-empty.
This follows if $|S_2| > n/4 - 1$ or $|T_1| > n/4 - 1$, so we may assume that $|S_2|, |T_1| \le n/4 - 1$, implying that $|S| = |S_1| + |S_2| = |T_1| + |S_2| \le n/2 - 2$, a contradiction.

If $|S_2| = |T_2| + 2$ (so $|S_1| \ge |T_1| + 1$) we need to show that $H_1$ is non-empty.
If it is empty, we have $|S_1|, |T_2| \le n/4 - 1$, implying that $|S| \le n/2$, a contradiction.

If $|S_2| = |T_2| + 1$, we need to show that $H_1$ contains a $2$-matching.
If not, we have $|S_1|, |T_2| \le n/4$, and in addition either $|S_1| \le n/4 - 1$ or $|T_2| \le n/4 - 1$. In particular, $|S_1| + |T_2| \le n/2 - 1$, hence $|S| \le n/2$, a contradiction.

\subsubsection*{Case 2.5: \normalfont \normalsize $|S_1| < |T_1|$}
We proceed as in Case 1.3, partitioning the graph remaining from $G[S_2 \cup T_2]$ into a blue path with suitable ends and a red path contained in $A_2$.

\subsubsection*{Case 2.6: \normalfont \normalsize $|S_2| \le |T_2|$}
We may continue as in the last part of Case 2.1.
We pick $u, v \in A_2$ such that $|N_R(u, v) \cap S_1| \ge n/8$ and proceed to partition $G[S' \cup T]$ (where $S' = S \setminus \{u, v\}$) into a blue cycle and a red path contained in the red neighbourhood of $u, v$ in $S_1$, using an analogue of Claim \ref{claim_paths_S2_T1}.

\end{proof}

The proof of Lemma \ref{lem_red_disconnected_2} concludes the proof of our main Theorem, Theorem \ref{thm_main}. We finish this paper with some concluding remarks.

\section{Concluding Remarks} \label{sec_conclusion}

\subsection*{More than two colours - results and future work}
As a further line of research, one may consider colourings of $K_n$ with more than two colours.
Gy\'arf\'as \cite{gyarfas2} conjectured that for every $r$-colouring of $K_n$, the vertex set may be partitioned into at most $r$ monochromatic paths.
Erd\H{o}s, Gy\'arf\'as and Pyber \cite{erdos_gyarfas_pyber} considered partitions into monochromatic cycles rather than paths. 
The defined $c(r)$ to be the smallest $t$ such that whenever a complete graph is $r$-coloured, it may be partitioned into $c(r)$ monochromatic cycles.
They proved that $c(r)$ is bounded, and furthermore $c(r) \le cr^2 \log r$ for some constant $c$.
Note that Bessy and Thomass\'e's result \cite{bessy_thomasse}, mentioned in the introduction, implies that $c(2) = 2$.

Gy\'arf\'as, Ruszink\'o, S\'ark\"ozy and Szem\'eredi \cite{ruszinko_et_al_improved_c_r} proved that $c(r) \le c r \log r$, which is the best known upper bound on $c(r)$ so far.
The same authors \cite{ruszinko_et_al} proved an approximate result of the last conjecture for $r = 3$. Furthermore, they showed that for large enough $n$, if $K_n$ is $3$-coloured, it may be partitioned into $17$ monochromatic cycles.
However, it turns out that the full conjecture is false, even for $r = 3$, as shown by Pokrovskiy \cite{pokrovskiy}.
Nevertheless, in the same paper, he proved Gy\'arf\'as' conjecture for $r = 3$, so it may still be the case that Gy\'arf\'as' conjecture holds in general. In addition, the counter examples given in \cite{pokrovskiy} are $3$-colourings of $K_n$ for which all but one vertex may be covered by vertex-disjoint monochromatic paths.
This raises the following question: is it true that for every $r$-colouring of $K_n$ all but at most $c=c(r)$ vertices may be covered by $r$ vertex-disjoint monochromatic paths?

Finally, it is natural to consider the Schelp-type version of these problems, namely for graphs with large minimum degree rather than for complete graph. An example for a concrete question of this type is: what is the smallest value of $c$ such that any $3$-coloured graph $G$ on $n$ vertices and minimum degree $\delta(G) \ge cn$ can be partitioned into three monochromatic paths?

We believe that the methods we have employed for the proof of Theorem \ref{thm_main} may prove useful in resolving the latter questions, as well as many others, regarding partitions of $r$-coloured graphs into paths and cycles.

\bibliographystyle{amsplain}
\bibliography{cycle}

\end{document}